\documentclass[11pt]{amsart}
\usepackage{graphicx} % Required for inserting images
\usepackage{amsmath, amssymb, amsthm}
\usepackage{adjustbox}

\newif\ifclean
\cleanfalse
\usepackage{amsmath}
\usepackage{amsfonts}
\usepackage{amssymb}
\usepackage{amscd}
\usepackage{color}
\usepackage{xcolor}
\usepackage{hyperref}
\usepackage{mathrsfs}
\usepackage{eucal}
\usepackage{upgreek}
\usepackage[makeroom]{cancel}
\usepackage[normalem]{ulem}
\usepackage{array}
\usepackage{verbatim}
\usepackage{mathtools}
\usepackage{stmaryrd}

\usepackage[inline]{enumitem}

\usepackage{xy}
\xyoption{all}
\usepackage{tikz-cd}

\newcommand{\nc}{\newcommand}

\nc{\md}{\operatorname{-}}

\renewcommand{\AA}{{\mathbb{A}}}
\nc{\CC}{{\mathbb{C}}}
\nc{\DD}{{\mathbb{D}}}
\nc{\LL}{{\mathbb{L}}}
\nc{\FF}{{\mathbb{F}}}
\nc{\RR}{{\mathbb{R}}}
\renewcommand{\P}{{\mathbb{P}}}
\nc{\OO}{{\mathbb{O}}}

\nc{\QQ}{{\mathbb{Q}}}
\nc{\ZZ}{{\mathbb{Z}}}
\nc{\Z}{{\mathbb{Z}}}

\nc{\cA}{{\mathcal{A}}}
\nc{\cB}{{\mathcal{B}}}
\nc{\cBstd}{{\cB^{\rm std}}}
\nc{\cC}{{\mathcal{C}}}
\nc{\cD}{{\mathcal{D}}}
\nc{\cE}{{\mathcal{E}}}
\nc{\cF}{{\mathcal{F}}}
\nc{\cG}{{\mathcal{G}}}
\nc{\cH}{{\mathcal{H}}}
\nc{\cI}{{\mathcal{I}}}
\nc{\cJ}{{\mathcal{J}}}
\nc{\cK}{{\mathcal{K}}}
\nc{\cKstd}{{\cK^{\rm std}}}
\nc{\cL}{{\mathcal{L}}}
\nc{\cM}{{\mathcal{M}}}
\nc{\cN}{{\mathcal{N}}}
\nc{\cO}{{\mathcal{O}}}
\nc{\cP}{{\mathcal{P}}}
\nc{\cQ}{{\mathcal{Q}}}
\nc{\cR}{{\mathcal{R}}}
\nc{\cS}{{\mathcal{S}}}
\nc{\cT}{{\mathcal{T}}}
\nc{\cU}{{\mathcal{U}}}
\nc{\cV}{{\mathcal{V}}}
\nc{\cW}{{\mathcal{W}}}
\nc{\cX}{{\mathcal{X}}}
\nc{\cY}{{\mathcal{Y}}}
\nc{\cZ}{{\mathcal{Z}}}

\nc{\rc}{{\mathrm{c}}}
\nc{\rd}{{\mathrm{d}}}
\nc{\rf}{{\mathrm{f}}}
\nc{\rh}{{\mathrm{h}}}
\nc{\rrm}{{\mathrm{m}}}
\nc{\rs}{{\mathrm{s}}}
\nc{\rch}{{\mathrm{ch}}}
\nc{\rtd}{{\mathrm{td}}}

\nc{\rA}{{\mathrm{A}}}
\nc{\rB}{{\mathrm{B}}}
\nc{\rC}{{\mathrm{C}}}
\nc{\rD}{{\mathrm{D}}}
\nc{\rE}{{\mathrm{E}}}
\nc{\rF}{{\mathrm{F}}}
\nc{\rG}{{\mathrm{G}}}
\nc{\rH}{{\mathrm{H}}}
\nc{\rI}{{\mathrm{I}}}
\nc{\rJ}{{\mathrm{J}}}
\nc{\rK}{{\mathrm{K}}}
\nc{\rL}{{\mathrm{L}}}
\nc{\rM}{{\mathrm{M}}}
\nc{\rN}{{\mathrm{N}}}
\nc{\rO}{{\mathrm{O}}}
\nc{\rP}{{\mathrm{P}}}
\nc{\rQ}{{\mathrm{Q}}}
\nc{\rR}{{\mathrm{R}}}
\nc{\rS}{{\mathrm{S}}}
\nc{\rT}{{\mathrm{T}}}
\nc{\rU}{{\mathrm{U}}}
\nc{\rV}{{\mathrm{V}}}
\nc{\rW}{{\mathrm{W}}}
\nc{\rX}{{\mathrm{X}}}
\nc{\rY}{{\mathrm{Y}}}
\nc{\rZ}{{\mathrm{Z}}}

\nc{\bA}{{\mathbf{A}}}
\nc{\bB}{{\mathbf{B}}}
\nc{\bC}{{\mathbf{C}}}
\nc{\bD}{{\mathbf{D}}}
\nc{\bE}{{\mathbf{E}}}
\nc{\bF}{{\mathbf{F}}}
\nc{\bG}{{\mathbf{G}}}
\nc{\bH}{{\mathbf{H}}}
\nc{\bI}{{\mathbf{I}}}
\nc{\bJ}{{\mathbf{J}}}
\nc{\bK}{{\mathbf{K}}}
\nc{\bL}{{\mathbf{L}}}
\nc{\bM}{{\mathbf{M}}}
\nc{\bN}{{\mathbf{N}}}
\nc{\bO}{{\mathbf{O}}}
\nc{\bP}{{\mathbf{P}}}
\nc{\bQ}{{\mathbf{Q}}}
\nc{\bR}{{\mathbf{R}}}
\nc{\bS}{{\mathbf{S}}}
\nc{\bT}{{\mathbf{T}}}
\nc{\bU}{{\mathbf{U}}}
\nc{\bV}{{\mathbf{V}}}
\nc{\bW}{{\mathbf{W}}}
\nc{\bX}{{\mathbf{X}}}
\nc{\bY}{{\mathbf{Y}}}
\nc{\bZ}{{\mathbf{Z}}}

\nc{\ba}{{\mathbf{a}}}
\nc{\bb}{{\mathbf{b}}}
\nc{\bc}{{\mathbf{c}}}
\nc{\bd}{{\mathbf{d}}}
\nc{\be}{{\mathbf{e}}}
\nc{\bg}{{\mathbf{g}}}
\nc{\bh}{{\mathbf{h}}}
\nc{\bi}{{\mathbf{i}}}
\nc{\bj}{{\mathbf{j}}}
\nc{\bk}{{\mathbf{k}}}
\nc{\bl}{{\mathbf{l}}}
\nc{\bm}{{\mathbf{m}}}
\nc{\bn}{{\mathbf{n}}}
\nc{\bo}{{\mathbf{o}}}
\nc{\bp}{{\mathbf{p}}}
\nc{\bq}{{\mathbf{q}}}
\nc{\br}{{\mathbf{r}}}
\nc{\bs}{{\mathbf{s}}}
\nc{\bt}{{\mathbf{t}}}
\nc{\bu}{{\mathbf{u}}}
\nc{\bv}{{\mathbf{v}}}
\nc{\bw}{{\mathbf{w}}}
\nc{\bx}{{\mathbf{x}}}
\nc{\by}{{\mathbf{y}}}
\nc{\bz}{{\mathbf{z}}}

\nc{\fA}{{\mathfrak{A}}}
\nc{\fB}{{\mathfrak{B}}}
\nc{\fC}{{\mathfrak{C}}}
\nc{\fD}{{\mathfrak{D}}}
\nc{\fE}{{\mathfrak{E}}}
\nc{\fF}{{\mathfrak{F}}}
\nc{\fG}{{\mathfrak{G}}}
\nc{\fH}{{\mathfrak{H}}}
\nc{\fI}{{\mathfrak{I}}}
\nc{\fJ}{{\mathfrak{J}}}
\nc{\fK}{{\mathfrak{K}}}
\nc{\fL}{{\mathfrak{L}}}
\nc{\fM}{{\mathfrak{M}}}
\nc{\fN}{{\mathfrak{N}}}
\nc{\fO}{{\mathfrak{O}}}
\nc{\fP}{{\mathfrak{P}}}
\nc{\fQ}{{\mathfrak{Q}}}
\nc{\fR}{{\mathfrak{R}}}
\nc{\fS}{{\mathfrak{S}}}
\nc{\fT}{{\mathfrak{T}}}
\nc{\fU}{{\mathfrak{U}}}
\nc{\fV}{{\mathfrak{V}}}
\nc{\fW}{{\mathfrak{W}}}
\nc{\fX}{{\mathfrak{X}}}
\nc{\fY}{{\mathfrak{Y}}}
\nc{\fZ}{{\mathfrak{Z}}}

\nc{\fa}{{\mathfrak{a}}}
\nc{\fb}{{\mathfrak{b}}}
\nc{\fc}{{\mathfrak{c}}}
\nc{\fd}{{\mathfrak{d}}}
\nc{\fe}{{\mathfrak{e}}}
\nc{\ff}{{\mathfrak{f}}}
\nc{\fg}{{\mathfrak{g}}}
\nc{\fh}{{\mathfrak{h}}}
\nc{\fj}{{\mathfrak{j}}}
\nc{\fk}{{\mathfrak{k}}}
\nc{\fl}{{\mathfrak{l}}}
\nc{\fm}{{\mathfrak{m}}}
\nc{\fn}{{\mathfrak{n}}}
\nc{\fo}{{\mathfrak{o}}}
\nc{\fp}{{\mathfrak{p}}}
\nc{\fq}{{\mathfrak{q}}}
\nc{\fr}{{\mathfrak{r}}}
\nc{\fs}{{\mathfrak{s}}}
\nc{\ft}{{\mathfrak{t}}}
\nc{\fu}{{\mathfrak{u}}}
\nc{\fv}{{\mathfrak{v}}}
\nc{\fw}{{\mathfrak{w}}}
\nc{\fx}{{\mathfrak{x}}}
\nc{\fy}{{\mathfrak{y}}}
\nc{\fz}{{\mathfrak{z}}}

\nc{\sA}{{\mathsf{A}}}
\nc{\sB}{{\mathsf{B}}}
\nc{\sC}{{\mathsf{C}}}
\nc{\sD}{{\mathsf{D}}}
\nc{\sE}{{\mathsf{E}}}
\nc{\sF}{{\mathsf{F}}}
\nc{\sG}{{\mathsf{G}}}
\nc{\sH}{{\mathsf{H}}}
\nc{\sI}{{\mathsf{I}}}
\nc{\sJ}{{\mathsf{J}}}
\nc{\sK}{{\mathsf{K}}}
\nc{\sL}{{\mathsf{L}}}
\nc{\sM}{{\mathsf{M}}}
\nc{\sN}{{\mathsf{N}}}
\nc{\sO}{{\mathsf{O}}}
\nc{\sP}{{\mathsf{P}}}
\nc{\sQ}{{\mathsf{Q}}}
\nc{\sR}{{\mathsf{R}}}
\nc{\sS}{{\mathsf{S}}}
\nc{\sT}{{\mathsf{T}}}
\nc{\sU}{{\mathsf{U}}}
\nc{\sV}{{\mathsf{V}}}
\nc{\sW}{{\mathsf{W}}}
\nc{\sX}{{\mathsf{X}}}
\nc{\sY}{{\mathsf{Y}}}
\nc{\sZ}{{\mathsf{Z}}}

\nc{\sa}{{\mathsf{a}}}
\nc{\sd}{{\mathsf{d}}}
\nc{\se}{{\mathsf{e}}}
\nc{\sg}{{\mathsf{g}}}
\nc{\sh}{{\mathsf{h}}}
\nc{\si}{{\mathsf{i}}}
\nc{\sj}{{\mathsf{j}}}
\nc{\sk}{{\mathsf{k}}}
\nc{\sm}{{\mathsf{m}}}
\nc{\sn}{{\mathsf{n}}}
\nc{\so}{{\mathsf{o}}}
\nc{\sq}{{\mathsf{q}}}
\nc{\sr}{{\mathsf{r}}}
\nc{\st}{{\mathsf{t}}}
\nc{\su}{{\mathsf{u}}}
\nc{\sv}{{\mathsf{v}}}
\nc{\sw}{{\mathsf{w}}}
\nc{\sx}{{\mathsf{x}}}
\nc{\sy}{{\mathsf{y}}}
\nc{\sz}{{\mathsf{z}}}

\nc{\oA}{{\overline{A}}}
\nc{\oB}{{\overline{B}}}
\nc{\oC}{{\overline{C}}}
\nc{\oD}{{\overline{D}}}
\nc{\oE}{{\overline{E}}}
\nc{\oF}{{\overline{F}}}
\nc{\oG}{{\overline{G}}}
\nc{\oH}{{\overline{H}}}
\nc{\oI}{{\overline{I}}}
\nc{\oJ}{{\overline{J}}}
\nc{\oK}{{\overline{K}}}
\nc{\oL}{{\overline{L}}}
\nc{\oM}{{\overline{M}}}
\nc{\oN}{{\overline{N}}}
\nc{\oO}{{\overline{O}}}
\nc{\oP}{{\overline{P}}}
\nc{\oQ}{{\overline{Q}}}
\nc{\oR}{{\overline{R}}}
\nc{\oS}{{\overline{S}}}
\nc{\oT}{{\overline{T}}}
\nc{\oU}{{\overline{U}}}
\nc{\oV}{{\overline{V}}}
\nc{\oW}{{\overline{W}}}
\nc{\oX}{{\overline{X}}}
\nc{\oY}{{\overline{Y}}}
\nc{\oZ}{{\overline{Z}}}

\nc{\oa}{{\overline{a}}}
\nc{\ob}{{\overline{b}}}
\nc{\oc}{{\overline{c}}}
\nc{\od}{{\overline{d}}}
\nc{\of}{{\overline{f}}}
\nc{\og}{{\overline{g}}}
\nc{\oh}{{\overline{h}}}
\nc{\oi}{{\overline{i}}}
\nc{\oj}{{\overline{j}}}
\nc{\ok}{{\overline{k}}}
\nc{\ol}{{\overline{l}}}
\nc{\om}{{\overline{m}}}
\nc{\on}{{\overline{n}}}
\nc{\oo}{{\overline{o}}}
\nc{\op}{{\overline{p}}}
\nc{\oq}{{\overline{q}}}
\nc{\os}{{\overline{s}}}
\nc{\ot}{{\overline{t}}}
\nc{\ou}{{\overline{u}}}
\nc{\ov}{{\overline{v}}}
\nc{\ow}{{\overline{w}}}
\nc{\ox}{{\overline{x}}}
\nc{\oy}{{\overline{y}}}
\nc{\oz}{{\overline{z}}}

\nc{\tA}{{\tilde{A}}}
\nc{\tB}{{\tilde{B}}}
\nc{\tC}{{\tilde{C}}}
\nc{\tD}{{\tilde{D}}}
\nc{\tE}{{\tilde{E}}}
\nc{\tF}{{\tilde{F}}}
\nc{\tG}{{\tilde{G}}}
\nc{\tH}{{\tilde{H}}}
\nc{\tI}{{\tilde{I}}}
\nc{\tJ}{{\tilde{J}}}
\nc{\tK}{{\tilde{K}}}
\nc{\tL}{{\tilde{L}}}
\nc{\tM}{{\tilde{M}}}
\nc{\tN}{{\tilde{N}}}
\nc{\tO}{{\tilde{O}}}
\nc{\tP}{{\tilde{P}}}
\nc{\tQ}{{\tilde{Q}}}
\nc{\tR}{{\tilde{R}}}
\nc{\tS}{{\tilde{S}}}
\nc{\tT}{{\tilde{T}}}
\nc{\tU}{{\tilde{U}}}
\nc{\tV}{{\tilde{V}}}
\nc{\tW}{{\tilde{W}}}
\nc{\tX}{{\tilde{X}}}
\nc{\tY}{{\tilde{Y}}}
\nc{\tZ}{{\tilde{Z}}}

\nc{\tfD}{{\tilde{\fD}}}
\nc{\tcA}{{\tilde{\cA}}}
\nc{\tcB}{{\tilde{\cB}}}
\nc{\tcC}{{\tilde{\cC}}}
\nc{\tcD}{{\tilde{\cD}}}
\nc{\tcE}{{\tilde{\cE}}}
\nc{\tcF}{{\tilde{\cF}}}
\nc{\tcM}{{\tilde{\cM}}}
\nc{\tcP}{{\tilde{\cP}}}
\nc{\tcT}{{\tilde{\cT}}}

\nc{\ta}{{\tilde{a}}}
\nc{\tb}{{\tilde{b}}}
\nc{\tc}{{\tilde{c}}}
\nc{\td}{{\tilde{d}}}
\nc{\te}{{\tilde{e}}}
\nc{\tf}{{\tilde{f}}}
\nc{\tg}{{\tilde{g}}}
\nc{\ti}{{\tilde{\imath}}}
\nc{\tj}{{\tilde{j}}}
\nc{\tk}{{\tilde{k}}}
\nc{\tl}{{\tilde{l}}}
\nc{\tm}{{\tilde{m}}}
\nc{\tn}{{\tilde{n}}}
\nc{\tp}{{\tilde{p}}}
\nc{\tq}{{\tilde{q}}}
\nc{\tr}{{\tilde{r}}}
\nc{\ts}{{\tilde{s}}}
\nc{\tu}{{\tilde{u}}}
\nc{\tv}{{\tilde{v}}}
\nc{\tw}{{\tilde{w}}}
\nc{\tx}{{\tilde{x}}}
\nc{\ty}{{\tilde{y}}}
\nc{\tz}{{\tilde{z}}}

\nc{\hA}{{\hat{A}}}
\nc{\hB}{{\hat{B}}}
\nc{\hC}{{\hat{C}}}
\nc{\hD}{{\hat{D}}}
\nc{\hE}{{\hat{E}}}
\nc{\hF}{{\hat{F}}}
\nc{\hG}{{\hat{G}}}
\nc{\hH}{{\hat{H}}}
\nc{\hI}{{\hat{I}}}
\nc{\hJ}{{\hat{J}}}
\nc{\hK}{{\hat{K}}}
\nc{\hL}{{\hat{L}}}
\nc{\hM}{{\hat{M}}}
\nc{\hN}{{\hat{N}}}
\nc{\hO}{{\hat{O}}}
\nc{\hP}{{\hat{P}}}
\nc{\hQ}{{\hat{Q}}}
\nc{\hR}{{\hat{R}}}
\nc{\hS}{{\hat{S}}}
\nc{\hT}{{\hat{T}}}
\nc{\hU}{{\hat{U}}}
\nc{\hV}{{\hat{V}}}
\nc{\hW}{{\hat{W}}}
\nc{\hX}{{\widehat{X}}}
\nc{\hY}{{\hat{Y}}}
\nc{\hZ}{{\hat{Z}}}

\nc{\ha}{{\hat{a}}}
\nc{\hb}{{\hat{b}}}
\nc{\hc}{{\hat{c}}}
\nc{\hd}{{\hat{d}}}
\nc{\he}{{\hat{e}}}
\nc{\hg}{{\hat{g}}}
\nc{\hh}{{\hat{h}}}
\nc{\hi}{{\hat{i}}}
\nc{\hj}{{\hat{j}}}
\nc{\hk}{{\hat{k}}}
\nc{\hl}{{\hat{l}}}
\nc{\hm}{{\hat{m}}}
\nc{\hn}{{\hat{n}}}
\nc{\ho}{{\hat{o}}}
\nc{\hp}{{\hat{p}}}
\nc{\hq}{{\hat{q}}}
\nc{\hr}{{\hat{r}}}
\nc{\hs}{{\hat{s}}}
\nc{\hu}{{\hat{u}}}
\nc{\hv}{{\hat{v}}}
\nc{\hw}{{\hat{w}}}
\nc{\hx}{{\hat{x}}}
\nc{\hy}{{\hat{y}}}
\nc{\hz}{{\hat{z}}}

\nc{\hcC}{{\widehat{\cC}}}
\nc{\hcT}{{\widehat{\cT}}}

\nc{\eps}{\upepsilon}
\nc{\lan}{\big\langle}
\nc{\ran}{\big\rangle}
\nc{\kk}{{\Bbbk}}
\nc{\io}{\upiota}
\nc{\Kr}{\mathsf{Kr}}
\nc{\cKr}{\mathcal{K}\!\mathit{r}}

\nc{\Dm}{\bD^{-}}
\nc{\Db}{\bD^{\mathrm{b}}}
\nc{\Dbc}{\bD^{\mathrm{b}}_{\mathrm{c}}}
\nc{\Dp}{\bD^{\mathrm{perf}}}
\nc{\Dperf}{\bD^{\mathrm{perf}}}
\nc{\Dqc}{\bD_{\mathrm{qc}}}
\nc{\Du}{\bD}
\nc{\Dsing}{\bD^{\mathrm{sg}}}
\nc{\Dg}{\bD^{\mathrm{sg}}}

\def\ol{\overline}

\newcommand{\opp}{\mathrm{op}}

\nc{\Rn}{\rR_{\mathrm{node}}}
\nc{\Cn}{\cC_{\mathrm{node}}}
\nc{\Dfd}[1]{\bD_{\mathrm{fd}}(#1)}
\def\bw#1#2{\textstyle{\bigwedge\hskip-0.9mm^{#1}}\hskip0.2mm{#2}}

\nc{\xrightiso}[1]{ \xrightarrow[{\ \raisebox{0.5ex}[0ex][0ex]{$\sim$}\ }]{#1} }

\nc{\thick}{\mathbf{thick}}

\DeclareMathOperator{\ev}{\mathbf{ev}}

\DeclareMathOperator{\Hom}{\mathrm{Hom}}

\DeclareMathOperator{\Ext}{\mathrm{Ext}}

\DeclareMathOperator{\End}{\mathrm{End}}
\DeclareMathOperator{\Aut}{\mathrm{Aut}}

\DeclareMathOperator{\Spec}{\mathrm{Spec}}

\DeclareMathOperator{\Coh}{\mathrm{Coh}}

\DeclareMathOperator{\Bl}{\mathrm{Bl}}

\DeclareMathOperator{\Pic}{\mathrm{Pic}}

\DeclareMathOperator{\Br}{\mathrm{Br}}
\DeclareMathOperator{\Am}{\mathrm{Am}}
\DeclareMathOperator{\SB}{\mathrm{SB}}
\DeclareMathOperator{\DAm}{\mathrm{DAm}}

\DeclareMathOperator{\CH}{\mathrm{CH}}

\DeclareMathOperator{\ExDiv}{\mathrm{ExDiv}}

\DeclareMathOperator{\Coker}{\mathrm{Coker}}
\DeclareMathOperator{\Ima}{\mathrm{Im}}

\DeclareMathOperator{\ind}{\mathrm{ind}}

\DeclareMathOperator{\GL}{\mathrm{GL}}

\DeclareMathOperator{\PGL}{\mathrm{PGL}}

\DeclareMathOperator{\id}{\mathrm{id}}

\DeclareMathOperator{\rank}{\mathrm{rk}}
\DeclareMathOperator{\codim}{\mathrm{codim}}

\DeclareMathOperator{\modd}{\mathrm{mod{--}}}
\DeclareMathOperator{\mmod}{\mathrm{{--}mod}}
\DeclareMathOperator{\Gal}{\mathrm{Gal}}

\DeclareMathOperator{\Rep}{\mathrm{Rep}}
\DeclareMathOperator{\cVar}{\mathcal{V}ar}
\DeclareMathOperator{\MMP}{\mathrm{MMP}}

\def\wt{\widetilde}

\def\xra{\xrightarrow}

\theoremstyle{plain}

\newtheorem{theorem}{Theorem}[section]
\newtheorem{conjecture}[theorem]{Conjecture}

\newtheorem{lemma}[theorem]{Lemma}
\newtheorem{proposition}[theorem]{Proposition}
\newtheorem{corollary}[theorem]{Corollary}

\newtheorem*{theorem*}{Theorem}

\theoremstyle{definition}

\newtheorem{definition}[theorem]{Definition}

\newtheorem{example}[theorem]{Example}

\theoremstyle{remark}

\newtheorem{remark}[theorem]{Remark}

\numberwithin{equation}{section} % equation numbering including section number
\newcommand{\dto}
{\dashrightarrow}
\newcommand{\dercat}[1]{\Db(#1)}

\ifclean
\newcommand{\ftn}[1]{}

\else
\newcommand{\ftn}[1]{\footnote{#1}}

\fi

\ifclean
\title{Atomic decompositions for derived categories of $G$-surfaces}
\else
\title{Atomic decompositions for derived categories of $G$-surfaces}
\fi

\date{December 2025}

\author[A. Elagin]{Alexey Elagin}
\address{School of Mathematical and Physical Sciences, University of Sheffield, S3 7RH, Sheffield, UK}
\curraddr{}
\email{alexey.elagin@gmail.com}

\author[J. Schneider]{Julia Schneider}
\address{Université Bourgogne Europe, CNRS, IMB UMR 5584, 21000 Dijon, France}
\curraddr{}
\email{julia.schneider@ube.fr}

\author[E. Shinder]{Evgeny Shinder}
\address{School of Mathematical and Physical Sciences, University of Sheffield, S3 7RH, Sheffield, UK}
\curraddr{}
\email{eugene.shinder@gmail.com}

\begin{document}

\begin{abstract}
We construct canonical semi-orthogonal decompositions for derived categories of smooth projective surfaces. These decompositions are compatible with the operations in the minimal model program, such as blow-ups and conic bundles. Therefore our construction confirms a conjecture of Kontsevich in dimension two.
We work in the $G$-equivariant setting and over an arbitrary perfect field, and canonical decompositions are consistent with group change  and algebraic field extensions.
Our method is based on the $G$-minimal model program for surfaces and on the Sarkisov link factorisation of birational maps between Mori fibre spaces. 
We characterise rationality of surfaces, and in certain cases,  birationality between surfaces in terms of the pieces of these decompositions, which we call atoms.
\end{abstract}

\maketitle

\tableofcontents

\section{Introduction}

One of the most important questions 
about derived categories of coherent sheaves in algebraic geometry is: given a birational map
$\phi\colon X \dashrightarrow Y$ between smooth projective varieties,  how
 $\dercat{X}$ and $\dercat{Y}$ are related?
This question has been considered in the pioneering work on derived categories by Bondal and Orlov for flips and flops 
\cite[p. 40]{BO-sod}, \cite{BO-ICM}. A
conjecture of Kawamata \cite{Kawamata-KD} predicts that if $\phi$ is a K-equivalence, then $\Db(X)$ and $\Db(Y)$ are equivalent. This is known to be the case for $3$-folds by a result of Bridgeland
\cite{Bridgeland-flops}.

A key observation is that if
 $f\colon Y \to X$ is a \emph{derived contraction}, that is, a morphism with $\rR f_* \mathcal{O}_Y = \mathcal{O}_X$ (e.g. a smooth blow-up), then the derived pullback functor $f^*\colon \dercat{X} \to \dercat{Y}$ 
 is fully faithful and
 we have a semi-orthogonal decomposition
\[
\dercat{Y} = \langle \mathcal{K}, f^* \dercat{X} \rangle, \text{ where } \mathcal{K} = \ker(f_*) := \{F \in \Db(Y) \colon f_*(F) = 0\}.
\]
The kernel category $\mathcal{K}$ decomposes further in special cases such as smooth blow-ups and projective bundles \cite{Orlov-blowup}, and some Fano fibrations \cite{Kuznetsov-quadrics}, \cite{Kuznetsov-dP6}, \cite{Xie-dP5}. The most prominent case is Orlov's blow-up formula, see \cite{Orlov-blowup} or Proposition~\ref{blowupformula}: if $\tX$ is the blow-up of $X$ at a smooth centre $Z$ of codimension $r$ then 
$$\Db(\tX)=\langle \underbrace{\Db(Z),\ldots,\Db(Z)}_{\text{$(r-1)$ times}},\Db(X)\rangle.$$
Thus derived categories of varieties which are birational, or more generally related by operations of the Minimal Model Program (MMP), are expected to have  common pieces in their semi-orthogonal decompositions.
On the other hand, it is known that semi-orthogonal decompositions into minimal pieces are in general not unique, already in dimension two.

\begin{conjecture}[Kontsevich]\label{conj:Kontsevich}
There should exist canonical semi-orthogonal decompositions, well-defined up to mutations, which are
compatible with geometric operations.   
\end{conjecture}

This conjecture has been made more precise in the Noncommutative Minimal Model Program by
Halpern-Leistner \cite{DHL-NMMP} relating semi-orthogonal decompositions to paths in the stability manifold.

Existence of such canonical decompositions 
would imply that the pieces of the decomposition for $\Db(X)$ can be an obstruction to rationality of $X$, because
only pieces coming from varieties of dimension at most $\dim(X) - 2$ will appear if $X$ is rational.
A particular case is the celebrated  cubic fourfold conjecture by Kuznetsov \cite{Kuznetsov-cubics} that says that  
a smooth complex cubic fourfold $X \subset \P^5$ is rational if and only if the main component $\bA_X \subset \Db(X)$ is equivalent to the derived category of a K3 surface.
In the recent work \cite{KKPY-atoms}, Katzarkov, Kontsevich, Pantev, and Yu  proved irrationality of a very general complex cubic fourfold $X$ using their groundbreaking theory of {Hodge atoms} which is based on decomposing quantum cohomology into canonical pieces. As one of the steps in the proof, they show that the Hodge atom analog of the component $\bA_X$ of a cubic fourfold $X$ is indeed not coming from surfaces.

In this paper we construct canonical decompositions for varieties up to dimension $2$, over a perfect field and in the equivariant setting.
Using these decompositions we obtain some new results on rationality and, more generally, birationality of surfaces. 

\subsection{Atomic decompositions for {$G$}-surfaces}

We will now formalize what it means for semi-orthogonal decompositions to be compatible with geometric operations.
Motivated by the theory of Hodge atoms \cite{KKPY-atoms, CKK-atoms},
we make the following definition. For a fixed group $G$, let $\cV$ be a category consisting of some class of $G$-varieties (Definition~\ref{def_Gvariety}) as objects and some class of $G$-equivariant derived contractions as morphisms.

\begin{definition}[see Definition \ref{def_GSAT}] A $G$-atomic theory for $\cV$
assigns for
all $X \in \mathrm{Obj}(\cV)$
a mutation-equivalence class of $G$-invariant
semi-orthogonal decompositions, called \emph{atomic decompositions}, with the following property.
For any morphism $f\colon Y \to X$  in $\cV$, there is
an atomic decomposition 
$\dercat{X} = \langle \bA_1, \dots, \bA_n \rangle$ and an atomic decomposition for $\dercat{Y}$ that is obtained from 
\[
\dercat{Y} = \langle 
\ker(f_*),
f^*\bA_1, \dots, f^*\bA_n \rangle
\]
by possibly splitting  $\ker (f_*)$ into several components.
\end{definition}

From now on until the end of this subsection, we work over an algebraically closed field $\kk$ in the equivariant setting: that is, we fix a group $G$ and consider $G$-varieties and $G$-equivariant morphisms. In particular, we allow the \emph{geometric case} of a finite group $G$ acting $\kk$-linearly, and the \emph{arithmetic case} where $G$ is a Galois group. 

\medskip

Our first main result is the construction of 
a $G$-atomic theory
for the category 
$G-\cVar_{\kk, \le 2}^{\MMP}$
consisting of $G$-varieties of $\dim\le 2$  and rational derived contractions. By a rational derived contraction in $\dim\le 2$ we mean a derived contraction with geometrically rational generic fibre, or equivalently, a composition of $G$-MMP (Minimal Model Program) steps, that is, $G$-blow-ups of surfaces and $G$-Mori fibre space structure maps.

\begin{theorem}
[see Theorem \ref{thm:atoms-surfaces}]
\label{thm:main-intro}
For every group $G$ there is an explicit
 atomic theory for 
$G-\cVar_{\kk, \le 2}^{\MMP}$.
Furthermore, the constructed theories are compatible with passing to subgroups $H \subset G$.
\end{theorem}

By compatibility for a subgroup $H \subset G$ we mean that $G$-atomic decompositions can be refined to $H$-atomic decompositions.
Theorem \ref{thm:main-intro} proves Conjecture \ref{conj:Kontsevich} in dimension 2.
In the case when $G$ is a Galois group
%$G = \Gal(\ol{\kk}/\kk)$ 
it
strengthens earlier results by Auel--Bernardara \cite{AuelBernardara} and Bernardara--Durighetto
\cite{BernardaraDurighetto}
who introduced the Kuznetsov component for geometrically rational surfaces in the arithmetic case and proved that it is a birational invariant.

\begin{example}\label{ex:P2-intro}
In the simplest case $\kk = \overline{\kk}$ and $G = \{e\}$, for  rational surfaces our result says the following. There are two types of minimal rational surfaces: $\P^2$ and Hirzebruch surfaces $\mathbb{F}_n$, $n \ge 0$, $n \ne 1$, with atomic decompositions
\[
\dercat{\P^2} = \langle \langle\cO(-2)\rangle, \langle\cO(-1)\rangle, \langle\cO \rangle\rangle, \quad
\dercat{\mathbb{F}_n} = \langle
\langle\cO(-s-h)\rangle, \langle\cO(-s)\rangle, \langle\cO(-h)\rangle, \langle\cO \rangle\rangle.
\]
Here $h$ is the fibre class and $s$ is the negative section class. Now if $X$ is any rational surface with contractions $X \to X_1, X\to X_2$ to minimal rational  surfaces then the two exceptional collections on $X$ induced from $X_1$ and $X_2$ by the blow-up formula both define atomic decompositions, in particular, they are mutation equivalent.
\end{example}

The proof of Theorem \ref{thm:main-intro} relies on the Minimal Model Program  for $G$-surfaces, using Sarkisov links to check compatibility as in the arithmetic case for components of derived categories in \cite{AuelBernardara}, \cite{BernardaraDurighetto} 
and the Grothendieck ring of varieties in \cite{LinShinderZimmermann}.
There are several steps in the proof.
The almost trivial step is to construct the (unique) atomic theory for curves.
Then we construct a decomposition of $\Db(X)$ for every two-dimensional $G$-Mori fibre space $X \to B$, based on the Karpov--Nogin $3$-block exceptional collections for del Pezzo surfaces \cite{KarpovNogin}. We check that these decompositions are \emph{compatible with Sarkisov links between Mori fibre spaces}: two atomic decompositions for two Mori fibre spaces forming the sides of a Sarkisov link induce mutation-equivalent decompositions on the roof. This requires case-by-case checking all possible types of links and constitutes the longest part of the proof.

This allows for a unique construction of an atomic theory for $G$-surfaces matching the decompositions for $G$-Mori fibre spaces: for arbitrary $G$-surface $Y$ we choose a sequence of blow-ups of $G$-orbits producing~$Y$ from a Mori fibre space $X/B$ or a surface $X$ with $K_X$ nef, and augment atomic decomposition for $X$ be adding an atom $\ker f_{i*}$ for any blow-up $f_i$ in the sequence as in the blow-up formula.

One important aspect of our atomic theory is that we do not decompose K-nef surfaces such as Enriques surfaces, i.e. by construction they are assigned a semi-orthogonal decomposition with a single piece. This is motivated by both the MMP, where such surfaces are the end steps, as well as by the quantum cohomology viewpoint, where such surfaces are assigned a single Hodge atom \cite{KKPY-atoms}. 

We refer to pieces of atomic semi-orthogonal decompositions
of $\dercat{X}$ as $G$-\emph{atoms}; these pieces, as  triangulated categories with a $G$-action,  are independent of the choice of an atomic decomposition because  mutations simply permute the pieces of a decomposition.

Of particular importance are \emph{$G$-atoms of permutation type}: these are $G$-atoms generated by an orthogonal exceptional collection $E_1, \ldots, E_n$ which is transitively permuted by $G$.

\begin{proposition}[see Proposition \ref{prop_Z}]
\label{prop:perm-atom-intro}
A $G$-atom of permutation type, up to $G$-equivariant equivalence, is determined by a transitive representation $\rho\colon G \to \rS_n$ and a cohomology class $[\alpha] \in \rH^2(G, (\kk^*)^n)$
where the $G$-action on $(\kk^*)^n$ is defined by the permutation $\rho$ and the action on scalars is fixed.
\end{proposition}

If $f\colon Y \to X$
is a blow up of a $G$-orbit $Z$, then the $G$-atoms of $\Db(Y)$
are those of $\Db(X)$ together with a new $G$-atom of permutation type corresponding to $Z$ with a trivial twisting class $[\alpha]$.
Therefore the set of   $G$-atoms of permutation of type with nontrivial twisting class is a $G$-birational invariant of $X$. This birational invariant is reminiscent to the Amitsur invariant \cite{Liedtke-SB, BCDP-finite}
and gives similar obstructions to $G$-birationality.

\begin{remark}
Conjecturally, the atomic decompositions we construct
correspond to Hodge atoms constructed by Katzarkov--Kontsevich--Pantev--Yu \cite{KKPY-atoms} and further used as an obstruction to rationality in \cite{CKK-atoms}.
We thank Jenya Tevelev for explaining to us the following observation in this direction.
The number of atoms for $G$-minimal del Pezzo surfaces -- two atoms for del Pezzo surfaces of degree $4$ and below, and three atoms for del Pezzo surfaces of degree $5$ and above -- matches the number of distinct eigenvalues of the quantum multiplication by the canonical class at the $G$-invariant point \cite[\S3]{BayerManin}, which by definition corresponds to the Hodge atoms in \cite{KKPY-atoms}.
\end{remark}

\begin{remark}
The exact connection between the atomic decompositions we construct and the Noncommutative Minimal Model Program by Halpern-Leistner \cite{DHL-NMMP}, further developed in \cite{HLJR} and
partially realized for surfaces \cite{HLJR, Karube} deserves a separate investigation. It is quite likely that the two theories coincide.
\end{remark}

Our first application of Theorem \ref{thm:main-intro}
is to the motivic invariant \cite{LinShinder, LinShinderZimmermann} of birational maps between surfaces, which was extended to the equivariant case by \cite{KreschTschinkel-invariant}.
In terms of atomic decompositions, the invariant of a birational map between $G$-surfaces $X$ and $Y$ can be understood as the formal difference 
between the multisets of equivalence classes of $G$-atoms in $X$ and $Y$.
Thus the motivic invariant for equivariant self-maps of a $G$-surface is always trivial, see Theorem \ref{thm:motivic}.
This extends the result for surfaces in the arithmetic case \cite{LinShinderZimmermann},
and contrasts with the case of $G$-threefolds where such invariants are in general nonzero \cite[\S3.3]{LinShinder},
\cite[\S7]{KreschTschinkel-invariant}, \cite[Corollary 4.6]{LinShinder-fibrations}

\subsection{Applications in the arithmetic case}

We now explain applications of our atomic theory to birational geometry of $G$-surfaces.
We assume that we are in the arithmetic case, that is, $\kk$ is a perfect field and
$G = \Gal(\ol{\kk}/\kk)$ acts on the scalar extension of varieties from $\kk$ to $\ol{\kk}$. 
Using Galois descent for semi-orthogonal decompositions we deduce from Theorem 
\ref{thm:main-intro}:

\begin{theorem}
[see Theorem \ref{thm:atoms-surfaces2}]
\label{thm:main-intro-perfect}
For any perfect field $\kk$
there is an explicit
 atomic theory for 
$\cVar_{\kk, \le 2}^{\MMP}$.
Furthermore the constructed theories are compatible with passing to algebraic field extensions $\kk \subset \LL$.
\end{theorem}

Theorem \ref{thm:atoms-surfaces2} is a more general version, allowing a $\kk$-linear action of a group %$H$
on our varieties,
and the theories will be compatible with taking subgroups just like in Theorem \ref{thm:main-intro}.
To prove Theorem we rely on a general results from the theory of descent of semi-orthogonal decompositions \cite{Elagin_2011}, \cite{Elagin_2012}, 
and in particular 
the Galois descent as developed in \cite{AuelBernardara}, \cite{BDM}. 
We use a bijection between $G$-invariant thick subcategories in $\Db(X_{\overline{\kk}})$ and thick subcategories in $\Db(X)$ which preserves semi-orthogonal 
decompositions (Lemma \ref{lemma_subcatdescent}, Proposition \ref{prop_SODdescent}).
This allows to descend
a $G$-atomic theory over $\ol{\kk}$ to an atomic theory over $\kk$.

\medskip

We aim for a birational classification of geometrically rational surfaces
in terms of their atoms.
Our first result in this direction is:

\begin{proposition}[see Theorem \ref{thm:rational-atoms}] \label{prop:rational-intro}
A surface $X$ 
is rational if and only if
all
atoms of  $\dercat{X}$  are of the form 
$\dercat{L}$
for finite field extensions $L/\kk$.
\end{proposition}

Proposition \ref{prop:rational-intro} refines the main results of \cite{AuelBernardara, BernardaraDurighetto} who proved that rational surfaces have semi-orthogonal decompositions into pieces of the form $\Db(L)$, and constructed a nonvanishing Kuznetsov component for nonrational geometrically rational surfaces.

In our approach the 'only if' implication of Proposition \ref{prop:rational-intro} holds by construction of the atomic theory because a blow up of a closed point with residue field $L$ adds an atom of the form $\Db(L)$. For the 'if' implication we have to prove that $X$ has a toric model and index $\ind(X) = 1$ (this uses the formula in Proposition \ref{prop:ind-intro} below) which then implies that $X$ is rational.
In the geometric case the 
'only if'
implication still holds, however the 'if' implication fails, see Example \ref{ex:dP5-rigid}. We conjecture that in the geometric case, 
atoms of a $G$-surface 
control stable linearisability, 
see Conjecture \ref{conj:triv-stab-rat}.

For the next application we need the following definition. We say that $X$ is \emph{birationally rich} if it is a geometrically rational surface which is
birational either to a del Pezzo surface of degree at least $5$, or to a conic bundle of degree at least $5$. Informally speaking, these are surfaces with the most number of birational maps, and 'birationally rich' serves as a moral antonym to 'birationally rigid', even though the two are not mutually exclusive. Every rational surface is by definition birationally rich, and for algebraically closed fields the two notions coincide.

\begin{proposition}[Proposition \ref{prop:bir-rich-non-closed}] 
A surface $X$ is birationally rich if and only if all its atoms are of the form $\Db(L, \ba)$ for some finite field extensions $L/\kk$ and Brauer classes $\ba \in \Br(L)$. 
\end{proposition}

Here $\Db(L, \ba)$ is the  derived category of $\ba$-twisted $L$-vector spaces; these atoms are obtained by Galois descent of $G$-atoms of permutation type, and the class $\ba$ corresponds to the twisting class $[\alpha]$ from Proposition \ref{prop:perm-atom-intro}.
We  call
atoms of the form
$\Db(L, \ba)$  \emph{small atoms},
and atoms of the form
$\Db(L)$ (that is small atoms with a trivial Brauer class)  \emph{trivial atoms}. In these terms Proposition \ref{prop:rational-intro} says that rationality of a surface is equivalent to triviality of its atoms. 

\begin{example}[\cite{Bernardara}, \cite{AuelBernardara}]
Let $A$ be a central simple algebra of degree $3$ (that is dimension $3^2 = 9$) over $\kk$. Let $\ba \in \Br(\kk)$ be its $3$-torsion Brauer class.
 The Severi--Brauer surface $X = \SB(A)$ which is the corresponding twisted form of $\P^2$ is birationally rich and has an atomic semi-orthogonal decomposition
\[
\Db(X) = \langle \Db(\kk, 2\ba), \Db(\kk, \ba), \Db(\kk) \rangle
\]
obtained by Galois descent from the atomic decomposition for $\Db(\P^2)$ in Example \ref{ex:P2-intro}.
All atoms of $\Db(X)$ are small, but not all of them are trivial unless 
$A$ is a split matrix algebra $\rM_3(\kk)$.

Let us consider the opposite Severi--Brauer surface $X^\mathrm{op}$
which corresponds to the opposite algebra $A^\mathrm{op}$
and has Brauer class $-\ba = 2\ba \in \Br(\kk)$.
Then the standard Cremona involution is easily seen to descend to a birational isomorphism between $X$ and $X^\mathrm{op}$, and both categories $\Db(X)$ and $\Db(X^\mathrm{op})$ consist of the same atoms.
\end{example}

\begin{theorem}[see Theorem \ref{thm:almost-same-atoms-bir-rich}]
\label{thm:bir-rich-bir-intro}
Let $X$ and $X'$ be birationally rich surfaces. Then $X$ and $X'$ are birational if and only if they have the same nontrivial atoms.
\end{theorem}

This result unifies some known statements about birationality of del Pezzo surfaces, considered in \cite{CTKM}, 
\cite{AuelBernardara},
\cite{Liedtke-SB}, \cite{ShramovVologodsky}, \cite{Trepalin}, \cite{KY-dP6}. 
As opposed to the proof of Proposition \ref{prop:rational-intro} which is uniform, our proof of Theorem \ref{thm:bir-rich-bir-intro} 
relies on some case-by-case analysis, and we need to use the biregular classification of rank one toric surfaces from \cite{CTKM}, \cite{Blunk}, \cite{Trepalin}.

The result analogous to Theorem \ref{thm:bir-rich-bir-intro} cannot hold without any assumptions on the surfaces, because K-nef surfaces can be derived equivalent without being birational.
We conjecture that the result still holds for all geometrically rational surfaces. 

We explain one more connection between atomic decompositions and arithmetic of surfaces.
Recall that the Amitsur group $\mathrm{Am}(X)$
 \cite{CTKM}, \cite{Liedtke-SB} and the index $\ind(X)$ are basic birational invariants of $X$. They are related to atomic decompositions as follows: 

\begin{proposition}[see Proposition \ref{prop:DAm-index-formula}]
\label{prop:ind-intro}
Let $X$ be a birationally rich surface with atomic semi-orthogonal decomposition
\[
\Db(X) = \langle \Db(L_1, \ba_1), \ldots, \Db(L_m, \ba_m) \rangle.
\]
Then $|\mathrm{Am}(X)|\cdot \ind(X) = \prod_{i=1}^m \ind(\ba_i)$, where $\ind(\ba_i)$ is the index of the corresponding Brauer class.
\end{proposition}

For del Pezzo surfaces this result can be deduced from the classification given in  \cite{AuelBernardara}. We provide a uniform proof using what we call the \emph{derived Amitsur group} which is an invariant of derived categories of surfaces, additive for semi-orthogonal decompositions.

We have another application for a class of surfaces which are not birationally rich.

\begin{theorem}[see Theorem \ref{thm:dP4-bir} for a stronger  statement]
\label{thm:Deg4}
If $X$ and $X'$ are minimal del Pezzo surfaces of degree $4$ 
which are birational, then they are isomorphic.
\end{theorem}

This result has been proven independently by Shramov and Trepalin \cite{ShramovTrepalin}
by an intricate direct analysis of conic bundles that appear in the link factorisation of a birational map.
Our approach is instead
first to show that $X$ and $X'$ have the same atoms, and then use a Torelli theorem for degree $4$ del Pezzo surfaces by the first author \cite{AE_Torelli}
to deduce that they are isomorphic.
We have an analog of Theorem~\ref{thm:Deg4} for any group action.

\medskip

Finally, let us explain the main difficulties for applying our strategy for constructing atomic semi-orthogonal decompositions in higher dimension. Firstly, varieties appearing in higher-dimensional MMP will be singular, and while it is conjectured that derived categories of singular varieties admit unique minimal resolutions \cite[Conjecture 5.1]{BO-ICM}, \cite[Conjecture 4.10]{Kuznetsov-crepant}, this is currently far out of reach. Furthermore, while it is a natural idea to relate derived categories of Mori fibre spaces using Sarkisov links, this seems to be a lot more difficult in higher dimension due to the lack of a complete classification of Sarkisov links that is available for surfaces due to work of Iskovskikh~\cite{Isk96}.

\subsection*{Acknowledgements} We thank 
Arend Bayer,
Tom Bridgeland,
Alexander Duncan,
Daniel Halpern-Leistner,
Ludmil Katzar\-kov,
Alexander Kuznetsov,
Ed Segal,
Jenya Tevelev
for discussions and interest to our work.
We thank St\'ephane Lamy for 
including a remark explaining a version of the Sarkisov link decomposition, which is essential for our proof of Theorem \ref{thm_compatible_gives_atomictheory}, 
into the
Cremona Group book \cite{LamyCremona}.
The authors were supported by the UKRI Horizon Europe guarantee award `Motivic invariants and birational geometry of simple normal crossing degenerations' EP/Z000955/1.

\section{Preliminaries}

Let $\kk$ be a field. By a \emph{variety over $\kk$} we mean a smooth projective geometrically integral scheme over $\kk$. We denote the category of varieties over $\kk$ by $\cVar_{\kk}$, and $\cVar_{\kk, n}$ (resp. $\cVar_{\kk, \le n}$) denotes the full subcategory of varieties of dimension $n$ (resp. $\le n$).
By a \emph{surface} we mean a variety  of dimension $2$. 

\subsection{Group actions and $G$-surfaces}

Recall that a \emph{group action} of a group $G$ on a scheme $X$ is a group homomorphism $\gamma\colon G\to \Aut(X)$, where $\Aut(X)$ is the group of scheme automorphisms of $X$.

\begin{remark}
    If $X$ is a variety over a field $\kk$ (in particular, projective and geometrically integral by assumption), then
    $\rH^0(X,\cO_X)\cong \kk$ and so any group action $\gamma \colon G\to\Aut(X)$ induces a group action $\check{\gamma}\colon G\to \Aut(\kk)$, that is, for every $g\in G$ there is a commutative diagram
\[\begin{tikzcd}
    X\ar[r, "\sim", "\gamma_g"']\ar[d,"\pi"]& X \ar[d,"\pi"] \\
    \Spec\kk\ar[r,,"\sim","\check\gamma_g"']&\Spec\kk,
\end{tikzcd}\]
where $\pi\colon X\to\Spec\kk$ denotes the structure morphism.
\end{remark}

For applications, we will mostly have the following types of group actions on varieties in mind:
\begin{enumerate}
    \item By a \emph{geometric action} we mean an injective group homomorphism $\gamma\colon G\to \Aut_\kk(X)$, where $G$ is a finite group, $\kk$ is an algebraically closed field, and $\Aut_\kk(X)$ denotes the automorphisms of $X$ over $\kk$.
    \item By an \emph{arithmetic action} we mean the following: let $X$ be a variety over a perfect field $\kk$. Then the absolute Galois group $G=\Gal(\bar \kk/\kk)$ acts on $X_{\bar \kk}=X\otimes_\kk \bar \kk$ via trivial action on $X$ and with the usual action of the Galois group on $\bar \kk$. This gives an action $\gamma\colon \Gal(\bar \kk/\kk)\to \Aut_\kk(X_{\bar\kk})\subset \Aut(X_{\bar\kk})$.
    \item A combination of the two above actions: let $X$ be a variety over a perfect field $\kk$, let $\Gal(\bar\kk/\kk)$ act tautologically on $X_{\bar\kk}$, and let $G$ be a finite group injecting to $\Aut_\kk(X)\subset\Aut_\kk(X_{\bar\kk})$. Since these two actions on $X_{\bar\kk}$ commute, we get an action $\gamma\colon G\times \Gal(\bar \kk/\kk)\to \Aut_\kk(X_{\bar\kk})$.
\end{enumerate}

\begin{definition}
Let $X$ and $Y$ be two varieties  over $\kk$ with a $G$-action, and let $f\colon X\to Y$ be a morphism over $\kk$. We say $f$ is a \emph{$G$-morphism} or $f$ is \emph{$G$-equivariant} if for every $g\in G$ the following diagram commutes: 
\[\begin{tikzcd}    X\ar[r,"\gamma_{g}^X"]\ar[d,"f",swap] & X\ar[d,"f"]\\ 
    Y\ar[r,"\gamma_{g}^Y"]& Y.
\end{tikzcd}\]
\end{definition}

A $G$-action on a variety $X$ induces a $G$-action on $\Pic(X)$, the set of isomorphism classes of line bundles on $X$, and on $N_1(X)$, the $\mathbb{R}$-vector space of $1$-cycles on $X$ modulo numerical equivalence. We denote by $\Pic(X)^G\subset\Pic(X)$ the subgroup of (isomorphism classes of) $G$-invariant line bundles, and by $N_1(X)^G\subset N_1(X)$ the $G$-invariant subspace. The \emph{$G$-Picard rank} of $X$, denoted by $\rho^G(X)$, is the dimension of the $\mathbb{R}$-vector space $N_1(X)^G$.

For the notion of a $G$-surface we will make stronger assumptions on the group action so that the $G$-equivariant Sarkisov Program works (see e.g. \cite[Section III, 16.1]{LamyCremona}):
\begin{definition}
\label{def_Gvariety}
Let $X$ be a  variety over $\kk$, and let $G$ be a group acting on $X$. We say that $X$, together with the group action, is a \emph{$G$-variety} if the induced action on $N_1(X)$ has finite orbits. In particular, a \emph{$G$-surface} is a $G$-variety of dimension $2$. 
\end{definition}
We remark that we will use the Definition~\ref{def_Gvariety} only for varieties of dimension $\le 2$, and the finiteness condition is empty for varieties of dimension $\le 1$. Also we remark that the three types of actions mentioned above satisfy Definition~\ref{def_Gvariety}.

\subsection{Birational geometry of $G$-surfaces}
In this subsection, we assume $\kk$ to be algebraically closed.

Let $Y$ be a $G$-surface, $Z\subset Y$ be a finite $G$-invariant set and 
$f\colon X\to Y$ be the blow-up of $Z$. Then $G$ acts naturally on $X$ and the map $f$ is $G$-equivariant. The induced $G$-action on $N_1(X)$ has finite orbits,  hence $X$ is a $G$-surface as in Definition~\ref{def_Gvariety}.

The other way, let $X$ be a $G$-surface and $f\colon X\to Y$ a birational morphism onto a surface $Y$ contracting disjoint $(-1)$-curves. If the set of contracted $(-1)$-curves is $G$-invariant then $f$ induces a $G$-surface structure on $Y$ such that $f\colon X\to Y$ is a $G$-morphism. Indeed, the $G$-action on $Y$ is given by $g\in G\mapsto f\circ \gamma_g\circ f^{-1}\in\Aut(Y)$, which is indeed in $\Aut(Y)$ since the contracted curves form $G$-orbits.
Moreover, the induced $G$-action on $N_1(Y)$ has finite orbits, hence $Y$ is a $G$-surface.

\bigskip

A $G$-surface $X$ is \emph{$G$-minimal} if every birational $G$-morphism $X\to Y$ to any smooth $G$-surface $Y$ is an isomorphism. 

\begin{definition}
    Let $X$ be a $G$-surface and $B$ a smooth $G$-variety of $\dim(B)\le  1$. Let $\pi\colon X\to B$ be a contraction (that is, a surjective morphism with $\pi_*\cO_X=\cO_B$) that is $G$-equivariant.
    \begin{enumerate}
        \item For $r\ge 1$, the morphism $\pi$ is a \emph{$G$-rank $r$ fibration} if $-K_X$ is $\pi$-ample and $\rho^G(X/B)=r$        . 
        \item A \emph{$G$-Mori fibre space} $X/B$ is a $G$-rank $1$ fibration.
        \item If $B$ is a curve and $-K_X$ is $\pi$-ample, then $\pi$ is a \emph{$G$-conic bundle}.
        \item A $G$-conic bundle that is a $G$-Mori fibre space is also called a \emph{$G$-Mori conic bundle.}
    \end{enumerate}
\end{definition}

\begin{remark}
    If $X/\Spec\kk$ is a $G$-rank $r$ fibration, then $X$ is a del Pezzo surface with $\rho^G(X)=r$.
\end{remark}

Recall that we assume a $G$-surface $X$ to have finite orbits for the induced action on $N_1(X)$. 
\begin{theorem}[$G$-equivariant MMP for surfaces, see~\cite{DolgachevIskovskikh}]
    Let $Y$ be a $G$-surface. Then there is a $G$-surface $X$ and a birational $G$-morphism $Y\to X$ such that either $K_X$ is nef, or $X$ admits the structure of a $G$-Mori fibre space $X/B$. For a $G$-Mori fibre space $X/B$, if $B=\Spec \kk$ is a point then $X$ is a $G$-minimal del Pezzo surface, and if $B$ is a curve then $X/B$ is a $G$-Mori conic bundle.
\end{theorem}

\begin{definition}
\label{def_isoMFS}
    Let $X_1/B_1$ and $X_2/B_2$ be two $G$-Mori fibre spaces. An isomorphism $\varphi\colon X_1\to X_2$ is an \emph{isomorphism of $G$-Mori fibre spaces from $X_1/B_1$ to $X_2/B_2$} if there is a $G$-equivariant diagram \[\begin{tikzcd}    X_1\ar[r,"\varphi"]\ar[d,"\pi_1",swap] & X_2\ar[d,"\pi_2"]\\ 
    B_1\ar[r,"\cong"] &B_2.
\end{tikzcd}\]
\end{definition}

\begin{definition}[Sarkisov links]\label{def:SarkisovLinks}
Let $X_1/B_1$ and $X_2/B_2$ be two $G$-Mori fibre spaces. A $G$-equivariant birational map $\chi\colon X_1\dashrightarrow X_2$ is a \emph{$G$-Sarkisov link from $X_1/B_1$ to $X_2/B_2$} if there is a $G$-equivariant commutative diagram 
    \[\begin{tikzcd}        &Z\ar[dl,swap,"\sigma_1"]\ar[dr,"\sigma_2"]&\\        X_1\ar[d,swap,"\pi_1"]\ar[dashed,rr,"\chi"]&& X_2\ar[d,"\pi_2"]\\ 
        B_1\ar[dr,"\beta_1",swap]&& B_2\ar[dl,"\beta_2"]\\
        &B,&
    \end{tikzcd}\] where $Z/B$ is a $G$-rank $2$ fibration, and $\chi$ is not an isomorphism of $G$-Mori fibre spaces. 
\end{definition}

\begin{remark}\label{rem:dPZ-link}
 For any $G$-rank $2$ fibration $Z/B$ there exists a Sarkisov link with two sides determined by two extremal contractions of $Z$ \cite[Lemma 16.18]{LamyCremona}. This applies in particular to a $G$-rank $2$ del Pezzo surface over $B = \Spec(\kk)$.  
\end{remark}

The following is the classification of Sarkisov links in terms of $K_{X_1}^2,K_Z^2,K_{X_2}^2$, obtained in \cite[Theorem~2.6]{Isk96} in the arithmetic case, and the same computations work for more general group actions \cite[\S7]{DolgachevIskovskikh}. Since $\rho^G(Z/B)=2$ and $\rho^G(X_i/B_i)=1$, exactly one of the maps $\sigma_1,\beta_1$ is an isomorphism, and similarly for $\sigma_2,\beta_2$. Therefore links divide into four types depending on which of the four combination realises. 
\begin{proposition}[Numerical classification of Sarkisov links]\label{prop:SarkisovLinks}
    Let $X_1/B_1$ and $X_2/B_2$ be two $G$-Mori fibre spaces and $\chi\colon X_1\dashrightarrow X_2$ be a $G$-Sarkisov link over a base $B$, with associated $G$-rank $2$ fibration $Z/B$ as in Definition~\ref{def:SarkisovLinks}.
    Then one of the following holds.

    If $\chi$ is a \textbf{link of type I}, that is, $\beta_1$ and $\sigma_2$ are isomorphisms, then $B_1=B=\Spec\kk$, $B_2=\P^1$, and
    \[(K_{X_1}^2,K_{X_2}^2)\in\{(9,8),(9,5),(8,6),(4,3)\}.\]

    If $\chi$ is a \textbf{link of type II}, that is, $\beta_1$ and $\beta_2$ are isomorphisms, then either 
    \begin{enumerate}
        \item $B$ is a curve and $K_{X_1}^2=K_{X_2}^2$, or
        \item $B=\Spec\kk$ and $(K_{X_1}^2,K_Z^2,K_{X_2}^2)$ is either symmetric (that is, $K_{X_1}^2=K_{X_2}^2$) and in one of the cases
        \begin{enumerate}
            \item $(d,1,d)$ for $d\in\{9,8,6,5,4,3,2\}$ (Bertini type)
            \item $(d,2,d)$ for $d\in\{9,8,6,5,4,3\}$ (Geiser type)
            \item $(9,6,9)$, $(9,3,9)$,  $(8,4,8)$, $(6,4,6)$, $(6,3,6)$,
        \end{enumerate}
        or asymmetric and in one of the cases (up to swapping $X_1$ and $X_2
        $)
        \begin{enumerate}[resume]
            \item $(9,7,8)$, $(9,4,5)$, $(8,5,6)$,  $(8,3,5)$.
        \end{enumerate}
    \end{enumerate} 
    
    If $\chi$ is a \textbf{link of type III}, that is, $\sigma_1$ and $\beta_2$ are isomorphisms, then $\chi^{-1}$ is a link of type I.

    If $\chi$ is a \textbf{link of type IV}, that is, $\sigma_1$ and $\sigma_2$ are isomorphisms, then $B=\Spec\kk$, $B_1=B_2=\P^1$ and $K_{X_1}^2=K_{X_2}^2\in\{1,2,4,8\}$.
\end{proposition}

For a surface $X$, we will call $K_X^2$ the \emph{degree} of $X$.

Note that if $X_1/B_1$ satisfies $K_{X_1}^2\le  4$, then either $K_{X_2}^2=K_{X_1}^2$, or $\chi$ (or $\chi^{-1})$ is a link of type I with $(K_{X_1}^2,K_{X_2}^2)=(4,3)$.
On the other hand, if $K_{X_1}^2\ge 5$ there are many Sarkisov links where $K_{X_1}^2$ and $K_{X_2}^2$ differ.
In this sense, we introduce the following notion:

\begin{definition}\label{def:BirationallyRich}
    We say that a $G$-surface is \emph{$G$-birationally rich} if it is $G$-birational to a
    $G$-Mori fibre space of degree $\ge 5$.
\end{definition}

\begin{proposition}
\label{prop_list}
    Let $X/B$ be a $G$-Mori fibre space.
    \begin{enumerate}
        \item\label{it:SarkisovLinks--BirRich} $X$ is $G$-birationally rich exactly if one of the following holds:
            \begin{enumerate}
            \item\label{it:SarkisovLinks--BirRich--dP} $B=\Spec\kk$ and $X$ is a $G$-minimal del Pezzo surface of degree $K_X^2\in\{9,8,6,5\}$, or
            \item\label{it:SarkisovLinks--BirRich--cb} $B=\P^1$ and $X/B$ is a $G$-Mori conic bundle of degree $K_X^2\in\{8,6,5\}$.
        \end{enumerate}
        \item\label{it:SarkisovLinks--NotBirRich} $X$ is not $G$-birationally rich exactly if one of the following holds:
        \begin{enumerate}
            \item\label{it:SarkisovLinks--NotBirRich--dP} $B=\Spec\kk$ and $X$ is a $G$-minimal del Pezzo surface of degree $K_X^2\in\{4,3,2,1\}$, or
            \item\label{it:SarkisovLinks--NotBirRich--P1} $B=\P^1$ and $X/B$ is a $G$-Mori conic bundle of degree $K_X^2\in\{4,3,2,1\}$, or 
            \item\label{it:SarkisovLinks--NotBirRich--P1neg} $B=\P^1$ and $X/B$ is a $G$-Mori conic bundle of degree $K_X^2\le  0$, or
            \item\label{it:SarkisovLinks--NotBirRich--curve} $B$ is a curve of genus $g(B)\ge 1$ and $X/B$ is a $G$-Mori conic bundle of degree $K_X^2\le 8(1-g(B))\le  0$.
        \end{enumerate}    
    \end{enumerate} 
    Moreover, if $X'/B'$ is a $G$-Mori fibre space such that $X'$ is $G$-birational to $X$, then $X/B$ and $X'/B'$ are both in one of the following, mutually excluding cases: 
    $\{\eqref{it:SarkisovLinks--BirRich--dP}, \eqref{it:SarkisovLinks--BirRich--cb}\}$, $\{\eqref{it:SarkisovLinks--NotBirRich--dP}, \eqref{it:SarkisovLinks--NotBirRich--P1}\}$, $\{\eqref{it:SarkisovLinks--NotBirRich--P1neg}\}$, $\{\eqref{it:SarkisovLinks--NotBirRich--curve}\}$.
\end{proposition}
\begin{proof}
One can see that the list has all $G$-Mori fibre spaces: for the classification of these  see~\cite[\S 8]{DolgachevIskovskikh}, and for the formula  $K_X^2 = 8(1-g)$ for a conic bundle without singular fibres over a curve of genus $g$~\cite[Prop. III.21]{Beauville}.   By definition, $G$-Mori fibre spaces of types \eqref{it:SarkisovLinks--BirRich--dP},\eqref{it:SarkisovLinks--BirRich--cb} are birationally rich.
By the Sarkisov program, any $G$-birational isomorphism between $G$-Mori fibre spaces decomposes into a composition of Sarkisov links (see~\cite{Isk96}, \cite{DolgachevIskovskikh}, or \cite[Th. 16.28(1)]{LamyCremona}).  
Checking the list of Sarkisov links from Proposition~\ref{prop:SarkisovLinks}, one sees that every $G$-Sarkisov link starting at a Mori fibre space of some type goes to a Mori fibre space of the same type, except for links between types \eqref{it:SarkisovLinks--BirRich--dP},\eqref{it:SarkisovLinks--BirRich--cb} and between types \eqref{it:SarkisovLinks--NotBirRich--dP},\eqref{it:SarkisovLinks--NotBirRich--P1}. In particular Mori fibre spaces of type (2) are not $G$-birationally rich.
\end{proof}

\begin{remark}\label{rem:non-minimal-C}
Not all $G$-Mori conic bundles are $G$-minimal surfaces. More precisely, a $G$-Mori conic bundle  $X/B$ is not $G$-minimal exactly if there is a $G$-Sarkisov link of type III from $X/B$. By examination of the links in Proposition~\ref{prop:SarkisovLinks} we have:
    \begin{enumerate}
        \item If $K_X^2\notin\{8,6,5,3\}$, then $X$ is $G$-minimal.   
        \item If $K_X^2=8$, then $X=\bF_n$ is a Hirzebruch surface with $n\ge0$, and it is minimal (and hence $G$-minimal) exactly if $n\neq1$. 
        \item If $K_X^2\in\{6,5,3\}$, then $X$ is not $G$-minimal.  In this case a link of type III always exists by Remark \ref{rem:dPZ-link} since $X/\Spec\kk$ is a rank $2$ fibration, see~\cite[Prop. 5.2]{DolgachevIskovskikh}. 
    \end{enumerate}
\end{remark}

We will denote the class of $G$-minimal del Pezzo surfaces of degree $d$ by $\cD_d$ and the class of rational $G$-Mori conic bundles of degree $d$ by $\cC_d$.

\begin{corollary}
\label{cor:bir-rich-models}
A $G$-surface $X$ is $G$-birationally rich  if and only if 
there is a birational $G$-equivariant morphism from $X$ to a $G$-minimal surface of one of the following types:
\[
\cD_9, \quad \cD_8, \quad \cD_6, \quad \cD_5, \quad \cC_8.
\]
\end{corollary}
\begin{proof}
    This follows from Proposition \ref{prop_list} as the surfaces $\cC_6$ and $\cC_5$ are not minimal (see Remark \ref{rem:non-minimal-C}), hence can be contracted, to one of the surfaces from the list.
\end{proof}

\subsection{Derived contractions and MMP contractions} 

Here we introduce a class of morphisms that will play a central role in the following. 
The base field~$\kk$ can be any perfect field.

\begin{definition}
\label{def_dc}
We will call a morphism $f\colon X\to Y$   between varieties over $\kk$ a \emph{derived contraction} if  $\rR f_*\cO_X=\cO_Y$.
\end{definition}

\begin{definition}
\label{def_GMMP}
Let $X,Y$ be smooth projective $G$-varieties of dimension $\le 2$ over an algebraically closed field $\kk$.
We will say that a $G$-morphism $f\colon X\to Y$ is
a \emph{$G$-MMP contraction} if $f$ is one of the following:
    \begin{enumerate}[label=(\alph*)]
    \item a $G$-Mori fibre space of dimension $1$ or $2$,
    \item the blow-up of a finite $G$-orbit on a $G$-surface $Y$.
    \end{enumerate}
In other words, $f$ is a $G$-MMP contraction if $f$ is an extremal contraction of a $K_X$-negative $G$-extremal ray of the $G$-Mori cone of $X$, see~\cite{KollarMori1998}.
\end{definition} 

\begin{proposition}
\label{prop_elementary}
Let $f\colon X\to Y$ be a $G$-morphism between smooth projective $G$-varieties of dimension $\le 2$ over an algebraically closed field $\kk$. Then the following are equivalent:
\begin{enumerate}[label=(\alph*)]
    \item generic fibre $f$ is geometrically rational,
    \item $f$ is a composition of G-MMP contractions.
\end{enumerate}
\end{proposition}
If this holds, then $f$ is a derived contraction.
\begin{proof}
We sketch the proof.

(a) $\Longrightarrow$   (b): run relative G-MMP for $X$ over $Y$ and show that eventually one gets $Y$ over $Y$.

(b) $\Longrightarrow$   (a): straightforward check.

The final statement follows since the class of derived contractions is closed under compositions, and every $G$-MMP contraction is a derived contraction.
\end{proof}

\begin{definition}
\label{def_rdc} 
We say that a morphism $f\colon X \to Y$ of surfaces over a perfect field $\kk$ is a \emph{rational derived contraction} if $f_{\ol{\kk}} \colon X_{\ol{\kk}} \to Y_{\ol{\kk}}$ satisfies the conditions of Proposition \ref{prop_elementary}
with respect to the Galois group $G = \Gal(\ol{\kk}/\kk)$.
\end{definition}

\begin{remark}
Not every derived contraction is a rational derived contraction. For example, if $X$ is an Enriques surface  then the projection $X \to \Spec(\kk)$ is a derived contraction but not a rational derived contraction. 
\end{remark}

\subsection{Divisor classes on surfaces}
\label{section_divisorsonsurfaces}
Let $X$ be a smooth projective rational surface over an algebraically closed field $\kk$. It is known that $X$ is an iterated blow-up of $\P^2$ or of a Hirzebruch surface $\bF_d$, $d\ge 0$. The Picard group $\Pic X$ of $X$ is a finitely generated free abelian group and is equipped with a non-degenerate intersection form. 

If $X$ is an iterated blow-up of $\P^2$ at $n$ points, then $\Pic X$ has the basis $H,E_1,\ldots,E_n$, where $H$ is the preimage of the line on $\P^2$ and $E_i$ is the preimage of the exceptional divisor of the $i$-th blow-up. One has
$$H^2=1, E_i^2=-1, H\cdot E_i=E_i\cdot E_j=0\quad\text{for $i\ne j$.}$$
Also one has
$$K_X=-3H+E_1+\ldots+E_n.$$

If $X$ is an iterated blow-up of a Hirzebruch surface $\bF_d$ at $n$ points, then $\Pic X$ has the basis $h,s,E_1,\ldots,E_n$, where $h$ and $s$ are the preimages of the fibre and the negative section of $\bF_d$ respectively, and $E_i$ is the preimage of the exceptional divisor of the $i$-th blow-up. One has
$$h^2=0, h\cdot s=1, s^2=-d, E_i^2=-1, h\cdot E_i=s\cdot E_i=E_i\cdot E_j=0\quad\text{for $i\ne j$.}$$
Also one has
$$K_X=-2s-(2+d)h+E_1+\ldots+E_n.$$
We will use the following terminology
\begin{definition}
Let us say that $C\subset X$ is an \emph{$r$-curve} if $C$ is a smooth irreducible rational curve and $C^2=r$.
Let us say that $D\in\Pic X$ is an \emph{$r$-class} if $D^2+D\cdot K_X=-2$ and $D^2=r$. 
\end{definition}
If $C\subset X$ is a smooth irreducible rational curve then by the adjunction formula one has $C^2+C\cdot K_X=-2$, hence the divisor class of an $r$-curve is an $r$-class.

\medskip
Recall that a \emph{del Pezzo surface} is a surface $X$ whose anticanonical divisor $-K_X$ is ample, and that its \emph{degree} is the self-intersection $d=K_X^2$. One has that $1\le d\le 9$, and $X$ is either isomorphic to the blow-up of $\P^2$ at $r=9-d$ points in general position, or $X\cong\P^1\times\P^1$. 

\begin{remark}\label{rem:general-pos}
Recall that $r\leq 8$ points on $\P^2$ are in general position if no three lie on a line, no six on a conic, and (if $r=8$) not all eight lie on a  singular cubic such that one of them is the singular point. Equivalently, the blow-up of $\P^2$ at these points is a del Pezzo surface.

More generally, for a del Pezzo surface of degree $d$, we say that $r$ points are in general position if their blow-up gives again a del Pezzo surface. For $r\leq d-2$, this is the case exactly if for every $k\geq 1$, no $k$ of the points lie on a $(k-2)$-curve.
\end{remark}

\begin{lemma}[See~{\cite[Prop.  2.9, Cor. 2.15]{EXZ}, \cite[Lemma 4.7]{LinShinderZimmermann}}]
\label{lemma_classes}
Assume $X$ is a del Pezzo surface of degree $d$ and $D$ is an $r$-class on $X$. Then 
\begin{enumerate}
    \item if $r=-2$ then the pair $(\cO, \cO(D))$ is completely orthogonal: $h^i(\cO(D))=h^i(\cO(-D))=0$ for all $i$, 
    \item if $d\ge 3$ then any $0$-class is realised by a smooth irreducible curve and defines a morphism $X\to \P^1$ with general fibre $\P^1$,
    \item if $d\ge 4$ then any $1$-class is realised by a smooth irreducible curve and defines a birational morphism $X\to \P^2$.
\end{enumerate}
\end{lemma}

There are symmetries between certain classes of divisors as the following lemma explains.

\begin{lemma}
\label{lemma_dualclasses}
Let $X$ be a del Pezzo surface of degree~$d$. 
\begin{enumerate}
    \item Assume $D,D'\in \Pic X$ and $D+D'=-K_X$. If $D$ is an $r$-class then $D'$ is an $r'$-class, where $r+r'=d-4$.
    \item Assume $D,D'\in \Pic X$ and $D+D'=-2K_X$. If $D$ is a $(d-2)$-class then $D'$ is also a $(d-2)$-class.
\end{enumerate}
\end{lemma}
\begin{proof}
Follows from the definitions, see~\cite[Lemma  2.5]{EXZ}, \cite[Lemma 4.7]{LinShinderZimmermann}.    
\end{proof}

For $-1\le r\le 1$ the number of $r$-classes on a del Pezzo surface of degree $d$ (different from $\bF_1$) is finite and given in Table~\ref{table_1}. Explicitly, assume $d\ge 5$. Then 
\begin{align}
\notag
    \text{the $-1$-classes are} \quad &E_i\quad \text{and $H-E_{i_1}-E_{i_2}$,}\\
\label{eq_0classes}
    \text{the $0$-classes are}\quad  & H-E_i \quad\text{and $2H-E_{i_1}-E_{i_2}-E_{i_3}-E_{i_4}$ (for $d= 5$),}\\
\label{eq_1classes}
    \text{$1$-classes are} \quad  &H \quad \text{and $2H-E_{i_1}-E_{i_2}-E_{i_3}$ (for $d\le 6$).}
\end{align}
On $\bF_0$, the $0$-classes are $h_1,h_2$ (the fibres of two $\P^1$-fibrations) and there are no $1$-classes.

Throughout the paper we will denote $0$-classes by $h_i$ and $1$-classes by $H_i$.

\begin{table}
    \centering
    $$
    \begin{array}{|c||c|c|c|c|c|}
    \hline
         d& 9&8&7& 6&5 \\ \hline
%        r=-2&8&20&40&72&126&240\\  
         r=-1&  0&0&3&  6&10\\  
         r=0&  0 &2& 2&   3&5\\  
         r=1&  1&0&1&  2&5  \\
    \hline
    \end{array}
    $$
    \caption{Number of $r$-classes on del Pezzo surfaces (different from $\bF_1$) of degrees $9$ to $5$}
    \label{table_1}
\end{table}

\section{Semi-orthogonal decompositions and their descent}

\subsection{Preliminaries on triangulated  categories}

Here we recall necessary notions, with an accent on SOD-s, exceptional collections, and mutations. Our references in this section are~\cite{BK} and~\cite{Huyb}.

Let $\bT$ be a triangulated category. We denote $\Hom^i(T_1,T_2):=\Hom(T_1,T_2[i])$ for $T_1,T_2\in\bT$, $i\in\Z$. Recall that a full subcategory $\bT'\subset \bT$ is called \emph{triangulated} if $\bT'$ is closed under shifts and taking cones, and \emph{thick} if $\bT'$ is triangulated and idempotent-closed in $\bT$.

\subsubsection{Semi-orthogonal decompositions and their mutations}
\begin{definition}
    A \emph{semi-orthogonal decomposition} (SOD) of $\bT$ is given by an ordered collection $\bA_1,\ldots,\bA_n$ of full triangulated subcategories in $\bT$ such that 
    \begin{enumerate}[label=(\alph*)]
        \item $\Hom(A_j,A_i)=0$ for all $i<j$ and $A_i\in\bA_i, A_j\in\bA_j$,
        \item $\bT$ is generated by $\bA_1,\ldots,\bA_n$ as a triangulated subcategory: $\bT$ is the smallest triangulated subcategory in $\bT$ that contains all of $\bA_i$.
    \end{enumerate}
    Semi-orthogonal decompositions are denoted by
    $$\bT=\langle\bA_1,\ldots,\bA_n\rangle.$$
\end{definition}
For a subcategory (or a collection of objects) $\bA\subset \bT$ the \emph{orthogonal subcategories}
$$\bA^{\perp}=\{T\mid \Hom^i(A,T)=0\quad\text{for all $A\in\bA$ and $i$}\}, \quad
^{\perp}\bA=\{T\mid \Hom^i(T,A)=0\quad\text{for all $A\in\bA$ and $i$}\}$$
are defined, they are full and triangulated.

Any component $\bA_i$ of an SOD is determined by the other components: one has 
$$\bA_i={}^\perp\langle\bA_1,\ldots,\bA_{i-1}\rangle\cap \langle\bA_{i+1},\ldots,\bA_n\rangle^\perp.$$
Hence the following trivial observation that will be frequently used in the paper:
\begin{lemma}
\label{lemma_thelastone}
    If $$\bT=\langle\bA_1,\ldots,\bA_n\rangle=\langle\bA'_1,\ldots,\bA'_n\rangle$$
    are two SOD-s and $\bA_i=\bA'_i$ for all $i=1,\ldots,n$, $i\ne k$, then $\bA_k=\bA'_k$ as well.
\end{lemma}

\begin{definition}
A full subcategory $\bA\subset \bT$ is called \emph{left (right) admissible} if the inclusion functor 
$\bA\to \bT$ has a left (right) adjoint functors. 
A full subcategory $\bA\subset \bT$ is called \emph{admissible} if it is both left and right admissible.
\end{definition}

For a left (right) admissible subcategory $\bA\subset \bT$ there is a semi-orthogonal decomposition
$\bT=\langle\bA, {}^{\perp}\bA\rangle$ (resp. $\bT=\langle\bA^{\perp}, \bA\rangle$).
For an admissible subcategory $\bA\subset \bT$ one has both:
$$\langle\bA^{\perp}, \bA\rangle=\bT=\langle\bA, {}^{\perp}\bA\rangle.$$
Moreover, the orthogonals $\bA^{\perp}$ and ${}^{\perp}\bA$ are equivalent via \emph{mutation functors}
$$R_{\bA}\colon \bA^{\perp}\to {}^{\perp}\bA, \quad L_{\bA}\colon  {}^{\perp}\bA\to \bA^{\perp},$$
where $R_{\bA}$ is the restriction to $\bA^{\perp}$ of the right adjoint $\bT\to {}^{\perp}\bA$ to the inclusion, and 
$L_{\bA}$ is the restriction to ${}^{\perp}\bA$ of the left adjoint $\bT\to \bA^{\perp}$ to the inclusion. These mutation functors are inverse to each other.

Assume a semi-orthogonal decomposition 
$$\bT=\langle\bA_1,\ldots,\bA_n\rangle$$
consists of admissible subcategories. For $1\le i\le n-1$, one defines the $i$-th right mutation as the SOD
$$\bT=\langle\bA_1,\ldots,\bA_{i-1},\bA_{i+1},R_{\bA_{i+1}}\bA_i, \bA_{i+2},\ldots,\bA_n \rangle.$$
Note that  $\bA_i$  and $R_{\bA_{i+1}}\bA_i$ are equivalent via mutation functors $R_{\bA_{i+1}}$,
$L_{\bA_{i+1}}$.
 
Similarly, for $i=2,\ldots,n$ one defines $i$-th left  mutation  as the SOD
$$\bT=\langle\bA_1,\ldots,\bA_{i-2},L_{\bA_{i-1}}\bA_i, \bA_{i-1},\bA_{i+1},\ldots,\bA_n \rangle.$$
As above, $\bA_i$  is equivalent to $L_{\bA_{i-1}}\bA_i$ via mutation functors. 

Assume that all subcategories involved are admissible, then right and left mutations define operations on the set of SOD-s of a triangulated category $\bT$ into $n$ components, we will denote them by $R_i$ (for $i=1,\ldots,n-1$) and $L_i$ (for $i=2,\ldots,n$). 
Operations $R_i$ and $L_{i+1}$ are inverse to each other for $i=1,\ldots,n-1$. 
Moreover, mutations define an action of the $n$-th braid group on the set of SOD-s of $\bT$ into $n$ components. We will say that two SOD-s are \emph{mutation-equivalent} if there is a sequence of mutations sending one into another, i.e., they belong to the same orbit.

\subsubsection{Exceptional objects}
Let $\kk$ be a field. Recall that a category  is \emph{$\kk$-linear} if its $\Hom$ spaces are $\kk$-vector spaces and the composition is $\kk$-bilinear. A $\kk$-linear category is \emph{$\Hom$-finite} if all $\Hom$-spaces are finite-dimensional. For a triangulated  $\kk$-linear category $\bT$ and $T_1,T_2\in\bT$ we denote 
$$\Hom^{\bullet}(T_1,T_2):=\oplus_{i\in\Z} \Hom^i(T_1,T_2),$$
this is a graded vector space. Category $\bT$ is called \emph{$\Ext$-finite} if the space $\Hom^{\bullet}(T_1,T_2)$
is finite-dimensional for all $T_1,T_2$ in $\bT$. 

Let $\bT$ be a triangulated $\kk$-linear category. 
\begin{definition}
\label{def_exceptionalobject}
An object $E\in\bT$ is called \emph{exceptional} if 
\begin{itemize}
    \item $\End(E)=\kk$, and
    \item $\Hom^i(E,E)=0$ for $i\ne 0$.
\end{itemize}
More generally, an object $E\in\bT$ is called \emph{weakly exceptional} if 
\begin{itemize}
    \item $\End(E)$ is a finite-dimensional division $\kk$-algebra, and
    \item $\Hom^i(E,E)=0$ for $i\ne 0$.
\end{itemize}
If $E$ is weakly exceptional and $\End(E)=D$ we say that $E$ is \emph{$D$-exceptional}. Note also that exceptional and weakly exceptional are the same if $\kk$ is  algebraically closed.
\end{definition}
\begin{definition}
\label{def_exceptionalcollection}
A collection of objects $(E_1,\ldots,E_n)$ is called \emph{exceptional (weakly exceptional)} if all $E_i$ are 
exceptional (weakly exceptional) and 
\begin{itemize}
    \item $\Hom^i(E_l,E_k)=0$ for all $i$ and $k<l$.
\end{itemize}
A weakly exceptional collection of objects $(E_1,\ldots,E_n)$ is called a \emph{block} if 
\begin{itemize}
    \item $\Hom^i(E_l,E_k)=0$ for all $i$ and $k\ne l$.
\end{itemize}
\end{definition}
If for some weakly  exceptional collection $(E_1,\ldots,E_n)$ a subcollection $(E_p,\ldots,E_q)$ is a block, it is standard to write it vertically:  
$$\left( E_1,\ldots,E_{p-1},
\begin{smallmatrix}
E_p\\ \ldots \\ E_q
\end{smallmatrix} ,E_{q+1},\ldots,E_n\right)$$

\begin{definition}
A weakly exceptional collection of objects $(E_1,\ldots,E_n)$ in $\bT$ is called \emph{full} if these objects generate $\bT$ as a triangulated category: $\bT$ is the smallest triangulated subcategory in $\bT$ that contains all of $E_i$.
\end{definition}

Let $\bT$ be an $\Ext$-finite $\kk$-linear triangulated category. For any $D$-exceptional object $E$ in $\bT$, the subcategory $\langle E\rangle$ generated by $E$ is admissible. The right adjoint functor $\bT\to \langle E\rangle$ is given by 
$$\bT\ni T\mapsto \Hom^{\bullet}(E,T)\otimes_D E=\oplus_i \Hom^i(E,T)\otimes_D E[-i]\in \langle E\rangle,$$ 
and the left adjoint functor is given by 
$$\bT\ni T\mapsto \Hom^{\bullet}(T,E)^*\otimes_D E=\oplus_i \Hom^i(T,E)^*\otimes_D E[i]\in \langle E\rangle.$$
Consequently, mutation functors $R_{E}=R_{\langle E\rangle}$ and $L_E=L_{\langle E\rangle}$ via $E$ fit into exact triangles 
\begin{gather}
\label{def_EmutationR}
R_E(T)\to T\xrightarrow{\mathrm{ev}} \Hom^{\bullet}(T,E)^*\otimes_D E \to R_E(T)[1],\\
\label{def_EmutationL} L_E(T)[-1]\to  \Hom^{\bullet}(E,T)\otimes_D E \xrightarrow{\mathrm{ev}} T\to  L_E(T),
\end{gather}
where the maps marked by $\mathrm{ev}$ are canonical evaluation maps.

It follows that any weakly exceptional and full collection of objects $E_1,\ldots,E_n$ in $\bT$ generates a semi-orthogonal decomposition
$$\bT=\langle \langle E_1\rangle, \ldots, \langle E_n\rangle
\rangle.$$

\subsubsection{Serre functor}

Let $\bT$ be a $\Hom$-finite $\kk$-linear category. A $\kk$-linear functor $S\colon \bT\to \bT$ is called a \emph{Serre functor} if
\begin{enumerate}[label=(\alph*)]
    \item $S$  is an equivalence, and
    \item there are isomorphisms of $\kk$-vector spaces for all $T_1,T_2\in\bT$
    \begin{equation}
        \label{eq_serre}
        \Hom(T_1,T_2)\xrightarrow{\sim} \Hom(T_2,S(T_1))^*,
    \end{equation}
    natural in $T_1,T_2$ (where $*$ denotes the dual $\kk$-vector space).
\end{enumerate}
A Serre functor, if exists, is unique up to an isomorphism. A Serre functor on a triangulated category is automatically exact. Equality~\eqref{eq_serre} is known as \emph{Serre duality}.

Assume $\bT$ has a Serre functor and
 $\bT=\langle\bA_1,\ldots,\bA_n\rangle$
is an SOD with admissible components. Then there are SOD-s
$$\bT=\langle S(\bA_n),\bA_1,\ldots,\bA_{n-1}\rangle, \qquad \bT=\langle \bA_2,\ldots,\bA_{n}, S^{-1}(\bA_1)\rangle.$$
In particular, by Lemma~\ref{lemma_thelastone}
\begin{equation}
    \label{eq_SerreMutat1}
    S(\bA_n)=L_{\bA_1}L_{\bA_2}\ldots L_{\bA_{n-1}}(\bA_n), \qquad
S^{-1}(\bA_1)=R_{\bA_n}R_{\bA_{n-1}}\ldots R_{\bA_{2}}(\bA_1).
\end{equation}

\subsection{Preliminaries on derived categories of coherent sheaves}\label{sec_derivedcat_coh}
Here we recall necessary background on derived categories of coherent sheaves on algebraic varieties and   provide some types of mutations for exceptional collections on surfaces that will be used in the proof of main results.

Recall that all varieties are supposed to be smooth and projective.  
We denote the bounded derived category of coherent sheaves on a variety $X$ over a field $\kk$ by $\Db(X)$. This is a $\kk$-linear $\Ext$-finite triangulated category. It has a Serre functor given by
$$S_X(-)=-\otimes \omega_X[\dim X].$$
It follows from~\cite{BK} that $\Db(X)$ and its left or right admissible subcategories are saturated, in particular, any left or right admissible subcategory is admissible. Therefore, any component of any SOD consists of admissible subcategories, and all mutations exist. 

For a morphism $f\colon X\to Y$ between algebraic varieties, we will denote by $f_*\colon \Db(X)\to \Db(Y)$ and $f^*\colon \Db(Y)\to \Db(X)$ the \textbf{derived} push-forward and pull-back functors respectively, unless  otherwise stated. Also, for two complexes of coherent sheaves, $\otimes$ will denote their \textbf{derived} tensor product.

\begin{example}[Beilinson~{\cite{Beilinson}}]
\label{example_Beilinson}
    The collection of sheaves $\cO(a-n),\ldots,\cO(a-1),\cO(a)$ on $\P^n$ is full and exceptional for any $a\in \Z$. Therefore, there are semi-orthogonal decompositions $\cA_a$:
    \begin{equation*}
              \Db(\P^n)=\langle\langle\cO(a-n)\rangle, \ldots, \langle\cO(a-1)\rangle,\langle\cO(a)\rangle\rangle.
    \end{equation*}
    It follows from Lemma~\ref{lemma_thelastone} or from~\eqref{eq_SerreMutat1} that  
    $$L_{\langle\cO(a-n)\rangle}\ldots L_{\langle\cO(a-1)\rangle}(\langle\cO(a)\rangle)\cong \langle\cO(a-n-1)\rangle,$$ hence $\cA_{a-1}$ is mutation-equivalent to $\cA_a$.
    Therefore, these SOD-s are mutation-equivalent for all $a\in\Z$.
\end{example}

\begin{lemma}[Bridgeland~{\cite[Prop. 2.3]{Bridgeland-flops}}]
\label{lemma_Bridgeland}
Let $f\colon X\to Y$ be a rational derived contraction between smooth varieties (see Definition~\ref{def_rdc}). Then the derived  inverse image functor
$$f^*\colon \Db(Y)\to \Db(X)$$
is fully faithful. 
\end{lemma}
\begin{proof}
The adjunction unit $\id_{\Db(Y)}\to f_*f^*$ is given by tensoring with the map $\cO_Y\to f_*\cO_X=\cO_Y$ by the projection formula, hence is an isomorphism. It follows then  that $f^*$ is fully faithful:
$$\Hom_X(f^*F_1,f^*F_2)=\Hom_Y(F_1,f_*f^*F_2)=\Hom_Y(F_1,F_2)$$
for all $F_1,F_2\in\Db(Y)$.
\end{proof}

For a blow-up at a smooth centre one can say more:
\begin{proposition}[Orlov's blow-up formula,~{\cite{Orlov-blowup}}]\label{blowupformula}
    Let $Z\subset X$ be smooth varieties, let $r=\codim_XZ$. Consider the blow-up of $X$ at $Z$:
    $$\xymatrix{\tZ \ar[r]^j \ar[d]^{\pi} & \tX\ar[d]^p\\ Z\ar[r] & X.}$$
    Then one has  SOD-s
    \begin{multline*}
    \Db(\tX)=\langle j_*(\pi^*\Db(Z)\otimes \cO_{\tZ}(-r+1)), \ldots, 
    j_*(\pi^*\Db(Z)\otimes \cO_{\tZ}(-1)), p^*\Db(X)\rangle=\\ =
    \langle p^*\Db(X),j_*(\pi^*\Db(Z)\otimes \cO_{\tZ}), \ldots, 
    j_*(\pi^*\Db(Z)\otimes \cO_{\tZ}(r-2))\rangle,    
    \end{multline*}
    where the component are the images of fully faithful functors $\Db(Z)\to \Db(\tX)$ and  $\Db(X)\to \Db(\tX)$.
q\end{proposition}
As a special case one gets (using Lemma~\ref{lemma_thelastone})
\begin{proposition}
\label{prop_blowup2}
    Let $X$ be a surface over $\kk$, let $P_1,\ldots,P_n\in X$ be $\kk$-points. Let $\tX\xrightarrow{p} X$ be the blow-up at $P_1,\ldots,P_n$ with the exceptional divisors $E_1,\ldots,E_n$. Then the objects $\cO_{E_i}(j)$ are exceptional for $i=1,\ldots,n$ and $j\in\Z$, and 
    $$ \left\langle \left\langle\begin{smallmatrix}
\cO_{E_{1}}(-1)\\ \ldots \\ \cO_{E_{n}}(-1)
\end{smallmatrix} \right\rangle, p^*\Db(X)\right\rangle \quad\text{and}\quad
\left\langle p^*\Db(X),  \left\langle\begin{smallmatrix}
\cO_{E_{1}}\\ \ldots \\ \cO_{E_{n}}
\end{smallmatrix} \right\rangle\right\rangle$$
are mutation-equivalent SOD-s of $\Db(\tilde X)$.
\end{proposition}

\begin{remark}\label{rem:leftmutation}
Let $E\in\Coh(X)$ be an exceptional vector bundle, and let $T\in\Coh(X)$. Then $\Hom^\bullet(E,T)\cong \rH^\bullet(X,T\otimes E^\vee)$.
If $\Hom^\bullet(E,T)=\kk^{\oplus r}[-d]$ (i.e. concentrated in degree $d$), then one can sometimes determine the exact triangle \eqref{def_EmutationL} (and hence the left mutation $L_E(T)$) with an exact sequence of sheaves: 
\begin{enumerate}
    \item If $d=0$ and the evaluation morphism $\ev\colon E^{\oplus r}\to T$ is surjective then $L_E(T)=K[1]$, where $K$ fits into an exact sequence of sheaves $0\to K \to  E^{\oplus r}\xrightarrow{\ev} T \to 0$.
    \item If $d=1$  then  $L_E(T)$ is given by the universal extension: $L_E(T)=A$ is the sheaf fitting into the exact sequence of sheaves $0\to T \to A\to   E^{\oplus r}\to 0$ induced by $\ev\colon E^{\oplus r}\to T[1]$.
\end{enumerate}
\end{remark}

Three following lemmas are standard, we include them for the convenience of the readers.

\begin{lemma}
\label{lemma_3mutations}
Let $X$ be a surface with $\rH^{\bullet}(X,\cO_X)=\kk$,
let $E$ be a $(-1)$-curve on $X$ and $\cO(D)$ be a line bundle on $X$. Then 
\begin{enumerate}
    \item\label{it:3mut_exceptional} the pairs $(\cO(D),\cO_E(a))$ and $(\cO_E(a),\cO(D-E))$ are exceptional if and only if $a=D\cdot E$;
    \item\label{it:3mut_mutations} three pairs below are exceptional and mutations of each other (up to a shift) as shown, where  $a=D\cdot E$:
$$\xymatrix{& (\cO(D-E), \cO(D)) \ar@<1mm>[rd]^{L_2}  \ar@<1mm>[ld]^(0.3){R_1}& \\ (\cO(D), \cO_E(a)) 
\ar@<1mm>[ru]^{L_2}  \ar@<1mm>[rr]^{R_1}& & (\cO_E(a),\cO(D-E)). \ar@<1mm>[ll]^{L_2}  \ar@<1mm>[lu]^(0.7){R_1}}$$
\end{enumerate} 
\end{lemma}
\begin{proof}
Note that all line bundles on $X$ are exceptional by the assumption $\rH^{\bullet}(X,\cO_X)=\kk$.

For \eqref{it:3mut_exceptional}, the pair $(\cO_E(a),\cO(D-E))$ is  exceptional if and only if 
$$0=\Hom^{\bullet}(\cO(D-E), \cO_E(a))=\rH^{\bullet}(\P^1,\cO_{\P^1}(a-E\cdot (D-E))),$$
which is equivalent to $a-E\cdot (D-E)=-1$,  $a=D\cdot E$. Similarly for the first pair.

For \eqref{it:3mut_mutations}, the mutations in question are given by the exact sequence (see e.g. Remark~\ref{rem:leftmutation})
\begin{equation}\label{eq_sesMinusOne}
    0\to \cO(D-E)\to \cO(D)\to \cO_E(a)\to 0.
\end{equation}
\end{proof}

\begin{lemma}
\label{lemma_OE}
Let $X$ be a surface with $\rH^{\bullet}(X,\cO_X)=\kk$, let $E_1,\ldots,E_n\subset X$ be a collection of disjoint $(-1)$-curves. Then the SOD-s 
\[\left\langle \left\langle\begin{smallmatrix}
\cO_{E_{1}}(-1)\\ \ldots \\ \cO_{E_{n}}(-1)
\end{smallmatrix} \right\rangle, \bA,\langle\cO\rangle\right\rangle
=\Db(X)= \left\langle \bA,\left\langle\begin{smallmatrix}
\cO(-E_{1})\\ \ldots \\ \cO(-E_{n})
\end{smallmatrix} \right\rangle,\langle\cO\rangle\right\rangle\]
are mutation-equivalent.
\end{lemma}
\begin{proof}
Follows from Proposition~\ref{prop_blowup2} and Lemma~\ref{lemma_3mutations}.
\end{proof}

\begin{lemma}
\label{lemma_hmutation}
Let $X$ be a del Pezzo surface of degree $K_X^2\ge 3$,  let $h$ be a $0$-class on $X$, and $\cO(D)$ be a line bundle on $X$. Then the following pairs are exceptional and are mutations of each other (up to a shift) as shown:
$$\xymatrix{(\cO(D-h), \cO(D)) \ar@<1mm>[rr]^{R_1} &&  (\cO(D), \cO(D+h))\ar@<1mm>[ll]^{L_2}}.$$
\end{lemma}
\begin{proof}
By Lemma~\ref{lemma_classes}(2), linear system $|h|$ defines a conic bundle $f\colon X\to\P^1$ so that $\cO(h)\cong f^*\cO_{\P^1}(1)$. By Lemma~\ref{lemma_Bridgeland} pull-back functor $f^*$ is fully faithful and hence commute with mutations. We have the mutation of exceptional objects  $L_{\langle \cO_{\P^1}\rangle}(\langle \cO_{\P^1}(1)\rangle)\cong \langle \cO_{\P^1}(-1)\rangle$ by Example~\ref{example_Beilinson}, therefore $L_{\langle \cO\rangle}(\langle \cO(h)\rangle)\cong \langle \cO(-h)\rangle$. The results follows now by twisting by $D$.
\end{proof}

\subsection{Three-block semi-orthogonal decompositions for del Pezzo surfaces}

Here we describe semi-orthogonal decompositions for del Pezzo surfaces of degree $\ge 5$ that consist of three blocks of exceptional objects and are invariant under all automorphisms of the surfaces. They have been found  by Karpov and Nogin in~\cite{KarpovNogin}, here we recall the construction. 

\begin{lemma}[Three-block decompositions]
\label{lem:3blockdecomp}
Let $\kk$ be  an algebraically closed field and $X=X_d$ be a del Pezzo surface of degree $d$ over $\kk$. We write $\{h_i\}_{i\in I}$ for the $0$-classes and $\{H_j\}_{j\in J}$ for the $1$-classes on $X$. Then the following are $3$-block decompositions of $\Db(X)$, and each block is preserved by $\Aut(X)$:
    \begin{align*}
        \Db(\P^2)= & \left\langle 
        \langle\cO(-2H)\rangle, 
        \langle\cO(-H)\rangle, 
        \langle\cO\rangle
        \right\rangle, & I&=\emptyset, &J&=\{1\},\\
        \Db(\P^1\times\P^1)= & \left\langle 
        \langle\cO(-h_1-h_2)\rangle,
        \left\langle \begin{smallmatrix}
        \cO(-h_1)\\ \cO(-h_2)
        \end{smallmatrix} \right\rangle,
        \langle\cO\rangle
        \right\rangle, & I&=\{1,2\},&J&=\emptyset, \\
        \Db(X_6)= & \left\langle
         \left\langle \begin{smallmatrix}
        \cO(-H_1)\\ \cO(-H_2)
        \end{smallmatrix} \right\rangle,
        \left\langle \begin{smallmatrix}
        \cO(-h_1)\\ \cO(-h_2) \\ \cO(-h_3)
        \end{smallmatrix} \right\rangle,
        \langle\cO\rangle
        \right\rangle,& I&=\{1,2,3\},&J&=\{1,2\},\\
        \Db(X_5)= & \left\langle
         \langle \cE\rangle,
        \left\langle \begin{smallmatrix}
        \cO(-h_1)\\ \ldots \\ \cO(-h_5)
        \end{smallmatrix} \right\rangle,
        \langle\cO\rangle
        \right\rangle, & I&=\{1,\ldots,5\},&J&=\{1,\ldots,5\},
    \end{align*}
    where $\cE$ is a vector bundle of rank $2$ on $X_5$ fitting into an exact sequence 
    \begin{equation}
        \label{eq_vectorbundleE}
        0\to \cO(-H_i)\to\cE\to \cO(-h_i)\to 0
    \end{equation}    for every $i\in\{1,\ldots,5\}$ such that $h_i+H_i=-K_X$.
    Moreover, for $d=6$ or $5$ these SOD-s are obtained from the SOD provided by the blow-up formula applied to the blow-up $X\to \P^2$ at $r=9-d$ points:
    \begin{equation}
    \label{eq_BlrP2}
    \Db(X_d)= \left\langle
        \left\langle \begin{smallmatrix}
        \cO_{E_1}(-1)\\ \cdots\\
        \cO_{E_{r}}(-1)
        \end{smallmatrix} \right\rangle, 
        \langle\cO(-2H)\rangle, 
        \langle\cO(-H)\rangle, 
        \langle\cO\rangle
        \right\rangle,
\end{equation}
by mutations and merging two completely orthogonal components into one.
\end{lemma}

\begin{proof}
The decompositions for $X=\P^2$ and $X=\P^1\times\P^1$ come from the standard Beilinson exceptional collections, see Example~\ref{example_Beilinson}.
For $d\in\{6,5\}$, decompositions with three blocks were found in \cite{KarpovNogin}, and it was remarked in \cite{Elagin} that each block is preserved by $\Aut(X)$.  However, since we use slightly different $3$-block decompositions, we recall how to obtain them from~\eqref{eq_BlrP2} by mutations.

In both cases, we first use Lemma~\ref{lemma_3mutations} to perform mutations
\begin{multline*}
\text{\eqref{eq_BlrP2}}\overset{L_2}\leadsto
\left\langle
    \langle\cO(-2H+\sum E_i)\rangle, 
        \left\langle \begin{smallmatrix}
        \cO_{E_1}(-1)\\ \cdots\\
        \cO_{E_{r}}(-1)
        \end{smallmatrix} \right\rangle, 
        \langle\cO(-H)\rangle, 
        \langle\cO\rangle
        \right\rangle
        \overset{R_2}\leadsto\\
\overset{}\leadsto
\left\langle
    \langle\cO(-2H+\sum E_i)\rangle, 
        \langle\cO(-H)\rangle, 
        \left\langle \begin{smallmatrix}
        \cO(-H+E_1)\\ \cdots\\
        \cO(-H+E_r)
        \end{smallmatrix} \right\rangle, 
        \langle\cO\rangle
        \right\rangle.
\end{multline*}

For $X=X_6$ we have $r=3$, $H_1=2H-(E_1+E_2+E_3)$, $H_2=H$, $h_i=H-E_i$ for $i=1,2,3$ (see~\eqref{eq_0classes}) and~\eqref{eq_1classes}). Since $\cO(-H_1)$ and $\cO(-H_2)$ are orthogonal, we obtain exactly the desired $3$-block decomposition.

For $X=X_5$, we have $r=4$, $h_i=H-E_i$, $H_i=2H-(E_1+E_2+E_3+E_4)+E_i$ for $i=1,\ldots,4$, and $h_5=2H-(E_1+E_2+E_3+E_4)$, $H_5=H$  (see~\eqref{eq_0classes}) and~\eqref{eq_1classes}).  
Here, an additional mutation is needed:
\begin{align*}
        \left\langle
        \langle\cO(-h_5)\rangle,
        \langle\cO(-H_5)\rangle, 
        \left\langle \begin{smallmatrix}
        \cO{(-h_1)}\\ \cdots\\
        \cO{(-h_4)}
        \end{smallmatrix} \right\rangle,
        \langle\cO\rangle
        \right\rangle \overset{L_2}\leadsto 
        \left\langle
        \langle\cE_5\rangle, 
        \langle\cO(-h_5)\rangle,
        \left\langle \begin{smallmatrix}
        \cO{(-h_1)}\\ \cdots\\
        \cO{(-h_4)}
        \end{smallmatrix} \right\rangle,
        \langle\cO\rangle
        \right\rangle.
\end{align*}
We used that
$\Hom^{\bullet}(\cO(-h_5), \cO(-H_5))\cong\rH^\bullet(X,\cO(-H_5+h_5))\cong \kk[-1]$, therefore the exact triangle obtained from \eqref{eq_vectorbundleE} is a rotation of~\eqref{def_EmutationL}, giving $\cE_5=L_{\cO(-h_5)}(\cO(-H_5))$ for $i=5$.
Since the $\cO(-h_i)$ are orthogonal to each other, we obtain a $3$-block decomposition. Using the map $X\to\P^2$ given by the 1-class $H_i$, one gets a similar SOD with $\cE_i$ fitting into~\eqref{eq_vectorbundleE}. Since all these $\cE_i$ generate the same subcategory in $\Db(X)$ and are exceptional, they are isomorphic and may be denoted by $\cE$. 

Finally, it is clear that each block is preserved by $\Aut(X)$ because the set of $0$-classes as well as the set of $1$-classes are preserved, and if two of three blocks are preserved then the third one also is by Lemma~\ref{lemma_thelastone}. 
\end{proof}

We will denote the $3$-block decomposition of $\Db(X)$ defined in Lemma~\ref{lem:3blockdecomp} by $\cS(X)$. These decompositions are compatible to each other in the following sense.

\begin{lemma}
\label{lem:StandardDecompbirrich}
    Let $f\colon X\to Y$ be a birational morphism between two del Pezzo surfaces of degrees $K_X^2,K_Y^2\in\{5,6,8,9\}$, $X,Y\not\cong\bF_1$.
    Then the $4$-block decomposition \[\langle \ker f_*, f^*\cS(Y)\rangle\] of $\Db(X)$ is mutation-equivalent to a refinement of $\cS(X)$.
\end{lemma}
\begin{proof}
    If $K_Y^2=9$, take the mutations from Lemma~\ref{lem:3blockdecomp}.
    The cases $(K_X^2,K_Y^2)\in\{(5,6),(5,8),(6,8)\}$ remain.

    {$\mathbf{(6,8)}$}: We have $K_X=-h_1-h_2-h_3$ and we compare \begin{equation}\label{eq:6-8--2}
        \langle \ker f_*,f^*\cS(Y)\rangle =
        \left\langle
        \left\langle \begin{smallmatrix}
        \cO_{E_1}(-1)\\ \cO_{E_2}(-1)
        \end{smallmatrix} \right\rangle,
        \langle \cO(-h_1-h_2) \rangle,
         \left\langle \begin{smallmatrix}
        \cO(-h_1)\\ \cO(-h_2)
        \end{smallmatrix} \right\rangle,
        \langle\cO\rangle
        \right\rangle
    \end{equation}
    with
    \begin{equation}\label{eq:6-8--1}
        \left\langle
        \left\langle \begin{smallmatrix}
        \cO(-H_1)\\ \cO(-H_2)
        \end{smallmatrix} \right\rangle,
         \left\langle \begin{smallmatrix}
        \cO(-h_1)\\ \cO(-h_2)
        \end{smallmatrix} \right\rangle,
        \langle \cO(-h_3) \rangle,
        \langle\cO\rangle
        \right\rangle,
    \end{equation}
    which is a refinement of $\cS(X)$. We do mutations as follows:
\begin{multline*}
\text{\eqref{eq:6-8--2}}\overset{R_1}\leadsto    
\left\langle 
\langle \cO(-h_1-h_2)\rangle,
\left\langle \begin{smallmatrix}
\cL_1\\ \cL_2
\end{smallmatrix} \right\rangle,
\left\langle \begin{smallmatrix}
\cO(-h_1)\\ \cO(-h_2)
\end{smallmatrix} \right\rangle,
\langle\cO\rangle, 
\right\rangle\overset{-K}\leadsto   
\left\langle 
\left\langle \begin{smallmatrix}
\cL_1\\ \cL_2
\end{smallmatrix} \right\rangle,
\left\langle \begin{smallmatrix}
\cO(-h_1)\\ \cO(-h_2)
\end{smallmatrix} \right\rangle,
\langle\cO\rangle, 
\langle \cO(h_3)\rangle\right\rangle\overset{L_4}\leadsto
\text{\eqref{eq:6-8--1}}
\end{multline*}
Here and later on we denote by $-K$ the mutation of the first component to the end, which is given by tensoring with $\cO(-K_X)$, see~\eqref{eq_SerreMutat1}. The last mutation is by Lemma~\ref{lemma_hmutation}. Note that we do not track one of the components but use Lemma~\ref{lemma_thelastone} to deduce that the result of our mutations coincides with~\eqref{eq:6-8--1}.

    $\mathbf{(5,6)}$: Classes $h_1,h_2,h_3,H_1,H_2$ on $X$ are the pull-backs from $Y$. We have $h_4=-H_1-K_X, h_5=-H_2-K_X$
    by Lemma~\ref{lemma_dualclasses}, and we compare
    \begin{equation}\label{eq:5-6--1}
        \langle \ker f_*,f^*\cS(X_6)\rangle = 
        \left\langle
        \langle \cO_{E}(-1) \rangle,
         \left\langle \begin{smallmatrix}
        \cO(-H_1)\\ \cO(-H_2)
        \end{smallmatrix} \right\rangle,
        \left\langle \begin{smallmatrix}
        \cO(-h_1)\\ \cO(-h_2) \\ \cO(-h_3)
        \end{smallmatrix} \right\rangle,
        \langle\cO\rangle
        \right\rangle
    \end{equation}
    with    \begin{equation}\label{eq:5-6--2}
        \left\langle
         \langle \cE\rangle,
        \left\langle \begin{smallmatrix}
        \cO(-h_1)\\ \cO(-h_2) \\ \cO(-h_3)
        \end{smallmatrix} \right\rangle,
        \left\langle \begin{smallmatrix}
        \cO(-h_4)\\  \cO(-h_5)
        \end{smallmatrix} \right\rangle,
        \langle\cO\rangle
        \right\rangle,
    \end{equation}
    which is a refinement of $\cS(X_5)$. They are mutation-equivalent via the following mutations:
    \begin{multline*}
\text{\eqref{eq:5-6--1}}{\overset{\text{Lemma}~\ref{lemma_OE}}\leadsto}
        \left\langle
        \left\langle \begin{smallmatrix}
        \cO(-H_1)\\ \cO(-H_2)
        \end{smallmatrix} \right\rangle,
        \left\langle \begin{smallmatrix}
        \cO(-h_1)\\ \cO(-h_2) \\ \cO(-h_3)
        \end{smallmatrix} \right\rangle,
        \langle \cO(-E) \rangle,
        \langle\cO\rangle
        \right\rangle
        \overset{-K}\leadsto
        \left\langle
        \left\langle \begin{smallmatrix}
        \cO(-h_1)\\ \cO(-h_2) \\ \cO(-h_3)
        \end{smallmatrix} \right\rangle,
        \langle \cO(-E) \rangle,
        \langle\cO\rangle,
         \left\langle \begin{smallmatrix}
        \cO(h_4)\\ \cO(h_5)
        \end{smallmatrix} \right\rangle
        \right\rangle
        \overset{L_4}\leadsto\\
        \leadsto
        \left\langle
        \left\langle \begin{smallmatrix}
        \cO(-h_1)\\ \cO(-h_2) \\ \cO(-h_3)
        \end{smallmatrix} \right\rangle,
        \langle \cO(-E) \rangle,
        \left\langle \begin{smallmatrix}
        \cO(-h_4)\\ \cO(-h_5)
        \end{smallmatrix} \right\rangle,
        \langle\cO\rangle
        \right\rangle
        \overset{L_2}\leadsto
        \eqref{eq:5-6--2}.
\end{multline*}

    $\mathbf{(5,8)}$: This time we denote by  $h_4,h_5$ the $0$-classes on $X$ pulled back from $Y$. For $i=1,2,3$ we note that $h_4+h_5$ is a $2$-class, $H_i=h_4+h_5-E_i$ is a $1$-class, and $h_i=-K_X-H_i$  by Lemma~\ref{lemma_dualclasses}.
    We compare \begin{equation}\label{eq:5-8--1}
        \langle \ker f_*,f^*\cS(Y)\rangle =\left\langle 
        \left\langle \begin{smallmatrix}
        \cO_{E_1}(-1)\\ \cO_{E_2}(-1)\\ \cO_{E_3}(-1)
        \end{smallmatrix} \right\rangle,
        \langle\cO(-h_4-h_5)\rangle,
        \left\langle \begin{smallmatrix}
        \cO(-h_4)\\ \cO(-h_5)
        \end{smallmatrix} \right\rangle,
        \langle\cO\rangle
        \right\rangle
    \end{equation}
    with \eqref{eq:5-6--2}, which is a refinement of $\cS(X_5)$.
    We do mutations
    \begin{multline*}
    \text{\eqref{eq:5-8--1}}\overset{R_1}\leadsto
    \left\langle 
        \langle\cO(-h_4-h_5)\rangle,
        \left\langle \begin{smallmatrix}
        \cO(-h_4-h_5+E_1)\\ \cO(-h_4-h_5+E_2)\\ \cO(-h_4-h_5+E_3)
        \end{smallmatrix} \right\rangle,
        \left\langle \begin{smallmatrix}
        \cO(-h_4)\\ \cO(-h_5)
        \end{smallmatrix} \right\rangle,
        \langle\cO\rangle
        \right\rangle=\\
           = \left\langle 
        \langle\cO(-h_4-h_5)\rangle,
        \left\langle \begin{smallmatrix}
        \cO(-H_1)\\ \cO(-H_2)\\ \cO(-H_3)
        \end{smallmatrix} \right\rangle,
        \left\langle \begin{smallmatrix}
        \cO(-h_4)\\ \cO(-h_5)
        \end{smallmatrix} \right\rangle,
        \langle\cO\rangle
        \right\rangle
        \overset{R_1}\leadsto
        \left\langle 
        \left\langle \begin{smallmatrix}
        \cO(-H_1)\\ \cO(-H_2)\\ \cO(-H_3)
        \end{smallmatrix} \right\rangle,
        \langle\widetilde \cE\rangle,
        \left\langle \begin{smallmatrix}
        \cO(-h_4)\\ \cO(-h_5)
        \end{smallmatrix} \right\rangle,
        \langle\cO\rangle
        \right\rangle
        \overset{-K}\leadsto\\ \leadsto
        \left\langle 
        \langle\widetilde \cE\rangle,
        \left\langle \begin{smallmatrix}
        \cO(-h_4)\\ \cO(-h_5)
        \end{smallmatrix} \right\rangle,
        \langle\cO\rangle,
        \left\langle \begin{smallmatrix}
        \cO(h_1)\\ \cO(h_2)\\ \cO(h_3)
        \end{smallmatrix} \right\rangle
        \right\rangle
        \overset{L_4}\leadsto
        \left\langle 
        \langle\widetilde \cE\rangle,
        \left\langle \begin{smallmatrix}
        \cO(-h_4)\\ \cO(-h_5)
        \end{smallmatrix} \right\rangle,
        \left\langle \begin{smallmatrix}
        \cO(-h_1)\\ \cO(-h_2)\\ \cO(-h_3)
        \end{smallmatrix} \right\rangle,
        \langle\cO\rangle
        \right\rangle
        \overset{R_2}\leadsto
        \eqref{eq:5-6--2}.
    \end{multline*}
\end{proof}

\begin{lemma}
\label{lemma_fCY}
Let $X$ be a del Pezzo surface of degree $1$ or $2$, let $\sigma\colon X\to X$ be the Bertini/Geiser involution. Let $\bA=\cO^{\perp}\subset \Db(X)$ and $S_{\bA}$ be the Serre functor on $\bA$. Then
\begin{align*}
S_{\bA}^3\cong \sigma^*[5] & \quad \text{for $d=1$,}\\
S_{\bA}^2\cong \sigma^*[3] & \quad \text{for $d=2$.}
\end{align*}
\end{lemma}
\begin{proof}
If $\deg X=1$ then $X$ is a double cover of $\P(1,1,2)$, ramified over a curve of degree $6$ (degree relative to $\cO_{\P(1,1,2)}(1)$), and $\sigma$ is the deck transformation. Viewing $\P(1,1,2)$ as a smooth stack, the map $X\to \P(1,1,2)$ is ramified exactly over the curve (and not over the stacky point). Hence the statement is a special case of~\cite[Cor. 5.3]{KuzPer}, where $M=\P(1,1,2)$, $k=1$, $d_1=3$, $\ind(M)=4$, $\ind(X)=1$,  $c=1$, $\cR_M=0$, and $\bT=\id$. Note that $(\cO, \cO(1), \cO(2), \cO(3))$ is a full and exceptional collection in $\Db(\P(1,1,2))$ by~\cite[Remark 2.2.6]{Canonaco}.

If $\deg X=2$ then $X$ is a double cover of $\P^2$ ramified over a quartic, and $\sigma$ is the deck transformation. The statement is a special case of~\cite[Cor. 5.3]{KuzPer}, where $M=\P^2=\P(1,1,1)$, $k=1$, $d_1=2$, $\ind(M)=3$, $\ind(X)=1$,  $c=1$, $\cR_M=0$, and $\bT=\id$.
\end{proof}

\subsection{Group actions on categories}
\label{section_group-actions}
Here we recall necessary definitions and facts about group actions on categories and equivariant categories. We refer to~\cite{Deligne}, \cite[4.1.3]{Drinfeld-et-al}, \cite{ShinderGroupActions} for details. 

\begin{definition}
\label{def_groupaction}
Let $G$ be a group and $\bC$ be a category. An \emph{action} of  $G$ on $\bC$ is given by 
\begin{itemize}
    \item a collection of autoequivalences $$\rho_g\colon \bC\to\bC, \quad g\in G,$$ and
    \item a collection of functorial isomorphisms 
    $$\epsilon_{g,h}\colon \rho_g\rho_h\to \rho_{gh}, \quad g,h\in G,$$
\end{itemize}
satisfying the cocycle condition given by the commutativity of the diagram of functorial isomorphisms: 
\begin{equation}
\label{eq_eeee}
\xymatrix{\rho_f\rho_g\rho_h \ar[rr]^{\epsilon_{f,g}\rho_h}\ar[d]^{\rho_f\epsilon_{g,h}}  && \rho_{fg}\rho_h\ar[d]^{\epsilon_{fg,h}}\\
\rho_f\rho_{gh}\ar[rr]^{\epsilon_{f,gh}} && \rho_{fgh}} 
\end{equation} 
for any $f,g,h\in\ G$.
An action is called \emph{strict} if $\rho_{gh}=\rho_g\rho_h$ and $\epsilon_{g,h}=\id$ for all $g,h$.
\end{definition}

\begin{remark}
    The definition of a strict action may seem to be more natural and appealing than that of a general action. 
    Moreover, there is a result (see~\cite[Th. 5.4]{ShinderGroupActions}) saying that any group action is equivalent to a strict group action. However, the framework of strict actions is too restrictive: many natural group actions on categories are not strict,
    and strictifiying them produces categories with unnecessary excessive objects. For example, non-strict $G$-actions on a category $\mathsf{pt}_{\kk}$ with one object and endomorphisms $\kk$ from Example~\ref{example_kpoint} are equivalent to strict actions on a category with $|G|$ isomorphic objects, but are not equivalent to a strict action on the one-object category $\mathsf{pt}_\kk$.
\end{remark}

We will call a category with a $G$-action a \emph{$G$-category}.

\begin{definition}
    Let $\bC, \bD$ be $G$-categories. A \emph{$G$-equivariant functor} from $\bC$ to $\bD$ is given by
\begin{itemize}
    \item a functor $$\phi\colon \bC\to\bD,$$ and 
    \item a collection of functorial isomorphisms
    $$\beta_g\colon \phi\rho_g^{\bC}\to \rho_g^{\bD}\phi,$$
\end{itemize}
satisfying for all $g,h\in G$ the compatibility conditions that are given by the commutativity of the diagram of functorial isomorphisms
\begin{equation}
\label{eq_ebbbe}
\xymatrix{\phi\rho_{g}^{\bC}\rho_h^{\bC} \ar[r]^{\beta_g\rho_h^{\bC}}\ar[d]^{\phi\epsilon_{g,h}^{\bC}}& \rho_{g}^{\bD}\phi\rho_h^{\bC} \ar[r]^{\rho_g^{\bD}\beta_{h}}& \rho_{g}^{\bD}\rho_h^{\bD}\phi\ar[d]^{ \epsilon_{g,h}^{\bD}\phi}\\
 \phi\rho_{gh}^{\bC}\ar[rr]^{\beta_{gh}}&& \rho_{gh}^{\bD}\phi.}
\end{equation}
\end{definition}

\begin{definition}
    Let $\bC, \bD$ be $G$-categories and $\phi,\psi\colon \bC\to \bD$ be $G$-equivariant functors. A \emph{$G$-morphism} (or a \emph{$G$-natural transformation}) from $\phi$ to $\psi$ is a  morphism of functors $f\colon \phi\to \psi$ such that
    $\beta_g^{\psi} \circ f\rho_g^{\bC}=\rho_g^{\bD}(f)\circ\beta_g^{\phi}$
    %$\beta_g^{\psi} f=\rho_g^{\bD}(f)\beta_g^{\phi}$ in $\Hom(\phi\rho_g^{\bC},\rho_g^{\bD}\psi)$
    for all $g$.
\end{definition}

Two categories with a $G$-action are called \emph{$G$-equivariantly equivalent} if there is a $G$-equivariant equivalence between them.   In particular, one can speak of $G$-equivariant equivalence of $G$-actions on a fixed category. The following observation will be used below, sometimes implicitly.
\begin{remark}
For two equivalent categories $\bC$ and $\bD$, $G$-actions on $\bC$ correspond to $G$-actions on $\bD$, up to equivariant equivalence.    
\end{remark}

Let $\bC$ be a  category with an action of a group $G$. We say that a  full subcategory $\bC'\subset \bC$ is \emph{$G$-invariant} if $\rho_g(C)\in\bC'$ for any object $C\in \bC'$, and remark that an invariant subcategory naturally has induced $G$-action. We say that an SOD of a triangulated category $\bT$ is \emph{$G$-invariant} if all its components are $G$-invariant.

\begin{remark}\label{rem_mutationOfGinvariantSOD}
    Any mutation of a $G$-invariant SOD is also a $G$-invariant SOD. Indeed, if 
$$\bT=\langle\bA_1,\ldots,\bA_n\rangle=\langle\bA_1,\ldots,\bA_{i-2},\bA_{i}, R_{\bA_i}\bA_{i-1},\bA_{i+1},  \bA_n\rangle$$ and all $\bA_j$ are $G$-invariant then
$$R_{\bA_i}\bA_{i-1}={}^\perp\langle\bA_1,\ldots,\bA_{i-2},\bA_i\rangle\cap \langle\bA_{i+1},\ldots,\bA_n\rangle^\perp$$
also is $G$-invariant. 
\end{remark}
The following lemma will be crucial for our study of atoms in subsequent sections.
\begin{lemma}
\label{lemma_atomswelldefined}    
Let $\bT$ be a triangulated category with an action of a group $G$.
If $\bT=\langle\bA_1,\ldots,\bA_n\rangle$ is a $G$-invariant SOD into admissible components then the mutation functors 
$$\bA_i\to R_{\bA_{i+1}}(\bA_i), \quad \bA_{i+1}\to L_{\bA_{i}}(\bA_{i+1})$$
are $G$-equivariant equivalences for all $i=1,\ldots,n-1$. 
In particular, if two SOD-s
$$\bT=\langle\bA_1,\ldots,\bA_n\rangle=\langle\bB_1,\ldots,\bB_n\rangle$$
are mutation-equivalent then
the collections of $G$-categories $\bA_1,\ldots,\bA_n$ and $\bB_1,\ldots,\bB_n$ are the same up to a permutation and $G$-equivariant equivalences.
\end{lemma}
\begin{proof}
    We check the claim for the right mutation.
    Recall that the mutation functor $\bA_i\to R_{\bA_{i+1}}(\bA_i)$ is the composition of the embedding functor $\bA_i\to \bT$ and the right adjoint functor $\bT\to {}^{\perp}\bA_{i+1}$ to the embedding 
    $^{\perp}\bA_{i+1}\to \bT$. The embedding of a $G$-invariant subcategory is naturally a $G$-equivariant functor, and the adjoint functor to a  $G$-equivariant functor is $G$-equivariant, see~\cite[Prop. 3.9]{ShinderGroupActions}. The second claim follows by induction.
\end{proof}

\begin{definition}
We say that an object $E$ in a $G$-category $\bC$ is \emph{$G$-invariant} if $\rho_g(E)$ is isomorphic to $E$ for all $g\in G$.
\end{definition}

\begin{definition}
\label{def_equivariant}
Let $\bC$ be a $G$-category and $E\in \bC$ be a $G$-invariant object. An \emph{equivariant structure} on $E$ (or a \emph{$G$-linearisation} of $E$) is given by a collection of isomorphisms for all $g$
$$\theta_g\colon E\to \rho_g(E)$$
subject to conditions for all $g,h$
\[\theta_{gh}=\epsilon_{g,h}\circ\rho_g(\theta_h)\circ\theta_g\]
in $\Hom(E,\rho_{gh}(E))$.

A $G$-invariant object is said to be \emph{$G$-linearisable} if it admits a $G$-equivariant structure. 
A \emph{$G$-equivariant object} $(E,\theta)$ in $\bC$ is an object $E\in\bC$ with fixed equivariant structure $\theta=(\theta_g)_{g\in G}$. 

A \emph{morphism between equivariant objects} $(E_1,\theta_1)$ and $(E_2,\theta_2)$ is given by a morphism $f\colon E_1\to E_2$ such that for all $g$
\[\rho_g(f)\circ\theta_{1,g}=\theta_{2,g}\circ f\]
in $\Hom(E_1, \rho_g(E_2))$.
Thus, one gets the category of $G$-equivariant objects in $\bC$ that we denote $\bC^G$.
\end{definition}

\begin{definition}
Let $\bC,\bD$ be $\kk$-linear categories and $\tau\in\Aut(\kk)$. We say that a functor $\phi\colon\bC\to\bD$ is \emph{$\tau$-linear} if $\phi(\lambda f)=\tau(\lambda)\phi(f)$ for any $\lambda\in\kk$ and a morphism $f$ in $\bC$.

Let $\sigma\colon G\to\Aut(\kk)$ be a homomorphism. We say that a $G$-action $(\rho_g,\epsilon_{g,h})$ on a $\kk$-linear category $\bC$ is \emph{$\sigma$-linear} if 
the functor $\rho_g\colon \bC\to \bC$ is $\sigma(g)$-linear for any $g\in G$. 

We will denote the multiplicative group $\kk^*$ with the induced $G$-action by $\kk^*_{\sigma}$: $g\in G$ takes $\lambda\in\kk^*$ to $\sigma(g)(\lambda)$. 
\end{definition}

\medskip

In this paper we will mostly be interested in categories with group actions as in the two examples below.
\begin{example}    
\label{example_G-action-on-X}
Let $X$ be an algebraic variety over $\kk$ with an action of a group $G$. Then the categories $\Coh X$ and $\Db(X)$ have $G$-actions given by push-forward functors $\rho_g=g_*$ so
that 
$g_*h_* = (gh)_*$ and we set $\epsilon_{g,h} = \id$. 
If a $G$-action is $\kk$-linear then the induced $G$-actions on $\Coh X$ and $\Db(X)$ are $\kk$-linear. More generally, if one has $\rH^0(X,\cO_X)=\kk$ then a $G$-action on $X$ defines a $G$-action $\sigma$ on $\kk$, and the induced $G$-actions on $\Coh X$ and $\Db(X)$ are $\sigma$-linear.

For a $G$-morphism of $G$-varieties $f\colon X\to Y$ the functors $f^*, f_*$ between $\Coh X, \Coh Y$ and between $\Db(X), \Db(Y)$ are $G$-equivariant.
\end{example}

\begin{example}
\label{example_kpoint}
Let $\kk$ be a field and $\mathsf{pt}_\kk$ be the category with one object $\bullet$ such that $\End(\bullet)=\kk$. Then an action $\rho$ of a group $G$ on  $\mathsf{pt}_\kk$ (by additive functors) is given by 
\begin{itemize}
    \item a homomorphism $\sigma\colon G\to\Aut(\kk)$ giving the action of $G$ on morphisms: $\sigma(g)=\rho_g$, and
    \item a collection $\alpha_{g,h}\in\kk^*$ giving isomorphisms $\epsilon_{g,h}$,   
\end{itemize}
where $\alpha_{g,h}$ are subject to relations~\eqref{eq_eeee}:
$$\alpha_{f,gh}\cdot \rho_f(\alpha_{g,h})=\alpha_{fg,h}\cdot \alpha_{f,g}$$
for $f,g,h$. These relations mean exactly that $\alpha$ is a $2$-cocycle of $G$ in the $G$-module $\kk^*_{\sigma}$. 
Vice versa, any homomorphism $\rho$ together with a $2$-cocycle $\alpha$ in $\kk^*_{\sigma}$ define a $G$-action on $\mathsf{pt}_\kk$ as above.
\end{example}

\begin{lemma}
\label{lemma_actiononapoint}
Two actions of $G$ on $\mathsf{pt}_\kk$ given by  $(\sigma,\alpha)$ and $(\sigma',\alpha')$ are $G$-equivariantly equivalent if and only if for some $\tau\in\Aut(\kk)$ one has $\sigma'=\tau\sigma\tau^{-1}$ and $[\alpha']=\tau([\alpha])\in \rH^2(G,\kk^*_{\sigma'})$ (where $\tau\colon \kk^*_{\sigma}\to \kk^*_{\sigma'}$ is an isomorphism of $G$-modules). In this case the equivalence is $\tau$-linear.

In particular, two such actions  are $G$-equivariantly \textbf{$\kk$-linearly} equivalent if and only if $\sigma=\sigma'$ and $[\alpha]=[\alpha']\in \rH^2(G,\kk^*_{\sigma})$. 
\end{lemma}
\begin{proof}
A $G$-equivariant  equivalence $\phi\colon \mathsf{pt}_\kk\to \mathsf{pt}_\kk$ is given by a functor $\phi$, which acts on morphisms by a field automorphism $\tau$, and a collection of isomorphisms $\beta_g\in\Aut(\bullet)=\kk^*$ satisfying the compatibility condition~\eqref{eq_ebbbe}:
$$\alpha'_{g,h} \cdot \sigma(g)(\beta_{h})\cdot \beta_g=\beta_{gh}\cdot \tau(\alpha_{g,h}),$$
which means exactly that $\tau(\alpha)/\alpha'$ is the coboundary of $\beta$.
\end{proof}

\begin{definition}
\label{def_obstruction}
Let $\bC$ be a $\kk$-linear category with a $\sigma$-linear action of a group $G$ and $E\in\bC$ be a $G$-invariant object with $\End(E)=\kk$. The \emph{(twisting) cohomology class associated to $E$} is the class $[\alpha]\in\rH^2(G,\kk_\sigma^*)$ obtained in the following way. 

The full subcategory $\tilde\bE\subset\bC$ of objects isomorphic to $E$ is $G$-invariant and so the action on $\bC$ restricts to a $G$-action on $\tilde\bE$.
The full subcategory $\bE\subset \Tilde{\bE}$ with one object $E$ is equivalent to $\Tilde{\bE}$, and so the $G$-action on $\Tilde{\bE}$ induces a $G$-action on $\bE$ (defined up to a $\kk$-linear $G$-equivalence).
The category $\bE$ is the one from Example~\ref{example_kpoint}, hence a cohomology class $[\alpha]\in \rH^2(G,\kk^*_{\sigma})$ is well-defined by Lemma~\ref{lemma_actiononapoint}.

Alternatively, we can define the twisting class as follows. For any $g\in G$, fix an isomorphism $\theta_g\colon E\to \rho_g(E)$. Take $\alpha_{g_1,g_2}\in\kk^*$ such that 
\begin{equation}
\label{eq_obstruction}
\theta_{g_1g_2}\cdot \alpha_{g_1,g_2}=\epsilon_{g_1,g_2}\circ \rho_{g_1}(\theta_{g_2})\circ\theta_{g_1}.
\end{equation}
One can check that $\alpha$ is a $2$-cocycle and its cohomology class is as above. 
\end{definition}

\begin{lemma}
\label{lemma_obstruction}
Under the assumptions of Definition~\ref{def_obstruction}, $E$ has a $G$-equivariant structure if and only if $[\alpha]=1$ in $\rH^2(G,\kk^*_{\sigma})$.
\end{lemma}
\begin{proof}
Changing the $\theta$-s by scalars, the $\alpha$-s change by a coboundary by~\eqref{eq_obstruction}. By definition, $E$ has an equivariant structure if and only if  one can choose $\theta$-s so that all $\alpha_{g_1,g_2}=1$ in~\eqref{eq_obstruction}, so the claim follows.
\end{proof}

Hence, the cohomology  class  $[\alpha]$ is the obstruction for $E$ to possess an equivariant structure.
If an invariant object fails to be equivariant, one can fix this by introducing
\begin{definition}
\label{def_twistedequivariant}
Let $G$ be a group, $\sigma\colon G\to \Aut \kk$ be a homomorphism and $\alpha\in Z^2(G,\kk_{\sigma}^*)$ be a $2$-cocycle.
\begin{enumerate}
\item 
Let $\bC$ be a $\kk$-linear additive category with a $\sigma$-linear $G$-action.   An \emph{$\alpha$-twisted $G$-equivariant object} is given by an object $E\in \bC$ and a collection $(\theta_g)_{g\in G}$ of isomorphisms $\theta_g\colon E\to \rho_g(E)$ satisfying~\eqref{eq_obstruction} for all $g_1,g_2\in G$. Morphisms between $\alpha$-twisted $G$-equivariant objects are defined as in Definition~\ref{def_equivariant}, so they form a category that we denote by $\bC^{G,\alpha}$.

\item
As a special case, let $X$ is an algebraic variety over $\kk$ with a $\sigma$-linear $G$-action. Then the category $\Coh (X)$ has a $\sigma$-linear $G$-action as in Example~\ref{example_G-action-on-X}. So $\alpha$-twisted $G$-equivariant coherent sheaves  are defined, and we denote the category of these by $\Coh^{G,\alpha} (X)$.

\item By taking $X=\Spec\kk$ above, one defines $\alpha$-twisted representations of $G$ over $\kk$, and we will denote the category of these by $\Rep(G,\alpha)$.
\end{enumerate}
\end{definition}

One can see that the category $\Coh^{G,\alpha} X$ depends, up to an equivalence, only on the cohomology class~$[\alpha]$.
The following properties follow from the above definition, see also~\cite[Prop. 1.2]{Elagin}.
\begin{lemma}
\label{lemma_alphaalpha}
Under the assumptions of Definition~\ref{def_twistedequivariant}(1), let $C\in \bC^{G,\alpha_1}$ and $V\in \Rep(G,\alpha_2)$. Then 
\begin{enumerate}
    \item $C\otimes_{\kk} V\in \bC^{G,\alpha_1\alpha_2}$.  
\end{enumerate}

Under the assumptions of Definition~\ref{def_twistedequivariant}(2), let $F_1\in\Coh^{G,\alpha_1}(X),F_2\in\Coh^{G,\alpha_2}(X)$. Then 
\begin{enumerate}[resume]
    \item $F_1\otimes F_2\in\Coh^{G,\alpha_1\alpha_2}(X)$ (this is a non-derived tensor product),
    \item $F_1^{\vee}\in\Coh^{G,\alpha_1^{-1}}(X)$.
\end{enumerate}

Finally, let $X_1,X_2$ be as in Definition~\ref{def_twistedequivariant}(2) and $f\colon X_1\to X_2$ be a $G$-equivariant morphism. Assume  $F_1\in\Coh^{G,\alpha}(X_1),F_2\in\Coh^{G,\alpha}(X_2)$. Then 
\begin{enumerate}[resume]
    \item $f_*(F_1)\in\Coh^{G,\alpha}(X_2)$,
    \item $f^*(F_2)\in\Coh^{G,\alpha}(X_1)$.
\end{enumerate}
\end{lemma}

\begin{example}
\label{example_cocycleforPn}
Let $n\ge 1$, $V$ be an $n$-dimensional $\kk$-vector space, and $\P^{n-1}_{\kk}=\P(V)$.
Let $\Aut(\P^{n-1}_{\kk})$ be the group of all (not only $\kk$-linear) scheme automorphisms of $\P^{n-1}_{\kk}$. Let $\Aut^{\mathrm{slin}}(\AA^n_{\kk})$ be the group of semi-linear 
(i.e., preserving linear functions) scheme automorphisms of  $\AA^n_{\kk}$.
There is an exact sequence of groups
\begin{equation}
    \label{eq_universalcocycle}
    1\to \kk^*\to \Aut^{\mathrm{slin}}(\AA^n_{\kk})\to \Aut(\P^{n-1}_{\kk})\to 1,
\end{equation}
let $\xi\in Z^2(\Aut(\P^{n-1}_{\kk}), \kk^*)$ be a corresponding cocycle. There is a natural $\xi$-twisted representation of $\Aut(\P^{n-1}_{\kk})$ in $V$, the space of linear functions on $\AA^n_{\kk}$. Also one can see that for any $k\in\Z$ the line bundle $\cO(-k)$ on $\P^{n-1}_{\kk}$ is $\xi^k$-twisted $\Aut(\P^{n-1}_{\kk})$-equivariant. Indeed, for $k=-1$ one remarks that the tautological line bundle  $\cO(-1)$ is an invariant subbundle in the constant bundle $\cO\otimes V$ and uses Lemma~\ref{lemma_alphaalpha}(1), then for general $k$ this follows by Lemma~\ref{lemma_alphaalpha}(2,3).

Now let $G$ be a group acting on $\P^{n-1}_{\kk}$. Restricting $\xi$ to $G$ along the homomorphism $G\to \Aut(\P^{n-1}_{\kk})$, we get a cocycle $\alpha\in Z^2(G,\kk^*)$ (note that $\kk^*$ carries a $G$-action given by the $G$-action  on $\rH^0(\P^{n-1}_\kk,\cO)=\kk$). The cocycle $\alpha$  corresponds to the central extension
\begin{equation}
    \label{eq_cocycleforPn}
    1\to \kk^*\to \Tilde{G}\to G\to 1
\end{equation}
induced from~\eqref{eq_universalcocycle}. As above, the line bundle $\cO(-k)$ on $\P^{n-1}_{\kk}$ is $\alpha^k$-twisted $G$-equivariant for any $k\in\Z$.
The cohomology class $[\alpha]$ vanishes if and only if~\eqref{eq_cocycleforPn} splits (hence the action of $G$ on $\P^{n-1}$ factors through $\AA^n$) and if and only if the line bundles $\cO(k)$ are all (untwisted) $G$-equivariant. 

We finally remark that $[\xi^n]=[\alpha^n]=1$, for example, because the line bundle $\cO_{\P^{n-1}}(-n)\cong \omega_{\P^{n-1}}$ is always equivariant.
\end{example}

Now let us consider an  example more general than Example~\ref{example_kpoint}. 
\begin{example}
Let $A$ be a commutative $\kk$-algebra, and $A^*$ be the group of its invertible elements. Let
$\mathsf{pt}_A$ be the category with one object $\bullet$ such that $\End(\bullet)=A$. 
As in Example~\ref{example_kpoint}, a $G$-action on $\mathsf{pt}_A$ is the same as a homomorphism $\rho$ from $G$ to $\Aut(A)$ (the group of ring automorphisms of $A$) and a cocycle $\alpha=(\alpha_{g,h})\in Z^2(G,A^*_{\rho})$. Indeed, let us check that the map $g\mapsto \rho_g\in \Aut (A)$, provided by a $G$-action on $\mathsf{pt}_A$, is a homomorphism. By definition~\ref{def_groupaction}, for $g,h\in G$ and $a\in \End(\bullet)$ one has $\alpha_{g,h}\cdot \rho_g(\rho_h(a))=\rho_{gh}(a)\cdot \alpha_{g,h}\in A$, and by commutativity of $A$ one gets 
$\rho_g(\rho_h(a))=\rho_{gh}(a)$. 
\end{example}

\begin{lemma}
\label{lemma_actiononapoint2}
Two actions of $G$ on $\mathsf{pt}_A$ given by  $(\rho,\alpha)$ and $(\rho',\alpha')$ are $G$-equivariantly equivalent if and only if for some $\tau\in\Aut(A)$ one has $\rho'=\tau\rho\tau^{-1}$ and $[\alpha']=\tau([\alpha])\in \rH^2(G,A^*_{\rho'})$ (where $\tau\colon A^*_{\rho}\to A^*_{\rho'}$ is an isomorphism of $G$-modules). 

In particular, two such actions  are $G$-equivariantly \textbf{$\kk$-linearly} equivalent if and only if 
the above is satisfied for some \textbf{$\kk$-linear} automorphism $\tau\in \Aut_{\kk}(A)$.
\end{lemma}
\begin{proof}
    The same as for Lemma~\ref{lemma_actiononapoint}.
\end{proof}

We will be interested in the special case when $A=\kk^n=\kk\times\ldots\times \kk$. We fix a homomorphism $\sigma \colon G\to \Aut (\kk)$ and consider only $\sigma$-linear $G$-actions on $\mathsf{pt}_A$. 
\begin{lemma}
\label{lemma_actiononapoint3}
Let $A=\kk^n$ with idempotents $e_1,\ldots,e_n$.
A $\sigma$-linear action of $G$ on $\mathsf{pt}_A$ is given by 
a homomorphism $\rho\colon G\to S_n$ to the symmetric group and a $2$-cocycle of $G$ in $(\kk^*)_{\rho,\sigma}^n$, where the
$G$-action on $(\kk^*)_{\rho,\sigma}^n$ is the following: $g(\lambda \cdot e_i)=\sigma(g)(\lambda)\cdot e_{\rho(i)}$. 
Two such actions given by 
$(\rho,\alpha)$ and $(\rho',\alpha')$ are $G$-equivariantly $\kk$-linearly equivalent if and only if for some $\tau\in S_n$ one has $\rho'=\tau\rho\tau^{-1}$ and $[\alpha']=\tau([\alpha])\in \rH^2(G,(\kk^*)^n_{\rho',\sigma})$.
\end{lemma}
\begin{proof}
Follows from Lemma~\ref{lemma_actiononapoint2}. Indeed, a $\sigma(g)$-linear automorphism $\rho_g$ of $\kk^n$ is uniquely determined by its action on the idempotents $e_1,\ldots,e_n$, hence an element in $S_n$ that we will also denote $\rho_g$. One gets a homomorphism $\rho\colon G\to S_n$ as the composition $G\to \Aut (\kk^n)\to S_n$. For the second statement note that a $\kk$-linear automorphism of $\kk^n$ is given by a permutation $\tau\in S_n$.
\end{proof}

The following definition generalises Definition~\ref{def_obstruction}.
\begin{definition}
\label{def_obstruction-for-blocks}
Suppose that $\bC$ is a $\kk$-linear category with a $\sigma$-linear action of a group $G$. Assume there are objects $E_1,\ldots,E_n\in\bC$ such that 
\begin{equation}
\label{eq_EorthoE}
    \Hom(E_i,E_i)=\kk, \quad \Hom(E_i,E_j)=0\quad\text{for $i\ne j$},
\end{equation}
and $G$ permutes $E_1,\ldots,E_n$: for all $g\in G$ and $i$ one has $\rho_g(E_i)\cong E_j$ for some $j$.
Take $E=E_1\oplus \ldots\oplus E_n$, then $E$ is $G$-invariant and $\End(E)=\kk^n$. Arguing as in Definition~\ref{def_obstruction}, we get a $G$-module structure $(\kk^*)^n_{\rho,\sigma}$ on $(\kk^*)^n=\Aut(E)$ (where $\rho\colon G\to S_n$ is the $G$-action on objects $E_1,\ldots,E_n$) and a \emph{twisting cohomology class} $[\alpha]\in \rH^2(G,(\kk^*)^n_{\rho,\sigma})$, associated to  
$\{E_1,\ldots,E_n\}$.
\end{definition}

Below we explain that the class $[\alpha]$ determines the action on the subcategory in $\bC$ generated by  $E_1,\ldots,E_n$ uniquely up to an equivalence. 
For our convenience, we restrict to the more special case of triangulated categories.

\medskip
Let $\bT$ be a triangulated $\kk$-linear category with a $\sigma$-linear action of a group $G$ (for some homomorphism $\sigma\colon G\to\Aut(\kk)$). We will use the following terminology.
\begin{definition}\label{def_Gblock}
A weakly exceptional block of objects $E_1,\ldots,E_n$ in $\bT$ is called a \emph{$G$-block} if $G$ permutes $E_1,\ldots,E_n$: for any $g\in G$ and $i$ one has $\rho_g(E_i)\cong E_j$ for some $j$. A $G$-block is \emph{transitive} if the action of $G$ on $E_1,\ldots,E_n$ is transitive: for any $i,j$ there is $g\in G$ such that $\rho_g(E_i)\cong E_j$.
\end{definition}

Assume $\bT$ is generated by a $G$-transitive exceptional block $E_1,\ldots,E_n$. A twisting cohomology class $[\alpha]\in \rH^2(G,(\kk^*)^n_{\rho,\sigma})$ is introduced in Definition~\ref{def_obstruction-for-blocks}, where the homomorphism $\rho\colon G\to S_n$ gives the $G$-action on the objects $E_1,\ldots,E_n$.

\begin{proposition}
\label{prop_Z}
A triangulated $\kk$-linear category with a $\sigma$-linear $G$-action and generated by an exceptional  $G$-block is determined by a transitive $G$-action on $\{1,\ldots,n\}$ given by a homomorphism $\rho\colon G\to S_n$ and a cohomology class $[\alpha]\in \rH^2(G,(\kk^*)_{\rho,\sigma}^{n})$. More precisely, let $\bT_1, \bT_2$ be two triangulated $\kk$-linear categories with $\sigma$-linear $G$-actions and generated by exceptional $G$-blocks. For $i=1,2$, let $\rho_i \colon G\to S_{n_i}$ be the corresponding homomorphisms and  $[\alpha_i]\in \rH^2(G,(\kk^*)_{\rho_i,\sigma}^{n_i})$ be the corresponding cohomology classes. Then $\bT_1$ is $\kk$-linearly $G$-equivariantly exact equivalent to $\bT_2$ if and only if 
\begin{equation}
\label{eq_alphaalpha}
   \text{$n_1=n_2, \quad$ for some $\tau\in S_{n_1}$ one has 
   $\rho_2=\tau\rho_1\tau^{-1},\quad$ and $[\alpha_2]=\tau([\alpha_1])\in \rH^2(G,(\kk^*)^n_{\rho_2,\sigma})$.}
\end{equation}
\end{proposition}
\begin{proof}
First assume that $\phi\colon \bT_1\to \bT_2$ is a $\kk$-linear $G$-equivariant exact equivalence. 
For $i=1,2$ let 
$E_i$ be the direct sum of objects in the block, 
$\bE_i\subset \bT_i$ be the full subcategory with one object $E_i$, and
$\Tilde{\bE}_i\subset \bT_i$ be the full subcategory with objects isomorphic to $E_i$. 

Changing $\phi$ by a shift if necessary one can assume that $\phi(E_1)$ is in $\Tilde{\bE}_2$. It follows then that $\phi$ restricts to a $\kk$-linear $G$-equivariant  equivalence between $\Tilde{\bE}_1$ and $\Tilde{\bE}_2$. Then the induced $G$-actions on $\bE_1$ and $\bE_2$ are
$\kk$-linearly equivalent. By Lemma~\ref{lemma_actiononapoint3}, \eqref{eq_alphaalpha} follows.

Now assume that \eqref{eq_alphaalpha} holds. By Lemma~\ref{lemma_actiononapoint3} it follows that $\bE_1$ and $\bE_2$ are $G$-equivariantly $\kk$-linearly equivalent,  let $\phi\colon  \bE_1\to \bE_2$ be such equivalence.
For $i=1,2$ let 
$\hat{\bE}_i\subset \bT_i$ be the full subcategory with objects isomorphic to finite direct sums of shifts of $E_i$.

For any $F\in\hat{\bE}_i$ fix an isomorphism $F\cong \oplus_j E_1[d_j]$ and define
$\phi(F)$ to be $\oplus_j \phi(E_1)[d_i]$.
We leave it to the reader to define $\phi$ on morphisms and to deduce an equivariant equivalence 
$\hat{\bE}_1\to \hat{\bE}_2$. Now remark that~$\bT_i$ is an idempotent closure of $\hat{\bE}_i$ and extend $\phi$ to an equivariant equivalence between $\bT_1$ and~$\bT_2$.
\end{proof}

There is a different way to introduce a twisting cohomology class associated to a $G$-block. 
As before, assume that objects $E_1,\ldots,E_n$ of a $\kk$-linear category $\bC$ form one orbit under $\sigma$-linear $G$-action on $\bC$ and satisfy~\eqref{eq_EorthoE}. One can pick an object, e.g. $E_1$, and consider its stabilising subgroup $H\subset G$, that is, $H=\{h\in G\mid \rho_h(E_1)\cong E_1\}$. Then $E_1$ is $H$-invariant, and there is the associated cohomology class   
\begin{equation*}
    [\alpha_H]\in \rH^2(H,\kk^*_{\sigma}).
\end{equation*}
This class $[\alpha_H]$ determines the action on the subcategory generated by $E_1,\ldots,E_n$ up to an equivariant equivalence. Indeed, the $G$-module $(\kk^*)^n_{\rho,\sigma}$ is induced from the $H$-module $\kk^*_{\sigma}$, and by Shapiro's Lemma one has
$$\rH^2(G,(\kk^*)^n_{\rho,\sigma})\cong \rH^2(H,\kk^*_{\sigma}).$$
Therefore $[\alpha_H]$ determines $[\alpha]$ uniquely, and the latter determines the action. We obtain the following

\begin{corollary}
\label{cor_Z}
A triangulated $\kk$-linear category with a $\sigma$-linear $G$-action and generated by a transitive exceptional $G$-block is determined by a subgroup of finite index $H\subset G$ and a cohomology class $[\alpha_H]\in \rH^2(H,\kk^*_{\sigma})$. In particular, the twisting class $[\alpha]\in \rH^2(G,(\kk^*)^n_{\rho,\sigma})$ associated with the block is trivial if and only if the twisting class $[\alpha_H]\in \rH^2(H,\kk^*_{\sigma})$ associated with one object is trivial. 
\end{corollary}

\subsection{Galois descent for semi-orthogonal decompositions}
\label{section_descent_for_SODs}
Below we recall how to do Galois descent for semi-orthogonal decompositions, specifically for exceptional blocks. Such results are not new, see e.g.~\cite{Elagin_2012},~\cite{AuelBernardara},~\cite{BDM} but we find it useful to give definitions and some proofs for three reasons: for completeness, to take care of the case of infinite field extension, and to provide a treatment as elementary as possible.

Let $\kk$ be a  field, and let $\LL/\kk$ be a Galois extension with
Galois group $G=\Gal(\LL/\kk)$. Let $X$ be a variety over $\kk$ and $X_{\LL}=X\times_{\kk}\Spec\LL$ be the scalar extension. Denote by $p\colon X_{\LL}\to X$ the natural morphism. For a thick subcategory $\bA\subset \Db(X)$ define $p^*\bA\subset \Db(X_{\LL})$  as the thick subcategory generated by the objects of the form $p^*A, A\in \bA$.
For a thick subcategory $\bB\subset \Db(X_{\LL})$ define $p_*\bB\subset \Db(X)$  as the full subcategory whose objects are such $A\in \Db(X)$ that $p^*A\in\bB$.
\begin{lemma}
\label{lemma_subcatdescent}
In the above assumptions, the correspondence $\bA\mapsto p^*\bA$, $\bB\mapsto p_*\bB$ gives a bijection between thick  subcategories in $\Db(X)$ and $G$-invariant thick subcategories in $\Db(X_{\LL})$.
\end{lemma}
\begin{proof}
First assume that $\LL/\kk$ is finite. Note two functorial isomorphisms:
\begin{align}
\label{eq_pp1}
    p_*p^*(-)&\cong (-)\otimes_\kk\LL\cong (-)^{\oplus |G|},\\
\label{eq_pp2}
    p^*p_*(-)&\cong \oplus_{g\in G} \rho_g(-),
\end{align}
where the first one is the projection formula and the second one is~\cite[Lemma 2.13]{BDM}.
We define, for a $G$-invariant thick subcategory $\bB\subset\Db(X_{\LL})$, the subcategory $\tp_*\bB\subset \Db(X)$ as  the thick subcategory generated by the object $p_*B$, $B\in\bB$ (this does not work for infinite field extensions because $p_*$ does not preserve coherent sheaves).

We claim that $\tp_*\bB=p_*\bB$. For $\tp_*\bB\subset p_*\bB$: note that any generator $p_*B$ of $\tp_*\bB$ is in $p^*\bB$ because $p^*p_*B=\oplus_g\rho_g(B)\in\bB$ by~\eqref{eq_pp2} (as $\bB$ is $G$-invariant). For $p_*\bB\subset \tp_*\bB$: assume $A\in p_*\bB$, then $p^*A\in \bB$ and $p_*p^*A\in \tp_*\bB$, hence $A\in\tp_*\bB$ as a direct summand in $p_*p^*A$ by~\eqref{eq_pp1}.

We check that $\bA=p_*p^*\bA$. Inclusion $\bA\subset p_*p^*\bA$ is evident. For inclusion $p_*p^*\bA\subset \bA$: $p_*p^*\bA=\tp_*p^*\bA$ is generated by objects $p_*p^*A\cong A\otimes L$ which belong to $\bA$.

Next we check $\bB=p^*p_*\bB$. For $\bB\subset p^*p_*\bB$: note that any $B$ is a direct summand in $\oplus_g \rho_g(B)\cong p^*p_*B$ which is in $p^*\tp_*\bB=p^*p_*\bB$. Inclusion $p^*p_*\bB\subset \bB$ is obvious.

Now we consider the general case and do not assume that $\LL/\kk$ is finite.  
Not surprisingly, everything will follow from the fact that any object/morphism in $\Db(X_{\LL})$ is defined over some subfield $\FF\subset \LL$ such that $\FF/\kk$ is Galois and finite.
Inclusions $\bA\subset p_*p^*\bA$ and $p^*p_*\bB\subset \bB$ follow from definitions as before. For $p_*p^*\bA\subset \bA$: assume $A\in p_*p^*\bA$, then $p^*A\in p^*\bA$. By definition, this means that there exists a generation diagram for $p^*A$:
\begin{equation}
    \label{eq_generation}
    0=B_0\xrightarrow{b_1} B_1\xrightarrow{b_2} B_2\to\ldots\xrightarrow{b_n} B_n\begin{smallmatrix}
        \xrightarrow{c} \\ \xleftarrow{d} 
    \end{smallmatrix}   
    p^*A,
\end{equation}
where $Cone(b_i)=p^*A_i$ for $A_i\in\bA$ and $cd=\id_{p^*A}$. 
One can find a finite Galois extension $\FF/\kk$ such that~\eqref{eq_generation} is defined over $\FF$: that is,~\eqref{eq_generation} is isomorphic to the pullback of the similar generating diagram for $p_{\FF}^*A$ in $\Db(X_{\FF})$ (where $p_{\FF}\colon X_{\FF}\to X$ is the natural morphism). Then $p_{\FF}^*A$ is in $p_{\FF}^*\bA$ and therefore $A$ is in $(p_{\FF})_*p_{\FF}^*\bA$, which is $\bA$ by the above considerations of the finite case.

Now we check that $\bB\subset p^*p_*\bB$. Let $B\in \bB$. One can find a finite Galois extension $\FF/\kk$ such that~$B$ is defined over $\FF$. Let $p_{\FF}\colon X_{\FF}\to X$ and $q\colon X_{\LL}\to X_{\FF}$ be the natural morphisms, then $B=q^*B'$ for some $B'\in\Db(X_{\FF})$.  Let $B'_1,\ldots, B'_n$ be the $\Gal(\FF/\kk)$-orbit of $B'$, then $q^*B'_1,\ldots,q^*B'_n$ is the $G$-orbit of $B$, hence $q^*B'_i\in\bB$ for all $i$. One has 
$$p^* (p_{\FF})_*B'\cong q^*p_{\FF}^*(p_{\FF})_*B'\cong \oplus_i q^*B'_i\in\bB$$
by~\eqref{eq_pp2}. It follows that  $(p_{\FF})_*B'\in p_*\bB$ (by definition) and 
$\oplus_i q^*B'_i\cong p^* (p_{\FF})_*B'\in p^*p_*\bB$. Therefore $B\cong q^*B'\in p^*p_*\bB$ as a direct summand, and we are done.
\end{proof}

\begin{remark}
One can prove that the subcategory $p^*\bA$ is in fact the closure under direct summands of the subcategory with objects $p^*A$, $A\in\bA$ (no adding of cones is needed), see the proof of Proposition~\ref{prop_descent} below.
\end{remark}
 
\begin{proposition}
\label{prop_SODdescent}
In the above assumptions there is a bijection between SOD-s of $\Db(X)$ and $G$-invariant SOD-s of $\Db(X_{\LL})$ which takes $\Db(X)=\langle\bA_1,\ldots,\bA_n\rangle$ to 
$\Db(X_{\LL})=\langle p^*\bA_1,\ldots,p^*\bA_n\rangle$, and $\Db(X_{\LL})=\langle\bB_1,\ldots,\bB_n\rangle$ to 
$\Db(X)=\langle p_*\bB_1,\ldots,p_*\bB_n\rangle$.
\end{proposition}
\begin{proof}
    Follows from Lemma~\ref{lemma_subcatdescent} and two easy observations: two objects $F_1,F_2$ in $\Db(X)$ are semi-orthogonal if and only if $p^*F_1, p^*F_2$ are semi-orthogonal in $\Db(X_{\LL})$; an object $F$ generates $\Db(X)$ if and only if $p^*F$ generates $\Db(X_{\LL})$. 
\end{proof}

One would like to relate subcategories of $\Db(X)$ and $G$-invariant subcategories of $\Db(X_{\LL})$ as abstract triangulated categories.
We restrict ourselves to the case of finite Galois extensions.

Recall that, for a category $\bB$ with a $G$-action we denote by $\bB^G$ the category of $G$-equivariant objects in $\bB$, see Definition~\ref{def_equivariant}.

For a $\kk$-linear idempotent complete category $\bA$ and a field extension $\kk\subset \LL$, one can define the scalar extension $\bA_{\LL}$ as the idempotent completion of the following category $\Tilde{\bA}_{\LL}$:  objects are the same as in $\bA$, and $\Hom$-spaces are $\Hom_{\bA}(-,-)\otimes_\kk\LL$, see e.g.~\cite[\S 5.1]{Laugwitz}. Note that $\Tilde{\bA}_{\LL}$ and ${\bA}_{\LL}$ carry natural $\Gal(\LL/\kk)$-actions.

\begin{proposition}[See~{\cite[\S 8.3]{Laugwitz}}]
\label{prop_descent}
Let $\kk\subset \LL$ be a finite Galois extension, $G=\Gal(\LL/\kk)$, $X$ be a variety over $\kk$ and $X_{\LL}=X\times_\kk \Spec \LL\xra{p} X$ be a scalar extension. Let $\bA\subset \Db(X)$ be a thick subcategory and $\bB=p^*\bA\subset \Db(X_{\LL})$ be the corresponding $G$-invariant thick subcategory. Then $\bA$ is equivalent to $\bB^G$, and $\bB$ is equivalent to $\bA_{\LL}$.
\end{proposition}
\begin{proof}
For the first statement  we recall that $\Db(X)$ is equivalent to $\Db(X_{\LL})^G$
(e.g. by~\cite[Th. 7.3]{Elagin_2011} applied to the morphism $p\colon X_{\LL}\to X$).
Consider the commutative diagram, where $u$ is the forgetful functor:
$$\xymatrix{&\Db(X_{\LL})&\\ \Db(X)\ar[ru]^{p^*}\ar[rr]^{\sim}_{\phi} && \Db(X_{\LL})^G.\ar[lu]_u}$$
Now, the equivalence $\bA\cong \bB^G$ follows from diagram chase:
$\bA=(p^*)^{-1}(\bB)=\phi^{-1}u^{-1}(\bB)=\phi^{-1}(\bB^G)$.

For the second statement, note that $\Tilde{\bA}_{\LL}$ is equivalent to the full subcategory  $\tp^*\bA\subset \Db(X_{\LL})$ with objects $p^*A$, $A\in \bA$. So we are to check that $p^*\bA$ is obtained from $\tp^*\bA$ by adding direct summands only. Let $\overline{\tp^*\bA}$ be the closure of $\tp^*\bA$ under direct summands. We need to check that $\overline{\tp^*\bA}$ is closed under  taking shifts (obvious) and cones. 

Let $B_1,B_2\in \overline{\tp^*\bA}$ and $f\colon B_1\to B_2$. Then there are $B'_1,B'_2\in \Db(X_{\LL})$, $A_1,A_2\in \bA$ such that $p^*A_i\cong B_i\oplus B'_i$ for $i=1,2$. Let $\of=(\begin{smallmatrix}
    f & 0 \\ 0 & 0
\end{smallmatrix})\colon p^*A_1\to p^*A_2$ be the morphism, then $Cone(f)$ is a direct summand in $Cone (\of)$. It follows from~\eqref{eq_pp2} that $\of$ is a direct summand in $p^*p_*\of$ and hence $Cone(\of)$ is a direct summand in $Cone (p^*p_*\of)\cong p^*Cone(p_*\of)$. Observe that $p_*\of$ is a morphism $p_*p^*A_1\to p_*p^*A_2$ and by~\eqref{eq_pp1} we have $p_*p^*A_i\in\bA$ for $i=1,2$, so that $Cone(p_*\of)\in \bA$ as well. Hence $Cone (p^*p_*\of)\in \tp^*\bA$  and therefore $Cone(f)\in \overline{\tp^*\bA}$. We conclude that $\overline{\tp^*\bA}$ is thick and so $\overline{\tp^*\bA}=p^*\bA$ as needed.
\end{proof}

Next we specialise to the case of subcategories generated by exceptional blocks, and discuss descent for these in more details. Good references on this subject include~\cite[Section 2]{AuelBernardara} and~\cite[Section 2.2]{BDM}. 

Recall Definitions~\ref{def_exceptionalobject} and~\ref{def_exceptionalcollection} of $G$-blocks and $G$-transitive blocks of (weakly) exceptional objects.
\begin{proposition}
\label{prop_descentforblocks}
    Take the assumptions of Lemma~\ref{lemma_subcatdescent} (field extension $\LL/\kk$ can be infinite).
\begin{enumerate}
    \item $\bA$ is generated by a block of weakly exceptional objects if and only if $\bB$ is generated by a $G$-block of weakly exceptional objects. The block generating $\bA$ has one object if and only if the block generating $\bB$ is $G$-transitive.
\end{enumerate}
Assume additionally that a group $H$ acts on $X$ by $\kk$-linear automorphisms, so that $H\times G$ acts on $X_{\bar\kk}$. 
\begin{enumerate}[start=2]
    \item $\bA$ is generated by an $H$-block of weakly exceptional objects if and only if $\bB$ is generated by an $H\times G$-block of weakly exceptional objects. The block generating $\bA$ is $H$-transitive if and only if the block generating $\bB$ is $H\times G$-transitive.
\end{enumerate}
\end{proposition}
\begin{proof}
For the purpose of the proof, let us call an object $E$ of $\kk$-linear triangulated category \emph{semi-exceptional} if $\Hom^i(E,E)=0$ for $i\ne 0$ and $\End(E)$ is a semi-simple finite-dimensional $\kk$-algebra. One can see that (provided that the category is idempotent complete, which is always true in our examples) an object is semi-exceptional if and only if it has the form $E_1^{d_1}\oplus\ldots\oplus E_n^{d_n}$ for some block $E_1,\ldots, E_n$ of weakly exceptional objects.

With this terminology, the first statement in (1) is equivalent to the following: an object $E$ in $\Db(X)$ is semi-exceptional if and only if $p^*E\in\Db(X_{\LL})$ is semi-exceptional. This, in turn, follows from the fact that a $\kk$-algebra $A$ is semi-simple if and only if $A\otimes_{\kk}\LL$ is semi-simple. We leave the details to the reader. Then the second statement in (1) is due to the fact that $\bA$ has non-trivial thick subcategories  if and only if $\bB$ has non-trivial thick $G$-invariant subcategories, see Lemma~\ref{lemma_subcatdescent}.

Item (2) follows easily: $\bA$ is $H$-invariant if and only if $\bB$ is $H$-invariant for the first  part, and the second part is proved as in part (1).
\end{proof}

\begin{lemma}
    \label{lemma_alphaD}
    Take the assumptions as in Lemma~\ref{lemma_subcatdescent} and suppose that $\LL/\kk$ is finite.
    Suppose $\bA$ is generated by a $D$-exceptional object $E$ for a division $\kk$-algebra $D$, and 
    $p^*E\cong F^d$ for some $\LL$-exceptional object $F\in\bB$ with $d\ge 1$ (so that $F$ is $G$-invariant and generates $\bB$).
    Let $[\alpha]\in \rH^2(G, \LL^*)$ be the class associated with $F$ (by Definition~\ref{def_obstruction}). 
    Then  $[\alpha]$ corresponds to the class $[D]$ of the central simple algebra~$D$.
\end{lemma}
\begin{proof}
    By Proposition~\ref{prop_descent} we have that $\langle F\rangle^{G}\cong \langle E\rangle$.  Consider the composition of equivalences 
    $$\Db(\Rep_{\LL}(G,\alpha^{-1}))\xrightarrow{\sim} \langle F\rangle^{G}\xrightarrow{\sim} \langle E\rangle\xrightarrow{\sim} \Db(\modd D)\xrightarrow{\sim} \Db(D^{\mathrm{op}}\mmod),$$
    where the first functor is given by $V\mapsto F\otimes V$, see Lemma~\ref{lemma_alphaalpha} and~\cite[Lemma 2.10]{Elagin}.
    Abelian category $\Rep_{\LL}(G,\alpha^{-1})$ of twisted representations is 
    equivalent to the module category of the skew product $ \LL\# _{\alpha^{-1}} G$ defined as follows. 
    As a vector space $\LL\# _{\alpha^{-1}} G=\LL\otimes_{\kk} \kk[G]$ and the multiplication is given by
    $$(\lambda\otimes g)(\mu\otimes h)=(\lambda\cdot g(\mu)\cdot \alpha^{-1}_{g,h}\otimes gh).$$
    Then algebra $\LL\# _{\alpha^{-1}} G$ is central over $\kk$ and simple, and its Brauer class corresponds to $[\alpha^{-1}]$ under the isomorphism $\Br(\LL/\kk)\cong \rH^2(G,\LL^*)$ (see~\cite[Ch. 8.4]{Jacobson_1989_BasicAlgebraII}). By definitions, left $\LL\# _{\alpha^{-1}} G$-modules are $\alpha^{-1}$-twisted representations of $G$ over $\LL$. Hence we get an equivalence
    $$\Db((\LL\# _{\alpha^{-1}} G)\mmod)\cong \Db(D^{\mathrm{op}}\mmod),$$
    and therefore $\LL\# _{\alpha^{-1}} G$ is Morita-equivalent to $D^{\mathrm{op}}$. Consequently, $[\alpha^{-1}]=[D^{\mathrm{op}}]$, and $[\alpha]=[D]$ as needed.
\end{proof}

Assume now that the field $\kk$ is perfect.
Triangulated $\kk$-linear categories generated by a weakly exceptional object $E$ are determined by the division algebra $\End(E)$. According to Corollary~\ref{cor_Z}, triangulated $\bar\kk$-linear categories with twisted $G=\Gal(\bar\kk/\kk)$-action and generated by a $G$-transitive  exceptional block  are classified by a subgroup $H\subset G$ of finite index and a cohomology class in $\rH^2(H,\bar\kk^*)$. In view of Proposition~\ref{prop_descentforblocks}(2), these two descriptions are related:
\begin{proposition}
\label{prop_alphaD}
For a perfect field $\kk$ and a variety $X$ over $\kk$, let $G=\Gal(\bar\kk/\kk)$, and $p\colon X_{\bar\kk}\to X$ be the natural morphism. Let $E\in\Db(X)$ be a $D$-exceptional object, where $D$ is a division $\kk$-algebra.
Let $\FF\subset D$ be the centre of $D$, $n=[\FF:\kk]$, and $d^2=\dim_{\FF}D$. Then  
$$p^*E\cong E_1^{d}\oplus\ldots\oplus E_n^{d},$$
where $(E_1,\ldots,E_n)$   is a $G$-transitive exceptional block in $\Db(X_{\bar\kk})$.
Moreover, let $H\subset G$ be the stabiliser of $E_1$ and $[\alpha]\in \rH^2(H,\bar\kk^*)$ be the class associated with the $H$-action on $E_1$ (by Definition~\ref{def_obstruction}). Then $\FF\cong \bar\kk^H$ and $[\alpha]$ corresponds to the class in $\Br(\FF)$ of the central simple $\FF$-algebra~$D$.
\end{proposition}
\begin{proof}
    By the proof of Proposition~\ref{prop_descentforblocks}, $p^*E$ is a semi-exceptional object: 
    $p^*E\cong E_1^{d_1}\oplus\ldots\oplus E_{n'}^{d_{n'}}$ for some $G$-transitive exceptional block $E_1, \ldots, E_{n'}$ on $X_{\bar\kk}$.
    One has $d_1=\ldots=d_{n'}=:d'$ because $p^*E$ is $G$-invariant. Computing the centre in 
    $$D\otimes_{\kk}\bar\kk\cong \End(p^*E)\cong M_{d'}(\bar \kk)^{\times n'}$$
    one sees that $n=n'$, and by dimension count one gets $d=d'$.

    Let $\FF'=(\bar\kk)^{H}$ be the invariant subfield so that $H=\Gal(\bar\kk/\FF')$ and $[\FF':\kk]=n$. Decompose
    $p$ as $X_{\bar\kk}\xrightarrow{q} X_{\FF'}\xrightarrow{p_{\FF}} X$. By Proposition~\ref{prop_descentforblocks}, the $H$-transitive block $E_1$ corresponds to a weakly exceptional object $E'$ on $X_{\FF'}$. Moreover, for some $r$ one has $q^*(E^{'r})\cong E_1^d$. 
    Next we compute
    $$D= (D\otimes_{\kk}\bar\kk)^G\cong (\End(E_1^d)\times\ldots\times \End(E_n^d))^G=\End(E_1^d)^H=(\End(E^{'r})\otimes_{\FF'}\bar\kk)^H=\End(E^{'r})\cong M_r(\End(E')).$$
    It follows that $r=1$ and $\End(E')\cong D$. Note that $\FF'\subset Z(\End(E'))$ so $\FF'$ embeds into $Z(D)=\FF$. 
    Recall that $\deg_{\kk}\FF'=\deg_{\kk}\FF=n$, hence $\FF'\cong \FF$.
    
    We see that the field extension from $\kk$ to $\FF=Z(\End(E))$ splits from $E$ a weakly exceptional direct summand~$E'$ with multiplicity $1$ and such that $\End(E')\cong \End(E)$.

    Now we explain the last statement about $[\alpha]$. One can choose a finite Galois extension $\FF'\subset \LL$ such that $E_1$ is defined over $\LL$, and let $E'_1\in \Db(X_{\LL})$ be the corresponding $\LL$-exceptional object. Let $H'=\Gal(\LL/\FF')$, then $E'_1$ is $H'$-invariant, moreover, the associated class $[\alpha']\in \rH^2(H',\LL^*)$ goes to $[\alpha]\in \rH^2(H,\bar\kk^*)$ under the epimorphism $H\to H'$.
    Now the claim follows from Lemma~\ref{lemma_alphaD} which we apply to the field extension $\FF'\subset \LL$ and the objects $E'\in \Db(X_{\FF'}),E'_1\in\Db(X_{\LL})$.
\end{proof}

\begin{example}
\label{example_SB}
Let $X$ be a Severi--Brauer variety over $\kk$, that is, a twisted form of $\P^{n-1}$ for some $n\ge 2$. Let~$A$ be the corresponding central simple algebra over $\kk$ of degree $n$ and $[\alpha]\in \rH^2(G,\bar\kk^*)$ the cohomology class corresponding to the Brauer class $[A]$. 
Then the twisting class  associated with $\cO(-1)$ (see Example~\ref{example_cocycleforPn}) coincides with $[\alpha]$, see e.g.~\cite[Cor. 4.7]{Bernardara}.
The $G$-invariant subcategory $\langle\cO(-1)\rangle\subset\Db(\P^{n-1}_{\bar\kk})$ descends to a subcategory $\bA\subset \Db(X)$, generated by a $D$-exceptional object $E$, where $D$ is a division algebra representing the class $[A]$ (see Proposition~\ref{prop_alphaD}). Note that the index $r=\ind [A]=\deg D$ divides $n=\deg A$ and that $E$ is a vector bundle of rank $r$. In the notation of Proposition~\ref{prop_alphaD} we have $\FF=\kk$, $n=1$, $d=r$, and $p^*E\cong \cO(-1)^r$.
\end{example}

\section{Atomic theories}

\subsection{Basics}

To construct an atomic theory means to specify some SOD-s (called atomic) of $\Db(X)$ for each variety $X$ in a certain class. This class should satisfy some assumptions that we discuss below.
\begin{definition}
    By an \emph{atomic domain} over a field $\kk$ we will mean a subcategory (usually not full) in the category  $\cVar_{\kk}$  of varieties over $\kk$ such that all morphisms in the subcategory are  derived contractions (see Definition~\ref{def_dc}). 

    If $G$ is a group, a \emph{$G$-atomic domain} over $\kk$ is a subcategory in the category  $G-\cVar_{\kk}$  of varieties over~$\kk$ with a $G$-action and $G$-equivariant morphisms such that all morphisms in the subcategory are derived contractions.
\end{definition}

\begin{definition}
\label{def_SAT}
Let  $\cV$ be an atomic domain over $\kk$.
An \emph{atomic theory} for $\cV$ assigns, to any variety $X$ in $\cV$,
a class of semi-orthogonal decompositions of $\Db(X)$, called \emph{atomic decompositions}, such that
\begin{enumerate}[label=({A\arabic*}),start=1]
    \item\label{ax:kA1} atomic decompositions form one orbit under mutations,
    \item\label{ax:kA2} for $X,Y$ in $\cV$ and a morphism $f\colon X\to Y$ in $\cV$, there are atomic decompositions
    $$\Db(X)=\langle \bA_1,\ldots,\bA_n,\bA_{n+1},\ldots,\bA_m\rangle,\quad \Db(Y)=\langle \bB_{n+1},\ldots,\bB_m\rangle,$$
    such that $f^*\bB_j=\bA_j$ (are equal as subcategories) for $j=n+1,\ldots,m$.
    \end{enumerate}
\end{definition}

More generally, we define the equivariant version of atomic theories. We do not assume that group actions are  $\kk$-linear.
For a group $G$ and a $G$-variety $X$, there is a natural $G$-action on the categories $\Coh(X)$ and $\Db(X)$ given by 
$\rho_g(-)=g_*(-)$
(where $\epsilon_{g,h}$ are canonical isomorphisms). We recall that a subcategory $\bT\subset \Db(X)$ is \emph{$G$-invariant} if $\rho_g(\bT)=\bT$ for all $g\in G$.
An SOD of $\Db(X)$ is \emph{$G$-invariant} if it consists of $G$-invariant subcategories.

\begin{definition}
\label{def_GSAT}
Let $G$ be a group and $\cV$ be a $G$-atomic domain over $\kk$. 
A \emph{$G$-equivariant atomic theory} for $\cV$ assigns, to any $G$-variety $X$ in $\cV$,
a class of $G$-invariant semi-orthogonal decompositions of $\Db(X)$, called \emph{atomic decompositions},  such that
\begin{enumerate}[label=({A\arabic*})]
    \item\label{ax:A1} atomic decompositions  form one orbit under mutations,
    \item\label{ax:A2} for any $X,Y$ in $\cV$ and any $G$-equivariant morphism $f\colon X\to Y$ in $\cV$, there exist atomic decompositions
    $$\Db(X)=\langle \bA_1,\ldots,\bA_n,\bA_{n+1},\ldots,\bA_m\rangle,\quad \Db(Y)=\langle \bB_{n+1},\ldots,\bB_m\rangle,$$
    such that $f^*\bB_j=\bA_j$ (are equal as subcategories) for $j=n+1,\ldots,m$.
    \end{enumerate}
\end{definition}

\begin{remark}
By Remark~\ref{rem_mutationOfGinvariantSOD}, to define $G$-equivariant atomic theory on some $X$, it suffices to specify just \emph{one} $G$-invariant SOD of $\Db(X)$ --- all other atomic SOD-s will be its mutations.
\end{remark}

\medskip

We believe that atomic theories for different groups should be compatible in a way that we explain next.
\begin{definition}
\label{def_refinement}
We say that an SOD $\bT=\langle \bB_1,\ldots, \bB_m\rangle$ is a \emph{refinement} of an SOD $\bT=\langle \bA_1,\ldots, \bA_n\rangle$ if there exist integers $0=a_0<a_1<\cdots <a_{n-1}<a_n=m$
     such that \[\bA_i=\langle \bB_{a_{i-1}+1},\ldots,\bB_{a_i}\rangle,\qquad i=1,\ldots, n,\]
    where the equality is meant as subcategories of $\bT$. 
\end{definition}

\begin{definition}\label{def_subgroupcompatible}
    Let $G$ be a group and $H\subset G$ be a subgroup. Let $\cV_G$ be a $G$-atomic domain and $\cV_H$ be an $H$-atomic domain over $\kk$.
    We say that a $G$-atomic theory for $\cV_G$ is \emph{subgroup compatible} with an $H$-atomic theory for $\cV_H$ if the following holds:
    \begin{enumerate}[label=(\alph*)]
        \item the restriction functor $G-\cVar_{\kk}\to H-\cVar_{\kk}$ sends $\cV_G$ to $\cV_H$, and
        \item for any $X\in\cV_G$, there exists a $G$-atomic decomposition $\Db(X)=\langle \bA_1,\ldots, \bA_n\rangle$, and an $H$-atomic decomposition $\Db(X)=\langle \bB_1,\ldots, \bB_m\rangle$ such that $\langle \bB_1,\ldots, \bB_m\rangle$ is a refinement of $\langle \bA_1,\ldots, \bA_n\rangle$.
    \end{enumerate}
\end{definition}

\begin{definition}\label{def_extcompatible}
    Let $H$ be a group and $\kk\subset \FF$ be a field extension. Let $\cV_{\kk}$ and $\cV_{\FF}$ be $H$-atomic domains over  $\kk$ and $\FF$ respectively, where we suppose that $H$ acts on varieties $\kk$- (resp. $\FF$-) linearly.
    We say that an $H$-atomic theory for $\cV_{\kk}$ is \emph{field extension compatible} with an $H$-atomic theory for $\cV_{\FF}$ if the following holds:
    \begin{enumerate}[label=(\alph*)]
        \item the scalar extension functor sends $\cV_{\kk}$ to $\cV_{\FF}$, and
        \item for any $X\in\cV_{\kk}$, there exist  $H$-atomic decompositions $\Db(X)=\langle \bA_1,\ldots, \bA_n\rangle$ and  $\Db(X_{\FF})=\langle \bB_1,\ldots, \bB_m\rangle$ such that $\langle\bB_1,\ldots,\bB_{m}\rangle$ refines $\langle q^*\bA_1,\ldots,q^*\bA_{n}\rangle$, where $q\colon X_{\FF}\to X$ is the natural morphism.
    \end{enumerate}
\end{definition}

\medskip

Assume that the field  $\kk$ is algebraically closed and an equivariant atomic theory for some $G$-atomic domain $\cV$ over $\kk$ is given.

\begin{definition}
For $X\in\cV$, the \emph{atoms} of $X$ are ($G$-equivalence classes of) $G$-categories $\bA_1,\ldots,\bA_n$, where  $\Db(X)=\langle\bA_1,\ldots,\bA_n\rangle$ is an atomic SOD. 
\end{definition}

By Lemma~\ref{lemma_atomswelldefined}, the atoms are  well-defined.

\begin{definition}
\label{def_permtype}
Let us say that a $G$-atom $\bA \subset \Db(X)$ is \emph{of permutation type}
if $\bA$ is generated by a $G$-transitive block
of exceptional objects.
\end{definition}

As a category, a permutation type atom $\bA$ is equivalent to a direct sum of copies of $\Db(\kk)$ hence it has a quite straightforward  structure, in particular its Serre functor is trivial so that $\Hom(F, G) \cong \Hom(G, F)^*$ for all $F, G \in \bA$. Furthermore, every thick triangulated subcategory $\bA' \subset \bA$ 
(not necessarily $G$-invariant)
is also generated by a block of exceptional objects.

By Proposition \ref{prop_Z}, a $G$-atom of permutation type is characterized by a $G$-action on $\kk$, a $G$-action on a finite set $Z$ of objects, and a \emph{twisting class} $[\alpha] \in \rH^2(G, (\kk^*)^Z)$.

\begin{definition}
\label{def_non-twisted-atom}
We say that an atom $\bA$ of permutation type is \emph{non-twisted} if the twisting class $[\alpha]$ vanishes. 
\end{definition}

\begin{lemma}
\cite[6.2, 6.3]{BCDP-finite}
\label{lem:blowup-perm}
Let $p\colon \wt{X} \to X$ be the blow-up of a finite $G$-orbit $Z$ on a smooth projective $G$-surface $X$, assume $p$ is a morphism from $\cV$. Then the $G$-atoms of $\wt{X}$ are the $G$-atoms of $X$ together with one non-twisted permutation type atom 
$G$-equivalent to $\Db(Z)$.
\end{lemma}
\begin{proof}
By the blow-up formula given in Proposition \ref{prop_blowup2}
we have
$$
\Db(\tX) = \left\langle  \left\langle\begin{smallmatrix}
\cO_{E_{1}}(-1)\\ \ldots \\ \cO_{E_{n}}(-1)
\end{smallmatrix} \right\rangle, p^*\Db(X)\right\rangle =
\left\langle p^*\Db(X),  \left\langle\begin{smallmatrix}
\cO_{E_{1}}\\ \ldots \\ \cO_{E_{n}}
\end{smallmatrix} \right\rangle\right\rangle.$$
The $G$-category $\left\langle\begin{smallmatrix}
\cO_{E_{1}}(-1)\\ \ldots \\ \cO_{E_{n}}(-1)
\end{smallmatrix} \right\rangle$ is semi-orthogonally indecomposable hence must be a $G$-atom by condition~\ref{ax:A2} from Definition~\ref{def_GSAT}, and the remaining atoms match. Use $G$-equivalences 
$\left\langle\begin{smallmatrix}
\cO_{E_{1}}(-1)\\ \ldots \\ \cO_{E_{n}}(-1)
\end{smallmatrix} \right\rangle\cong \left\langle\begin{smallmatrix}
\cO_{E_{1}}\\ \ldots \\ \cO_{E_{n}}
\end{smallmatrix} \right\rangle \cong \Db(Z)$ to see that this $G$-atom is of permutation type and non-twisted.
\end{proof}

As a warm-up, we construct atomic theory for curves.

\begin{lemma}\label{lem:AtomicDimOne}
    Let $G$ be a group, $\kk$ be algebraically closed and $\cV_G=G-\cVar_{\kk,\le 1}^{DC}$ be the class of smooth projective irreducible $G$-varieties over $\kk$ of dimension $\le 1$ and derived contractions between them. Then a $G$-equivariant atomic theory for $\cV_G$ exists and is unique.  Explicitly,  decompositions
    \begin{align*}
    \Db(\Spec \kk)&=\langle \Db(\Spec \kk)\rangle\\
    \Db(X) &=\langle \cO_X^\perp, \langle \cO_X\rangle\rangle, &&\quad\text{for a rational curve $X$,}\\ 
    \Db(X) &=\langle \Db(X)\rangle, &&\quad\text{for an irrational curve $X$.} 
    \end{align*}
    are atomic.
\end{lemma}
\begin{proof}
    Note that all derived contractions in $\cV_G$ are either isomorphisms or of the form
    $X\to \Spec \kk$, where $X$ is a rational curve.
    Relations (A2) from Definition~\ref{def_GSAT} impose that 
    $\langle\cO_X\rangle$ splits as an atom in $\Db(X)$ for $X$ a rational curve (and 
    nothing more). Also note that for a curve $X$ the category
    \begin{itemize}
        \item $\cO_X^\perp$ for  rational $X$,
        \item $\Db(X)$ for irrational $X$
    \end{itemize}
    is indecomposable (see~\cite{Okawa} for irrational curves), so further splitting is not possible.
\end{proof}

\begin{remark}\label{rem_dimOneSubgroupComp}
    It is immediate that  the $G$-atomic decompositions in Lemma~\ref{lem:AtomicDimOne} do not actually depend on $G$. Hence, if $H\subset G$ is a subgroup then $G$- and $H$-atomic theories in dimension $\le  1$ are subgroup compatible.
\end{remark}

In Sections~\ref{sec_atomictheoryMFS} and \ref{sec_construction}, we construct a $G$-atomic theory for $G$-surfaces which we believe is \emph{canonical}. Our main result can be rephrased as follows:

\begin{theorem}[{Theorem~\ref{thm_construction}, Proposition~\ref{prop_subgroupcompstd}}]
\label{thm:atoms-surfaces}
Let $\kk$ be an algebraically closed field. 
For any group~$G$ there exists a $G$-equivariant atomic theory for smooth projective irreducible $G$-varieties over $\kk$ of dimension $\le 2$ and rational derived contractions.    
Moreover, for any subgroup $H\subset G$, the constructed theories are subgroup compatible.
\end{theorem}

We sketch the construction of the canonical atomic theory in dimension $2$ that we perform in the next sections.
%used in  of Theorem~\ref{thm:atoms-surfaces}.

\begin{enumerate}
\item We use the atomic theory in dimension $\le 1$  constructed in Lemma~\ref{lem:AtomicDimOne}.
\item We define explicitly atomic SOD-s of $\Db(X)$ for any $G$-Mori fibre space $X/B$, where $X$ is a surface. We denote these SOD-s by $\cB(X/B)$. 
\item Next we define atomic SOD-s for an arbitrary surface $Y$. We choose a birational contraction $f\colon Y\to X$ to a $G$-Mori fibre space $X/B$, pull $\cB(X/B)$ back to $Y$ and augment the SOD $\Db(Y)=\langle \ker f_*, f^*\cB(X/B)\rangle$ by splitting $\ker f_*$ as follows. We  choose a decomposition of $f$ into blow-ups of $G$-orbits $Z_1,\ldots,Z_n$ and write $\ker f_*=\langle \Db(Z_1),\ldots,\Db(Z_n)\rangle$ inductively  using the blow-up formula. We denote the SOD for $\Db(Y)$ defined thus by $\cA(f\colon Y\to X;\cB(X/B))$.
\item We prove next that the atomic SOD-s $\cB(X/B)$ that we defined for Mori fibre spaces are compatible with each other under isomorphisms and under Sarkisov links.
\item Then we deduce that $\cA(f\colon Y\to X;\cB(X/B))$ is independent (up to mutation) on the choices of $f\colon Y\to X$, its decomposition into blow-ups, and the Mori fibre space structure $X/B$. It follows that the SOD-s $\cA:=\cA(f\colon Y\to X;\cB(X/B))$  form an atomic theory.
\item Finally we check subgroup compatibility.
\end{enumerate}

In Section~\ref{section_descent_for_AT} we deduce the following statement
from  Theorem~\ref{thm:atoms-surfaces} by Galois descent. 
\begin{theorem}\label{thm:atoms-surfaces2}
Let $\kk$ be a perfect field.
For any group $H$, there exists an $H$-equivariant atomic theory for smooth projective geometrically integral $H$-varieties over $\kk$ of dimension $\le 2$ \emph{with $\kk$-linear group action}.  
These theories are subgroup compatible and compatible under algebraic field extensions.
\end{theorem}

\begin{remark}
Note that the group action in Theorem~\ref{thm:atoms-surfaces2} is supposed to be $\kk$-linear while in Theorem~\ref{thm:atoms-surfaces} it is not. So neither theorem is a special case of the other.
\end{remark}

\subsection{Atomic theory for surfaces: extending theory from Mori fibre spaces}

\label{sec_atomictheoryMFS}

In this section we assume that $\kk$ is algebraically closed and fix a group $G$. Recall that by a $G$-variety we mean a smooth projective irreducible variety over $\kk$ with a $G$-action (which is not required to be $\kk$-linear). Moreover, we require that the $G$-action on $N_1(X)$ has finite orbits. By a $G$-morphism of $G$-varieties we mean a morphism over $\kk$ commuting with the group action. Further in this section we will always work in $G$-equivariant category, assuming (sometimes implicitly) that varieties are $G$-varieties and morphisms are $G$-morphisms.

We have seen that in dimension $\le 1$ there is unique atomic theory for the class of all derived contractions. However, we believe that in dimension $\le 2$ it is reasonable to restrict to the smaller atomic domain formed only by \emph{rational} derived contractions, see Definition~\ref{def_rdc} and Proposition~\ref{prop_elementary}.
That is, we consider only morphisms coming from the MMP: birational contractions and structure morphisms of Mori fibre spaces (and their compositions). 
We denote the atomic domain of $G$-varieties of dimension $\le 2$ over $\kk$  and  rational derived contractions by $G-\cVar_{\kk,\le 2}^{\MMP}$.

We will denote the category of $G$-surfaces of Kodaira dimension $\ge 0$ by 
$G-\cVar_{\kk, 2}^{\kappa\ge 0}$, while $G-\cVar_{\kk, \le 2}^{\kappa< 0}$ will denote the category of $G$-surfaces birational to a Mori fibre space and \textbf{all} $G$-varieties of dimension $\le 1$. We will denote the corresponding atomic domains with morphisms being rational derived contraction by 
$G-\cVar_{\kk, 2}^{\kappa\ge 0, \MMP}$ and $G-\cVar_{\kk, \le 2}^{\kappa<0, \MMP}$ respectively.
Note that these two subdomains are disjoint in $G-\cVar_{\kk,\le 2}^{\MMP}$, i.e. there are no rational derived contractions between them:
$$G-\cVar_{\kk, \le 2}^{\MMP}=G-\cVar_{\kk, 2}^{\kappa\ge 0, \MMP}\bigsqcup G-\cVar_{\kk, \le 2}^{\kappa<0, \MMP}.$$
Therefore, atomic theories for these two subdomains can be constructed independently.  

\bigskip
First, we observe that, given a $G$-equivariant birational contraction of surfaces $f\colon Y\to X$, there is an SOD given by the blow-up formula:
$$\Db(Y)=\langle \ker f_*, p^*\Db(X)\rangle$$
and a natural SOD of $\ker f_*$ into components corresponding to blow-ups of $G$-orbits. Here is the construction.

For the blow-up $f\colon X\to Y$ at one $G$-orbit, let $E_1,\ldots,E_d\subset Y$ be the exceptional divisors. We introduce the notation (see Proposition~\ref{blowupformula})
$$\bK_f:=\ker f_*=\langle \cO_{E_1}(-1),\ldots,\cO_{E_d}(-1)\rangle,$$
this is a $G$-invariant subcategory in $\Db(Y)$, equivalent to $\Db(\kk^{\times d})$.
Next, if $f=f_1\circ\cdots\circ f_n$ is a composition of blow-ups of $G$-orbits then one has a $G$-invariant SOD (again by Proposition~\ref{blowupformula})
\begin{equation}\label{eq_Kf}
    \ker f_*=\langle \bK_{f_n},f_{n}^*\bK_{f_{n-1}},\ldots, (f_2\ldots f_n)^*\bK_{f_1} \rangle.
\end{equation} 
Explicitly, for $i=1,\ldots,n$, let $E_{ij}$ ($j=1\ldots d_i$) be the exceptional divisors of the blow-up $f_i$. 
Then~\eqref{eq_Kf} is
\begin{equation}
\ker f_*=\left\langle 
        \left\langle \begin{smallmatrix}
        \cO_{E_{n1}}(-1)\\ \ldots \\ \cO_{E_{nd_n}}(-1)
        \end{smallmatrix} \right\rangle,\ldots,
        \left\langle \begin{smallmatrix}
        \cO_{E_{11}}(-1)\\ \ldots \\ \cO_{E_{1d_1}}(-1)
        \end{smallmatrix} \right\rangle
        \right\rangle,
\end{equation}
where we identify the sheaves and their pull-backs on $Y$ (by abuse of notation).

The following lemma says that this SOD of $\ker f_*$ does not depend on the chosen factorisation of $f$ into blow-ups, up to mutations of (mutually) orthogonal components.

\begin{lemma}
\label{lemma_ker}
    Let $f\colon Y\to X$ be a birational $G$-morphism between smooth $G$-surfaces, and let \[f_1\circ\cdots\circ f_n = f = f'_1\circ\cdots\circ f'_n\]
    be two decompositions of $f$ into a sequence of blow-ups of $G$-orbits.
    Then, the two SOD-s of $\ker f_*$
    \begin{align*}
        \langle \bK_{f_n},f_{n}^*\bK_{f_{n-1}},\ldots, (f_2\ldots f_n)^*\bK_{f_1} \rangle, \\
        \langle {\bK_{f'_n}},{f'_{n}}^*\bK_{f'_{n-1}},\ldots, (f'_2\ldots f'_n)^*\bK_{f'_1} \rangle &
    \end{align*}
    are the same up to mutations of orthogonal components.
\end{lemma}

\begin{proof}
Any two decompositions $f=f_1\circ\cdots\circ f_n = f = f'_1\circ\cdots\circ f'_n$ can be obtained from each other by the following operations: if $f_i$,$f_{i+1}$ are two blow-ups of two distinct $G$-orbits, one can blow these orbits up in different order: $\tf_{i+1}$, $\tf_i$ and replace $f_1\circ\cdots\circ f_i\circ f_{i+1}\circ \ldots\circ f_n$ with $f_1\circ\cdots\circ \tf_{i+1}\circ \tf_{i}\circ \ldots\circ f_n$. Then the corresponding kernels are orthogonal: $\bK_{f_i}=\tf_{i+1}^*\bK_{\tf_i}$ is orthogonal to $f_i^*\bK_{f_{i+1}}=\bK_{\tf_{i+1}}$, and two corresponding SOD-s of $\Db(Y)$ are obtained from each other by a mutation of two orthogonal components. Hence the claim follows by induction.
\end{proof}

Lemma~\ref{lemma_ker} motivates the following notation for SOD-s. 

\begin{definition}
\label{def_AgivenBandf}
    Let $f\colon Y\to X$ be a birational $G$-morphism between $G$-surfaces.
    We denote by 
    $$\cK(f\colon Y\to X), \quad\text{or simply by}\quad \cK(f)$$ 
    any of the $G$-invariant SOD-s of $\ker f_*$ from Lemma~\ref{lemma_ker}.
    Furthermore, if $\cB$ is a $G$-invariant SOD of $\Db(X)$ we write (where $\cE$ stands for extension)
    $$\cE(f\colon Y\to X;\cB):=\langle \cK(f),f^*\cB\rangle,$$
    which is a $G$-invariant SOD of $\Db(Y)$. 
\end{definition}

Note that $\cE(f\colon Y\to X;\cB)$ is well-defined up to mutations of  orthogonal components in $\cK(f)$. 
Also note that, for birational morphisms $Z\xra{g} Y\xra{f} X$ and an SOD $\cB$ of $\Db(X)$ 
\begin{equation}
\label{eq_Acompos}
\cE(fg\colon Z\to X,\cB)\quad\text{is mutation-equivalent to}\quad \cE(g\colon Z\to Y;\cE(f\colon Y\to X;\cB)).
\end{equation}

\bigskip

Recall from Lemma~\ref{lem:AtomicDimOne} that for smooth projective varieties of dimension $\le 1$ over an algebraically closed field $\kk$ there exists unique atomic theory, and the following are atomic decompositions:
\begin{enumerate}
    \item $\cA_{\le 1}(\Spec\kk):=\langle \Db(\Spec\kk) \rangle $,
    \item $\cA_{\le 1}(C):=\langle \Db(C) \rangle $ if $C$ is a curve of genus $g(C)\ge1$, and 
    \item $\cA_{\le 1}(\P^1):=\langle\langle \cO(-1)\rangle,\langle\cO \rangle \rangle$.
\end{enumerate}

The next theorem states that, in order to define an atomic theory for surfaces and rational derived contractions, it suffices to define atomic SOD-s only for the Mori fibre spaces and to check compatibility conditions between them. Then the theory extends uniquely to all surfaces by splitting kernels of birational contractions as in  Definition~\ref{def_AgivenBandf}. Note that this theorem applies only to the surfaces that are birational to a Mori fibre space --- construction of an atomic theory for other surfaces is much easier, see Proposition~\ref{prop_compatible_gives_atomictheory}.

\begin{theorem}
\label{thm_compatible_gives_atomictheory}
Assume that for every $G$-Mori fibre space $\pi\colon X\to B$ a $G$-invariant SOD $\cB(X/B)$ of $\Db(X)$ is chosen such that the following conditions are satisfied:
\begin{enumerate}[label=(\roman*)]
    \item\label{i_rk0} compatibility with atomic decomposition on the base:\\
    $\cB(X/B)$ is obtained from $\langle\ker \pi_*,\pi^*\cA_{\le 1}(B) \rangle$ by splitting the kernel into several components.
    \item\label{i_rk1} compatibility with isomorphisms of $G$-Mori fibre spaces:\\
    if $\phi\colon X_1\to X_2$ is an isomorphism between $G$-Mori fibre spaces $X_1/B_1$ and $X_2/B_2$ (see Definition~\ref{def_isoMFS}), then $\phi^*\cB(X_2/B_2)$ is equal to $\cB(X_1/B_1)$.    
    \item\label{i_rk2} compatibility with Sarkisov links:\\ 
    for any $G$-Sarkisov link $\chi\colon X_1\dashrightarrow X_2$ (see Definition~\ref{def:SarkisovLinks})
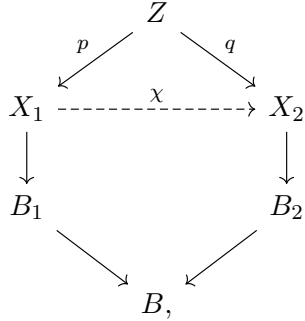
\[\begin{tikzcd}        &Z\ar[dl,swap,"p"]\ar[dr,"q"]&\\        X_1\ar[d,swap]\ar[dashed,rr,"\chi"]&& X_2\ar[d]\\ 
        B_1\ar[dr,swap]&& B_2\ar[dl,]\\
        &B,&
    \end{tikzcd}\]
    the following  SOD-s for $\Db(Z)$
    are mutation-equivalent:
    \begin{align*}
        \cE(p\colon Z\to X_1,\cB(X_1/B_1))=&\langle \cK(p), p^*\cB(X_1/B_1)\rangle, \\
        \cE(q\colon Z\to X_2,\cB(X_2/B_2))=&\langle \cK(q), q^*\cB(X_2/B_2)\rangle.
    \end{align*}   
\end{enumerate}

Then, the following hold:
\begin{enumerate}
\item\label{i_Awelldef} For any $G$-surface $Y$ and any two $G$-equivariant birational contractions $f\colon Y\to X$ and $f'\colon Y\to X'$ to $G$-Mori fibre spaces $X/B,X'/B'$
the SOD-s \[\cE(f\colon Y\to X;\cB(X/B))\quad\text{and}\quad \cE(f'\colon Y\to X';\cB(X'/B'))\] 
of $\Db(Y)$ are mutation-equivalent.
\item\label{i_atomictheory} 
For each $Y\in G-\cVar_{\kk, \le 2}^{\kappa<0, \MMP}$ define $\cA(Y)$ as follows: 
\begin{itemize}
    \item if $\dim Y\le 1$ set $\cA(Y):=\cA_{\le 1}(Y)$, 
    \item if $\dim Y=2$ set  $\cA(Y):=\cE(f\colon Y\to X;\cB(X/B))$,
where $f\colon Y\to X$ is a birational morphism onto a $G$-Mori fibre space $X/B$.
\end{itemize}
Then $Y\mapsto \cA(Y)$ defines a $G$-atomic theory on $G-\cVar_{\kk, \le 2}^{\kappa<0, \MMP}$.
\end{enumerate}
\end{theorem}
\begin{proof}
We first show \eqref{i_Awelldef}.
Consider the composition $\phi=f'f^{-1}$ as a birational map $X\dashrightarrow X'$. One of the following holds (see \cite[Th. 16.28(1), Rem. 16.73]{LamyCremona} for the rational case): 
\begin{enumerate}[label=(\alph*)]
    \item\label{i_iso} $\phi$ is an isomorphism of Mori fibre spaces, or
    \item\label{i_link}$\phi$ can be decomposed into a sequence of Sarkisov links \textbf{dominated by $Y$}: that is, there exists a commutative \emph{merry-go-round} diagram
$$
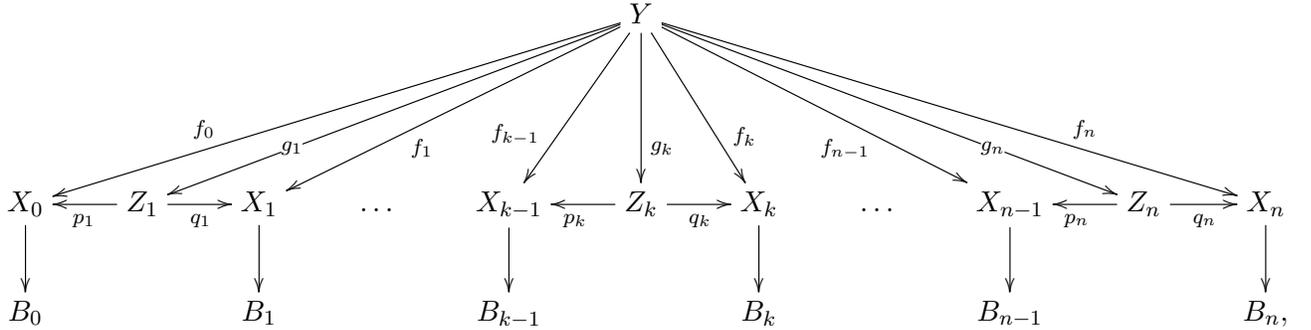
\xymatrix{ &&&&& Y\ar[llllldd]_(0.7){f_0}\ar[llldd]^(0.6){f_1}\ar[ldd]_(0.7){f_{k-1}}\ar[rdd]^(0.7){f_k} \ar[rrrdd]_(0.6){f_{n-1}}\ar[rrrrrdd]^(0.7){f_n} \ar[lllldd]|(0.7){g_1} \ar[dd]^(0.7){g_k}\ar[rrrrdd]|(0.7){g_n} &&&&&\\
\\
X_0\ar[d] & Z_1 \ar[l]^{p_1}\ar[r]_{q_1} & X_1\ar[d] &\ldots& X_{k-1}\ar[d] &Z_k \ar[l]^{p_k}\ar[r]_{q_k}& X_k\ar[d] &\ldots& X_{n-1}\ar[d] & Z_n \ar[l]^{p_n}\ar[r]_{q_n}& X_n\ar[d] \\
 B_0 && B_1   && B_{k-1} && B_k && B_{n-1} && B_n, \\}$$
where $X_{i-1}\gets Z_i\to X_i$ are Sarkisov links and $f_0=f, f_n=f'$. 
\end{enumerate}

In  case \ref{i_iso} the claim is straightforward: 
$$\cE(f'\colon Y\to X';\cB(X'/B'))\cong \cE(f\colon Y\to X;\cE(\phi\colon X\to X';\cB(X'/B')))= \cE(f\colon Y\to X;\cB(X/B)),$$
where the first mutation equivalence is by~\eqref{eq_Acompos} and the second equality is by assumption \ref{i_rk1}.

In  case \ref{i_link} one has mutation equivalences:
\begin{align*}
\cE(f_0\colon Y\to X_0,\cB(X_0/B_0))&\cong \cE(g_1\colon Y\to Z_1,\cE(p_1\colon Z_1\to X_0,\cB(X_0/B_0))) & \text{by~\eqref{eq_Acompos}}\\
& \cong \cE(g_1\colon Y\to Z_1,\cE(q_1\colon Z_1\to X_1,\cB(X_1/B_1))) & \text{by assumption \ref{i_rk2}}\\
& \cong \cE(f_1\colon Y\to X_1,\cB(X_1/B_1)) & \text{by~\eqref{eq_Acompos}}\\
& \ldots &\\
& \cong \cE(f_n\colon Y\to X_n,\cB(X_n/B_n)).
\end{align*}
This proves \eqref{i_Awelldef}.

We now prove \eqref{i_atomictheory}, that is, we check condition \ref{ax:A2}  of Definition~\ref{def_GSAT}:  for any $G$-varieties $X,Y$ in $G-\cVar_{\kk,\le 2}^{\kappa<0,\MMP}$ and any $G$-equivariant rational derived contraction $f\colon Y\to X$, there exist atomic decompositions
    $$\Db(Y)=\langle \bA_1,\ldots,\bA_n,\bA_{n+1},\ldots,\bA_m\rangle,\quad \Db(X)=\langle \bB_{n+1},\ldots,\bB_m\rangle,$$
    such that $f^*\bB_j=\bA_j$ for $j=n+1,\ldots,m$.
Any such $f$ is a composition of several $G$-MMP contractions (by Proposition~\ref{prop_elementary}), so it suffices to check condition \ref{ax:A2} only when $f$ is a $G$-MMP  contraction:
    
    If $\dim(Y)\le  1$, then it holds 
    by assumption since $\cA$ extends the atomic theory in dimension $\le  1$. 

    If $f\colon Y\to X$ is a $G$-Mori fibre space, then one can define $\cA(Y)$ as $\cB(Y/X)$ and then \ref{ax:A2} holds by assumption \ref{i_rk0}.

    If $f\colon Y\to X$ is the blow-up at one $G$-orbit, one can choose a birational contraction $f_0\colon X\to X_0/B_0$ to a Mori fibre space and define $\cA(X)$ as $\cE(f_0,\cB(X_0/B_0))$ and $\cA(Y)$ as $\cE(f_0f,\cB(X_0/B_0))$. Then \ref{ax:A2} follows from~\eqref{eq_Acompos}:
    $$\cA(Y)=\cE(f_0f,\cB(X_0/B_0))\cong \cE(f,\cE(f_0,\cB(X_0/B_0)))=\cE(f,\cA(X))=\langle \ker f_*, f^*\cA(X)\rangle.$$
\end{proof}

We finish this section with a (much easier) counterpart of Theorem~\ref{thm_compatible_gives_atomictheory} for  surfaces birational to a surface $X$ with $K_X$ nef. 
Recall that these are the surfaces of Kodaira dimension $\ge 0$ and that we  denote the atomic domain formed by $G$-surfaces of Kodaira dimension $\ge 0$ and rational derived contractions between them by $G-\cVar_{\kk, 2}^{\kappa\ge 0, \MMP}$. 
\begin{proposition}
\label{prop_compatible_gives_atomictheory}
For any $G$-surface $X$ with $K_X$ nef, we set all $\Db(X)$ to be its unique atom.
For each $Y\in G-\cVar_{\kk, 2}^{\kappa\ge 0, \MMP}$ define 
$\cA(Y):=\cE(f\colon Y\to X;\Db(X))$,
where $f\colon Y\to X$ is a birational $G$-morphism onto a $G$-surface $X$ with $K_X$ nef.
Then $Y\mapsto \cA(Y)$ defines a $G$-atomic theory on $G-\cVar_{\kk, 2}^{\kappa\ge 0, \MMP}$.
\end{proposition}
\begin{proof}
Note that if $Y$ is birational to a surface with nef canonical divisor, then there exists a unique birational contraction $f\colon Y\to X$ to a unique surface $X$ with nef canonical class \cite[Proposition III.4.6]{BPV}. 
Hence $\cA(Y)$ is well-defined up to mutations by Lemma~\ref{lemma_ker}. 

To prove that SOD-s $\cA(Y)$ form an atomic theory, we need to check condition~\ref{ax:A2} from Definition~\ref{def_GSAT}. Note that any morphism $f\colon Y_1\to Y_2$ in $G-\cVar_{\kk, 2}^{\kappa\ge 0, \MMP}$ is a composition of blow ups, therefore we can argue as in the proof of Theorem~\ref{thm_compatible_gives_atomictheory}.
\end{proof}

\subsection{Atomic theory for surfaces: defining theory for Mori fibre spaces}
\label{sec_construction}

In this section we verify that the conditions of Theorem~\ref{thm_compatible_gives_atomictheory} are satisfied and thus obtain an atomic theory for surfaces and rational derived contractions. That is, for any $G$-Mori fibre space $X/B$ of dimension $2$ we specify a $G$-invariant semi-orthogonal decomposition of $\Db(X)$ (which we call \emph{standard}) so that three conditions of Theorem~\ref{thm_compatible_gives_atomictheory} are fulfilled: the chosen SOD-s are compatible with atomic SOD-s of $B$, compatible with isomorphisms of Mori fibre spaces, and compatible with Sarkisov links (up to mutations).

We work with the same assumptions as in the previous section.

\begin{definition}[standard decompositions for Mori fibre spaces]
\label{def_standard}
Let $G$ be a group, $X$ be a $G$-surface, and  $\pi\colon X\to B$ be a $G$-Mori fibre space. We consider the following SOD $\cBstd(X/B)$ of $\Db(X)$.

If $X$ is \textbf{not $G$-birationally rich} (see Definition \ref{def:BirationallyRich}), we set \[\cBstd(X/B):=\langle \ker  \pi_*,\pi^*\cA_{\le 1}(B) \rangle,\]
    which consists of $2$ or $3$ pieces depending on $\cA_{\le 1}(B)$, namely
    \begin{align*}
        \cBstd(X/\Spec\kk):= & \langle\cO_X^{\perp}{},\langle\cO_X\rangle\rangle\quad\text{if $X$ is a $G$-minimal del Pezzo surface of degree $\le  4$,}\\
        \cBstd(X/\P^1):= & \langle\ker \pi_*, \langle \pi^*\cO_{\P^1}(-1)\rangle, \langle\cO_X\rangle \rangle\quad\text{if $X/\P^1$ is a $G$-Mori conic bundle of degree $\le  4$,}\\
        \cBstd(X/B):= & \langle \ker  \pi_*,\pi^*\Db(B) \rangle\quad\text{if $B$ is a curve of genus $g(B)\ge 1$.}
    \end{align*}
    
    If $X$ is \textbf{$G$-birationally rich}, we use the notation from Lemma~\ref{lem:3blockdecomp}. 
    If $B\cong \Spec\kk$, we set $\cBstd(X/\Spec\kk)=\cS(X)$ to be the $3$-block decomposition of 
    Karpov--Nogin, namely
    \begin{align*}
        \cBstd(\P^2/\Spec\kk):= & \left\langle 
        \langle\cO(-2H)\rangle, 
        \langle\cO(-H)\rangle, 
        \langle\cO\rangle
        \right\rangle,\\
        \cBstd(\P^1\times\P^1/\Spec\kk):= & \left\langle 
        \langle\cO(-h_1-h_2)\rangle,
        \left\langle \begin{smallmatrix}
        \cO(-h_1)\\ \cO(-h_2)
        \end{smallmatrix} \right\rangle,
        \langle\cO\rangle
        \right\rangle,\\
        \cBstd(X_6/\Spec\kk):= & \left\langle
         \left\langle \begin{smallmatrix}
        \cO(-H_1)\\ \cO(-H_2)
        \end{smallmatrix} \right\rangle,
        \left\langle \begin{smallmatrix}
        \cO(-h_1)\\ \cO(-h_2) \\ \cO(-h_3)
        \end{smallmatrix} \right\rangle,
        \langle\cO\rangle
        \right\rangle,\\
        \cBstd(X_5/\Spec\kk):= & \left\langle
         \langle \cE\rangle,
        \left\langle \begin{smallmatrix}
        \cO(-h_1)\\ \ldots \\ \cO(-h_5)
        \end{smallmatrix} \right\rangle,
        \langle\cO\rangle
        \right\rangle.
    \end{align*}
    If $B\cong\P^1$, we set 
    \begin{align*}
        \cBstd(\bF_n/\P^1):= & \left\langle
         \left\langle \cO(-s-h) \right\rangle,
        \left\langle \cO(-s) \right\rangle,
        \langle \cO(-h) \rangle, 
        \langle\cO\rangle
        \right\rangle,\\
        \cBstd(X_6/\P^1):= & \left\langle
         \left\langle \begin{smallmatrix}
        \cO(-H_1)\\ \cO(-H_2)
        \end{smallmatrix} \right\rangle,
        \left\langle \begin{smallmatrix}
        \cO(-h_1)\\ \cO(-h_2)
        \end{smallmatrix} \right\rangle,
        \langle \cO(-h_3) \rangle, 
        \langle\cO\rangle
        \right\rangle,\\
        \cBstd(X_5/\P^1):= & \left\langle
         \langle \cE\rangle,
        \left\langle \begin{smallmatrix}
        \cO(-h_1)\\ \ldots \\ \cO(-h_4)
        \end{smallmatrix} \right\rangle,
        \langle \cO(-h_5)\rangle,
        \langle\cO\rangle
        \right\rangle,
    \end{align*}
    where $h$ (respectively $h_3$, respectively $h_5$) is the $0$-class giving $\pi\colon X\to \P^1$, and $s\in\Pic(\bF_n)$ is the negative section.

Note that all these decompositions are $G$-invariant.
We will call $\cBstd(X/B)$ \emph{standard atomic SOD} of a $G$-Mori fibre space $X/B$.
\end{definition}

\begin{remark}
While the definition of $\cBstd(X/B)$ was done case by case for birationally rich surfaces, there is an alternative uniform way to define it in these cases. Unless  $X=\P^2$ or $\FF_n$, $n\ge1$, one has:
\[\cBstd(X/B)=\langle \bA,\left\langle \begin{smallmatrix}
        \cO(-h_1)\\ \ldots \\ \cO(-h_r)
        \end{smallmatrix} \right\rangle,\pi^*\cA_{\le 1}(B)\rangle,\]
        where $h_1,\ldots,h_r$ are the $0$-classes contracted to a point by $\pi\colon X\to B$, and $\bA$ is the orthogonal complement.
\end{remark}

\begin{lemma}
\label{lemma_std_compbase}
For any $G$-Mori fibre space $X/B$ the SOD $\cBstd(X/B)$ has the form 
\begin{equation}
\label{eq_Kstd}
\langle \bB_1,\ldots,\bB_m, \pi^*\cA_{\le 1}(B) \rangle.
\end{equation}
Therefore, the SOD-s $\cBstd(X/B)$ satisfy condition \ref{i_rk0} from Theorem~\ref{thm_compatible_gives_atomictheory}.
Furthermore, we have 
\begin{equation*}
    \ker \pi_*=\langle \bB_1,\ldots,\bB_m\rangle.
\end{equation*}

\end{lemma}
\begin{proof}
The first part of the statement follows immediately from Definition~\ref{def_standard}.
The last claim holds by Lemma \ref{lemma_Bridgeland}.
\end{proof}

We sometimes denote the SOD of $\ker \pi_*$ given in Lemma \ref{lemma_std_compbase} by $\cK^{\rm std}(\pi)$ so that
Definition~\ref{def_standard} takes the form
$$\cBstd(X/B)=\langle\cK^{\rm std}(\pi),\pi^*\cA_{\le  1}(B)\rangle.$$

\begin{lemma}\label{lem_std_compbaseAndIso}
The semi-orthogonal decompositions $\cBstd(X/B)$ satisfy condition \ref{i_rk1} in Theorem~\ref{thm_compatible_gives_atomictheory}: they are compatible with isomorphisms of Mori fibre spaces.
\end{lemma}

\begin{proof}
    Let 
    \[\begin{tikzcd}    X_1\ar[r,"\phi"]\ar[d,"\pi_1",swap] & X_2\ar[d,"\pi_2"]\\ 
    B_1\ar[r,"\psi"] &B_2
    \end{tikzcd}\]
    be an isomorphism between two $G$-Mori fibre spaces $X_1/B_1$ and $X_2/B_2$.
    Then $\psi^*\cA_{\le 1}(B_2)$ coincides with $\cA_{\le  1}(B_1)$, so it is enough to compare $\cK^{\rm std}(\pi_1)$ with $\cK^{\rm std}(\pi_2)$ in each case of Definition~\ref{def_standard}.
    If $X$ is not $G$-birationally rich, then $\cK^{\rm std}(\pi_i)=\ker \pi_i$ has one component and the claim follows directly.
    If $X$ is $G$-birationally rich, then $\cK^{\rm std}(\pi_1)$ and $\cK^{\rm std}(\pi_2)$ both split into two components. One checks that at least one of them is preserved by all isomorphisms $X_1\to X_2$ since it is defined invariantly in terms of certain divisor classes. Hence, the other component is preserved, too.
    So the pull-back  $\phi^*\cBstd(X_2/B_2)$ equals $\cBstd(X_1/B_1)$.
\end{proof}

\begin{proposition}
\label{prop_sarkisov}
The semi-orthogonal decompositions $\cBstd(X/B)$ satisfy condition \ref{i_rk2} in Theorem~\ref{thm_compatible_gives_atomictheory}. That is, let $\chi\colon X_1\dashrightarrow X_2$ be a $G$-Sarkisov link 
 \[\begin{tikzcd}        &Z\ar[dl,swap,"\sigma_1"]\ar[dr,"\sigma_2"]&\\        X_1\ar[d,swap,"\pi_1"]\ar[dashed,rr,"\chi"]&& X_2\ar[d,"\pi_2"]\\ 
        B_1\ar[dr,"\beta_1",swap]&& B_2\ar[dl,"\beta_2"]\\
        &B.&
    \end{tikzcd}\] 
    Then the following SOD-s of $\Db(Z)$ are mutation-equivalent:
    \begin{equation*}
        \langle \cK(\sigma_1), \sigma_1^*\cBstd(X_1/B_1)\rangle \cong
        \langle \cK(\sigma_2), \sigma_2^*\cBstd(X_2/B_2)\rangle.
    \end{equation*}
\end{proposition}

To prove Proposition~\ref{prop_sarkisov}, we go through the list of Sarkisov links and find mutations for each case. While some cases only require mutations as in Section~\ref{sec_derivedcat_coh}, some others are more involved. We have moved the proof to Appendix~\ref{ap_proof_prop} because it is quite lengthy.

    In fact, our proof of Proposition~\ref{prop_sarkisov} shows a slightly stronger statement. For a rational derived contraction $\beta$ in dimension $\le 1$,  denote by $\cK(\beta)$  the trivial  SOD of $\ker \beta_*$. Then  the following SOD-s
    of $\ker f_*$, where $f=\beta_1\pi_1\sigma_1=\beta_2\pi_2\sigma_2\colon Z\to B$, are mutation-equivalent:
   \begin{equation*}
        \langle \cK(\sigma_1), \sigma_1^*\cK^{\rm std}(\pi_1), \sigma_1^*\pi_1^*\cK(\beta_1)\rangle \cong
        \langle \cK(\sigma_2), \sigma_2^*\cK^{\rm std}(\pi_2), \sigma_2^*\pi_2^*\cK(\beta_2)\rangle.
    \end{equation*}

\begin{theorem}\label{thm_construction}
    Let $\kk$ be an  algebraically closed field and $G$ be a group.
    For any $G$-Mori fibre space $X/B$ over $\kk$ of dimension $2$, let $\cBstd(X/B)$ be as in Definition~\ref{def_standard}.
    Consider the following assignment:\\
    for $Y\in G-\cVar_{\kk, \le  2}^{\MMP}$  
    set $\cA(Y):=\cA_{\le 1}(Y)$ if $\dim Y\le 1$, and if $\dim Y=2$ let $f\colon Y\to X$ be any $G$-birational morphism onto a $G$-surface $X$ with nef canonical divisor, or onto a $G$-Mori fibre space $X/B$, and set  \[\cA(Y):=\begin{cases}\cE(f\colon Y\to X,\Db(X)) & \text{if $K_X$ nef,}\\
        \cE(f\colon Y\to X,\cBstd(X/B)) & \text{if $X/B$ is a $G$-Mori fibre space}.
    \end{cases}\]
    Then the semi-orthogonal decompositions $\cA(Y)$ of $\Db(Y)$  do not depend on the choice of $f$ (up to mutations) and define an atomic theory for $G-\cVar_{\kk, \le 2}^{\MMP}$.
\end{theorem}

\begin{proof}
Recall that the atomic domain
$G-\cVar_{\kk, \le 2}^{\MMP}$ is the disjoint union of its subdomains $G-\cVar_{\kk, 2}^{\kappa\ge 0, \MMP}$ (surfaces birational to surfaces with nef canonical class) and $G-\cVar_{\kk, \le 2}^{\kappa<0, \MMP}$ (surfaces birational to a Mori fibre space, all curves and points). Therefore it suffices to check the claims for these subdomains.

By Theorem~\ref{thm_compatible_gives_atomictheory}, SOD-s $\cA(Y)$ for $Y\in G-\cVar_{\kk, \le 2}^{\kappa<0}$ form an atomic theory on $G-\cVar_{\kk, \le 2}^{\kappa<0, \MMP}$: indeed
the three conditions of Theorem~\ref{thm_compatible_gives_atomictheory} are satisfied by the SOD-s $\cBstd$ by Lemma~\ref{lemma_std_compbase}, Lemma~\ref{lem_std_compbaseAndIso}, and Proposition~\ref{prop_sarkisov} respectively. Also, by Proposition~\ref{prop_compatible_gives_atomictheory}, SOD-s $\cA(Y)$ for $Y\in G-\cVar_{\kk,  2}^{\kappa\ge 0}$ form an atomic theory on $G-\cVar_{\kk, 2}^{\kappa\ge 0, \MMP}$. 
\end{proof}

\begin{proposition}\label{prop_subgroupcompstd}
    Let $G$ be a group and $H\subset G$ a subgroup. The $G$-atomic theory $\cA_G$ for $G-\cVar_{\kk,\le 2}^{\MMP}$ and the $H$-atomic theory $\cA_H$ for $H-\cVar_{\kk,\le 2}^{\MMP}$ constructed in Theorem~\ref{thm_construction} are subgroup compatible.
\end{proposition}

\begin{proof}
We have already observed in Remark~\ref{rem_dimOneSubgroupComp} that the $G$- and $H$-atomic theories in dimension $\le  1$ are subgroup compatible. Let $X$ be a $G$-surface. We need to show that the $H$-atomic decomposition $\cA_H(X)$ is (mutation-equivalent) to a refinement of the $G$-atomic decomposition $\cA_G(X)$.

We first reduce to the case when $X$ is $G$-minimal:
    Let $f\colon X\to Z$ be a birational $G$-morphism onto a $G$-minimal surface $Z$. We write $\cK_H(f)$ and $\cK_G(f)$ for the SOD of  $\ker f_*$ from Definition \ref{def_AgivenBandf} for the respective action. Note that $\cK_H(f)$ is (mutation-equivalent to) a refinement of $\cK_G(f)$.
    Indeed, 
    this holds by Lemma~\ref{lemma_ker}, because a blow-up of a finite $G$-orbit is a composition of blow-ups of finite $H$-orbits.
    We have  $\cA_H(X)=\langle \cK_H(f),\cA_H(Z)\rangle$ and similarly for~$G$. So if $\cA_H(Z)$ is a refinement of $\cA_G(Z)$, then $\cA_H(X)$ refines $\cA_G(X)$.

    We will now check the condition for a $G$-minimal surface $X$.
    If $K_X$ is nef, $\cA_G(X)=\cA_H(X)=\Db(X)$ and there is nothing to verify.
    It remains to consider the case when $X$ is $G$-minimal with a $G$-Mori fibre space structure $\pi\colon X\to B$.
    By running  $H$-MMP for $X$ over $B$, one can find an $H$-equivariant birational morphism $f\colon X\to Y$ over $B$ such that $Y/B'$ is an $H$-Mori fibre space and $B'/B$ is an isomorphism or a map $\P^1\to \Spec\kk$. There is a commutative diagram
    \[\begin{tikzcd}
        X\ar[d,"\pi",swap]\ar[r,"f"] &Y\ar[d,"\pi'"] \\
         B&B'.\ar[l, "\beta", swap] 
    \end{tikzcd}\]    
     We will check that 
     \begin{align*}\cA_H(X)=\langle \cK_H(f), f^*\cB_H^{\rm std}(Y/B')\rangle  &=
     \langle \cK_H(f), f^*{\cK^{\mathrm{std}}_H}(\pi'), f^*(\pi')^*\cA_{\le 1}(B')\rangle
     \end{align*}
     is (up to mutations) a refinement of 
     $$\cA_G(X)=\cB_G^{\rm std}(X/B)=\langle {\cK^{\mathrm{std}}_G}(\pi), \pi^*\cA_{\le 1}(B)\rangle.$$

    If $X$ is not $G$-birationally rich, then ${\cK^{\mathrm{std}}_G}(\pi)=\ker \pi_*$ consists of one component, and the result follows since $\beta^*\cA_{\le 1}(B)$ is a right sub-SOD in $\cA_{\le 1}(B')$ (by the axioms of atomic theory in dimension $\le 1$).

    Assume now that $X$ is $G$-birationally rich. 
    We consider two cases:  either
    \begin{enumerate*}
        \item\label{PfsubgroupcompB} $B \cong \P^1$ and then $B=B'$, or
        \item\label{PfsubgroupcompK} $B=\Spec\kk$.        
    \end{enumerate*}
    Since $X$ is $G$-minimal, case \eqref{PfsubgroupcompB} happens only if $X$ is a Hirzebruch surface $\bF_n$ (by Proposition~\ref{prop_list} and Remark~\ref{rem:non-minimal-C}). So $f$ is an isomorphism over $\P^1$ and there is nothing to prove.
    In case \eqref{PfsubgroupcompK}, we have that $X$ is either a del Pezzo surface of degree $\{5,6,9\}$, or $X\cong \P^1\times\P^1$ (again by Proposition~\ref{prop_list}). We can assume that $Y$ is $H$-minimal and so $Y$ is also such  surface. We have that $\cB_G^{\rm std}(X/B)=\cS(X)$ is the $3$-block decomposition from Lemma~\ref{lem:3blockdecomp}.  Note that in almost all cases one has $B'=\Spec\kk$ and so $\cB_H^{\rm std}(Y/B')=\cS(Y)$.
    The only exception is $Y\cong \P^1\times\P^1$, where it is also possible that $B'\cong \P^1$ and then $\cB_H^{\rm std}(Y/B') = \cB_H^{\rm std}(\P^1 \times \P^1/\P^1)$, which has $4$ pieces,  refines $\cS(Y) = \cS(\P^1 \times \P^1)$, which has $3$ pieces.
    The result now follows:
     \begin{align*}
     \cA_H(X)=  \langle \cK_H(f),\, & f^*\cB_H^{\rm std}(Y/B')\rangle,  &&\text{which refines} \\
       \langle \ker(f_*),\, & f^*\cB_H^{\rm std}(Y/B')\rangle,  &&\text{which refines}\\
      \langle \ker(f_*),\, & f^*\cS(Y)\rangle, &&\text{which refines}\quad
      \cS(X)=\cB_G^{\rm std}(X/B)=\cA_G(X)    
     \end{align*}
    by Lemma~\ref{lem:StandardDecompbirrich} up to mutations.
\end{proof}

\subsection{Atomic theory for surfaces over non-closed fields}
\label{section_descent_for_AT}

We explain here, using descent techniques from Section~\ref{section_descent_for_SODs}, that for a perfect field $\kk$, building atomic theory for varieties over $\kk$ is the same as building Galois-equivariant atomic theory for varieties over $\bar\kk$. And the latter was constructed in Section~\ref{sec_construction}, so atomic theory for surfaces over any perfect field follows.

\begin{proposition}
\label{prop_ATdescent}
We assume that $\LL/\kk$ is a Galois field extension and $G=\Gal(\LL/\kk)$.
Let $H$ be a group and $\cV_{\kk}$ be an $H$-atomic domain over $\kk$, where we assume that $H$ acts on varieties $\kk$-linearly.  For any $X\in \cV_{\kk}$, there is an action of $H\times G$ on $X_{\LL}=X\times_\kk \Spec\LL$, coming from commuting $\kk$-linear actions of $H$  on $X$ and of $G$ on $\Spec\LL$. Assume $\cV_{\LL}$ is an $H\times G$-atomic domain of  varieties over $\LL$ such that scalar extension $X\mapsto X_\LL, f\mapsto f_\LL$ is a functor from $\cV_{\kk}$ to $\cV_\LL$. Then  $H\times G$-equivariant  atomic theory for $\cV_{\LL}$ yields $H$-equivariant atomic theory for $\cV_\kk$.    
\end{proposition} 
\begin{proof}
Let $X$ be in $\cV_\kk$. Recall the bijection between SOD-s of  $\Db(X)$ and $G$-invariant SOD-s of $\Db(X_\LL)$ from Proposition~\ref{prop_SODdescent}. Note that this bijection sends $H$-invariant SOD-s to $H\times G$-invariant SOD-s.
For any $X$ in $\cV_\kk$, let us say that an $H$-invariant SOD $\Db(X_{\kk})=\langle\bB_1,\ldots,\bB_n\rangle$
is atomic if and only if  the $H\times G$-invariant SOD $\Db(X_{\LL})=\langle p^*\bB_1,\ldots,p^*\bB_n\rangle$ is an atomic SOD for $X_\LL$.
We check that such SOD-s form an $H$-equivariant atomic theory on $\cV_\kk$.

For condition~\ref{ax:A1} in Definition~\ref{def_GSAT}, note that the correspondence in Proposition~\ref{prop_SODdescent} respects mutations: if $\langle\bB_1,\ldots,\bB_n\rangle$ is a mutation of $\langle\bB'_1,\ldots,\bB'_n\rangle$ then $\langle p^*\bB_1,\ldots,p^*\bB_n\rangle$ is a mutation of $\langle p^*\bB'_1,\ldots,p^*\bB'_n\rangle$.
For condition~\ref{ax:A2} in Definition~\ref{def_GSAT}, 
let $f\colon Y\to X$ be a morphism in $\cV_\kk$, then $f_\LL\colon Y_\LL\to X_\LL$ is a morphism in $\cV_\LL$. Choose 
atomic SOD-s 
$\Db(Y_\LL)=\langle \bA_1,\ldots,\bA_n,\bA_{n+1},\ldots,\bA_m\rangle$, $\Db(X_\LL)=\langle \bB_{n+1},\ldots,\bB_m\rangle$, 
such that $(f_\LL)^*\bB_j=\bA_j$ for $j=n+1,\ldots,m$, and consider SOD-s
$\Db(Y)=\langle p_*\bA_1,\ldots,p_*\bA_n,p_*\bA_{n+1},\ldots,p_*\bA_m\rangle$, $\Db(X)=\langle p_*\bB_{n+1},\ldots,p_*\bB_m\rangle$. Then for all  $j=n+1,\ldots,m$ one has
$$p^*(f^*p_*\bB_j)=(f_\LL)^*p^*p_*\bB_j=f^*\bB_j=\bA_j=p^*(p_*\bA_j),$$
hence $f^*p_*\bB_j=p_*\bA_j$ as required.
\end{proof}

Now we are ready for
\begin{proof}[Proof of Theorem~\ref{thm:atoms-surfaces2}]
    We apply Proposition~\ref{prop_ATdescent} to $\LL=\bar\kk$, $G=\Gal(\bar\kk/\kk)$. 
    
    Take $\cV_\kk$ be the atomic domain of $H$-varieties over $\kk$ of dimension $\le 2$,  where the $H$-action is  $\kk$-linear,  and rational derived contractions,  and take $\cV_{\bar\kk}=(H\times G)-\cVar_{\bar\kk, \le 2}^{\MMP}$.
    For any variety $X$ over $\kk$, the action of $G$ on $N_1(X_{\bar\kk})$ has finite orbits (since any irreducible curve on $X_{\bar \kk}$ is defined over a finite field extension of $\kk$). Therefore, for any $H$-surface $X$ over $\kk$, the surface $X_{\bar\kk}$ is an $H\times G$-surface (recall that by Definition~\ref{def_Gvariety} of a $G$-variety, the $G$-action on $N_1$  must have finite orbits).
    Recall that scalar extension preserves rational derived contractions and hence defines a functor $\cV_{\kk}\to \cV_{\bar\kk}$. Theorem~\ref{thm:atoms-surfaces} provides an atomic theory for $\cV_{\bar\kk}$. Now the first claim follows from Proposition~\ref{prop_ATdescent}.

    For compatibility, in view of Proposition~\ref{prop_ATdescent} it is enough to check compatibility of equivariant atomic theories on $X_{\bar\kk}$ under group restriction, which holds by Theorem~\ref{thm:atoms-surfaces}. If $H'\subset H$ is a subgroup, use compatibility of $(H\times G)$- and $(H'\times G)$-equivariant atomic SOD-s for $\Db(X_{\bar\kk})$. If $\kk\subset \FF$ is an algebraic field extension, let $G'=\Gal(\bar \kk/\FF)\subset G$, then  use compatibility of $(H\times G)$- and $(H\times G')$-equivariant atomic SOD-s of $\Db(X_{\bar\kk})$. 
\end{proof}

\section{Applications to rationality and birationality}

In this section we work with the atomic theory for $G$-surfaces constructed in Theorem \ref{thm:atoms-surfaces} and obtain applications to the birational geometry of surfaces.
In subsections~\ref{section_51atomspermtype} and ~\ref{section_52invariants} we consider the general case of $G$-surfaces over an algebraically closed field. Then, in subsection~\ref{section_53geomcase} we consider the \emph{geometric case}
when $\kk$ is algebraically closed
and $G$ is a finite group acting faithfully and $\kk$-linearly on a surface $X$ defined over~$\kk$, 
and in subsection~\ref{section_54arithmcase} we consider the \emph{arithmetic case},
when $\kk$ is a perfect field,
$X$ is surface over $\kk$ and
$G = \Gal(\ol{\kk}/\kk)$ acts on $X_{\ol{\kk}}$.

\subsection{Atoms of permutation type}
\label{section_51atomspermtype}
Let $\kk$ be an algebraically closed field and $X$ be a $G$-surface over $\kk$. 

The atomic theory from Theorem \ref{thm:atoms-surfaces} associates to $X$ a finite collection of triangulated $G$-categories, called the \emph{$G$-atoms} of $X$. 
Recall (Definition~\ref{def_permtype}) that a $G$-atom is of \emph{permutation type} if it is generated by a $G$-transitive block of exceptional objects. To any atom of permutation type a twisting cohomology class $[\alpha]\in\rH^2(G,(\kk^*)^n)$ is associated (see Definition~\ref{def_obstruction-for-blocks}), and the atom is called \emph{non-twisted} if this class vanishes (Definition~\ref{def_non-twisted-atom}). Also, if $H\subset G$ is the stabiliser of one of the objects in the block, a twisting  cohomology class $[\alpha_H]\in \rH^2(H,\kk^*)$ is defined. By Corollary~\ref{cor_Z}, the atom is non-twisted if and only if $[\alpha_H]$ vanishes. 

The main application of the atomic theory is the following:

\begin{proposition}
\label{prop_essential}
$G$-birational surfaces $X$ and $Y$ have the same collection of atoms (with multiplicities), after removing all non-twisted permutation type atoms.
\end{proposition}
\begin{proof}
    If $X$ and $Y$ are $G$-birational then they can be connected by a chain of blow-ups of finite $G$-orbits (taken in different directions). For one blow-up $\tZ\to Z$ the $G$-atoms of $\tZ$ and $Z$ are the same up to one non-twisted atom of permutation type by Lemma~\ref{lem:blowup-perm}. 
\end{proof}

We will characterize surfaces that have only permutation type atoms.
Recall that $X$ is \emph{$G$-birationally rich} if it is $G$-birational to a $G$-del Pezzo surface of degree $\ge 5$ 
or to a $G$-conic bundle of degree $\ge 5$ (Definition~\ref{def:BirationallyRich}).
Note that for the trivial group and an algebraically closed field, birationally rich means the same thing as rational.

\begin{proposition}
\label{prop:bir-rich}
A $G$-surface $X$ has only permutation type atoms if and only if $X$ is $G$-birationally rich. 
\end{proposition}

\begin{proof}
Note that having only atoms of permutation type is a birational property, because any blow-up (resp. blow-down) 
adds (resp. removes) atoms of permutation type by Lemma \ref{lem:blowup-perm}.
Therefore, we can assume that $X$ is $G$-minimal.

If $K_X$ is nef, then there is only one atom $\Db(X)$ and it cannot be generated by a fully orthogonal exceptional collection because its Serre functor $S(-) = -\otimes \omega_X[2]$ is non-trivial.

If there is a $G$-Mori fibre space structure $\pi\colon X\to B$, we show that $\cK^{\rm std}(\pi)$, which is the SOD of $\ker\pi_*$ such that $\cBstd(X/B)=\langle \cK^{\rm std}(\pi), \pi^*\cA(B)\rangle$, consists of permutation type atoms if and only if $X$ is $G$-birationally rich.

Indeed, if $X$ is $G$-birationally rich, then one sees in Definition~\ref{def_standard} that all atoms are of permutation type.
If $X$ is not $G$-birationally rich, then $\cK^{\rm std}(\pi)=\ker\pi_*$ is one atom. We claim that $\ker\pi_*$ contains $\Db(C)$ of some curve $C$, and is therefore not generated by one block (because every thick triangulated subcategory of a category generated by a block is generated by a block). Indeed, if $B$ is a curve, then $\ker \pi_*$ contains a copy of $\Db(B)$ by Orlov's theorem on projective bundles \cite{Orlov-blowup}. If $B=\Spec \kk$, then $X$ is a del Pezzo surface of degree $\le  4$, so $\pi$ factorises through a contraction $X\to \P^1$ (which is not $G$-equivariant), and therefore $\ker\pi_*$ contains a copy of $\Db(\P^1)$.
\end{proof}

\begin{proposition}
\label{prop:bir-rich2}
    A $G$-minimal del Pezzo surface of degree $5$ has only non-twisted permutation type atoms.
\end{proposition}
    
Note that the other types of birationally rich 
$G$-del Pezzo surfaces $\cD_9$, $\cD_8$, $\cD_6$ can have twisted or non-twisted permutation type atoms. 

\begin{proof}
We need to analyse the
$G$-atoms of the del Pezzo surface $X$ of degree $5$ in more detail. Let $H_1, \ldots, H_5$ be the $1$-classes and  $h_1, \ldots, h_5$ be the $0$-classes on $X$ respectively, one has  $h_i+H_i=-K_X$ by Lemma~\ref{lemma_dualclasses}. 
The group $G$ acts transitively (by $G$-minimality of $X$) on these two sets.
Recall that (by Definition~\ref{def_standard}) the $G$-atoms of $X$ are $\langle\cO\rangle$ (non-twisted), 
$\langle \cO(-h_1), \ldots, \cO(-h_5) \rangle$, and $\langle \cE\rangle$, where $\cE$ is a vector bundle of rank $2$ fitting into 
the exact  sequence~\eqref{eq_vectorbundleE}. Let $H\subset G$ be the stabiliser of $h_1$ (or of $H_1$). We claim that the twisting classes $[\alpha_{\cO(-h_1)}], [\alpha_{\cO(-H_1)}]\in \rH^2(H,\kk^*)$ are trivial. Indeed, $\cO(-h_1)$ is the pull-back of $\cO_{\P^1}(-1)$ under an $H$-equivariant map $X\to \P^1$ given by $|h_1|$. Hence $[\alpha_{\cO(-h_1)}^2]=1$ by Lemma~\ref{lemma_alphaalpha}(5) and Example~\ref{example_cocycleforPn}. Similarly $[\alpha_{\cO(-H_1)}^3]=1$. But $\cO(-h_1)\otimes \cO(-H_1)\cong \omega_X$ is a $G$-equivariant bundle, hence $[\alpha_{\cO(-h_1)}][\alpha_{\cO(-H_1)}]=1$ and both classes  must  vanish.
Therefore the $G$-atom $\langle \cO(-h_1), \ldots, \cO(-h_5) \rangle$ is non-twisted (see Corollary~\ref{cor_Z}). 

It remains to check that the twisting class
$[\alpha_\cE] \in \rH^2(G, \kk^*)$ corresponding to the $G$-atom $\langle \cE \rangle$ is also trivial.
By~\eqref{eq_vectorbundleE}, $\Lambda^2 \cE \cong \omega_X$  and is $G$-linearisable, we deduce that
$[\alpha_{\cE}]^2$ is trivial (using $\cE\cong\cE^\vee\otimes\omega_X$ and  Lemma~\ref{lemma_alphaalpha}(2,3)).
On the other hand, $\langle \cE \rangle$ as an $H$-atom is non-twisted. Indeed, there is an $H$-equivariant birational contraction $X\to \P^2$ given by $|H_1|$, hence $H$-atoms of $X$ are non-twisted or are pulled back from $\P^2$. As above, $H$-twisting classes  pulled back from $\P^2$ have order dividing $3$. Combining this with $[\alpha_{\cE}^2]=1$ we get that the restriction of $[\alpha_{\cE}]$ onto $H$ is trivial. Therefore by the restriction-corestriction formula one has $[\alpha_{\cE}^5]=1$ since $[G:H]=5$. It now follows that  $[\alpha_{\cE}]=1$. 
\end{proof}

We relate permutation atoms and the standard obstructions to rationality of $G$-surfaces,
formulated in terms of the $G$-action on $\Pic(X)$.
Recall that the Bogomolov--Prokhorov obstruction 
(see \cite[\S2.2]{Manin-rational} for the arithmetic case, \cite{BogomolovProkhorov} for the geometric case, and \cite{Kunyavskii} for a more general case)
is
the first cohomology group
$
\rH^1(G, \Pic(X))
$
and the Amitsur obstruction 
(see \cite{Liedtke-SB} for the arithmetic case
and 
\cite[\S6]{BCDP-finite} for the geometric case)
is defined as
\begin{equation}
\label{eq_Amitsur-geom}
\Am(X,G) = \Coker(\Pic^G(X) \to \Pic(X)^G),
\end{equation}
where $\Pic(X)^G\subset \Pic(X)$ denotes the subset of classes of $G$-invariant line bundles, and $\Pic^G(X)$ the set of isomorphism classes of $G$-linearised line bundles.
It is easy to see that  $\Am(X,G)$ and (in the case of finite or profinite groups) $\rH^1(G, \Pic
(X))$ are $G$-birational invariants of $X$.

Recall that a \emph{permutation $G$-module} is a finitely generated free abelian group with a $G$-action that admits an integral basis which is permuted by $G$.

\begin{proposition}\label{ref:prop-atoms-obstructions}
\begin{enumerate}
    \item\label{it:prop-atoms--perm} If all $G$-atoms of $X$ are of permutation type, then $\Pic(X)$ is a stably permutation $G$-module. More precisely, $\Pic(X) \oplus \Z^2$ with the trivial action on $\Z^2$ is a permutation $G$-module. In particular,  we have $\rH^1(H, \Pic(X)) = 0$ for all finite subgroups $H \subset G$.
    \item\label{it:prop-atoms--non-twisted} If all $G$-atoms of $X$ are of non-twisted permutation type, then $\Am(X,H) = 0$ for all $H \subset G$.
\end{enumerate}
\end{proposition}
\begin{proof}
\eqref{it:prop-atoms--perm}
By Proposition \ref{prop:bir-rich}, $X$ must be $G$-birationally rich, in particular, it is a rational surface. We have an isomorphism, 
see e.g. 
\cite[Lemma 4.2]{KarmazynKuznetsovShinder}
\[
\rK_0(X) \cong \Z \oplus \Pic(X) \oplus \Z, \quad
[F] \mapsto (\rank(F), c_1(F), \chi(F)).
\]
By construction it commutes with the $G$-action, where $G$ acts trivially on the two $\Z$ summands.
Thus it suffices to check that $\rK_0(X)$ is a permutation module. We have
\[
\rK_0(X) \cong
\rK_0(\bA_1) \oplus \dots \oplus \rK_0(\bA_n),
\]
where $\bA_i$ are $G$-atoms of permutation type,
and each $\rK_0(\bA_i)$ is a permutation $G$-module by definition of a permutation type atom.
The final statement follows when $G$ is finite as in \cite[Corollary 2.5.2]{BogomolovProkhorov}.

{\eqref{it:prop-atoms--non-twisted}} Since the atomic theory is compatible with passing to a subgroup, 
all $H$-atoms of $X$ are of non-twisted permutation type, and 
we can assume that $H = G$. 
Moreover, we can assume $X$ to be $G$-minimal by Proposition~\ref{prop_essential}.
By Proposition \ref{prop:bir-rich}, $X$ is $G$-birationally rich, hence
by Corollary \ref{cor:bir-rich-models} we have to check the cases $\cD_9$, $\cD_8$, $\cD_6$, $\cD_5$, $\cC_8$. 
We will check that in each case $\Pic(X)^G$ is generated by classes of $G$-linearisable line bundles.

Assume first that $X$ is a $G$-del Pezzo surface. We have $\Pic(X)^G \otimes \QQ = \QQ \cdot K_X$. Since the canonical bundle always admits a $G$-linearisation, if $K_X$ is indivisible, which holds in the cases $\cD_5$ and $\cD_6$, then $\Am(X,G) = 0$. In the case
$\cD_9$, if atoms are non-twisted, then $\cO(-1)$ is $G$-linearisable, and $\Am(X, G) = 0$. Similarly, in the
$\cD_8$ case, if atoms are non-twisted, the line bundle $\cO(-1,-1)$ on $\P^1 \times \P^1$ is $G$-linearisable, and $\Am(X, G) = 0$.

Finally let us consider the $\cC_8$ case, that is $X = \bF_n$. Let $h, s \in \Pic(X)$ be the standard basis. The assumption that $G$-atoms of $X$ are non-twisted implies that both $\cO(-h)$ and $\cO(-s)$ are $G$-linearisable, hence $\Am(X,G) = 0$.
\end{proof}

We will now give an application for surfaces which are not $G$-birationally rich.

\begin{theorem}\label{thm:dP4-bir}
Let $X$ and $X'$ be $G$-minimal del Pezzo surfaces of degree $4$.
Then the following conditions are equivalent:
\begin{enumerate}[label=(\alph*)]
    \item\label{it:dP4--iso} $X$ and $X'$ are $G$-isomorphic;
    \item\label{it:dP4--bir} $X$ and $X'$ are $G$-birational;
    \item\label{it:dP4--atoms} $X$ and $X'$ have the same $G$-atoms.
\end{enumerate}
\end{theorem}
\begin{proof}
Implication \ref{it:dP4--iso} $\implies$ \ref{it:dP4--bir} is trivial. Implication \ref{it:dP4--bir} $\implies$ \ref{it:dP4--atoms} follows from Proposition~\ref{prop_essential} and Definition~\ref{def_standard} for atoms of a $G$-minimal del Pezzo surface of degree $4$. For the implication \ref{it:dP4--atoms} $\implies$ \ref{it:dP4--iso} see~\cite{AE_Torelli}.
\end{proof}

\subsection{Invariants of birational maps}
\label{section_52invariants}
Let $\kk$ be an algebraically closed field.
Let $\mathrm{Atoms}_G$ denote
the set of equivalence classes of 
atoms of $G$-surfaces.
For every $G$-surface $X$
we can consider 
\[
\mathrm{At}(X) \in \Z[\mathrm{Atoms}_G],
\]
the sum of equivalence classes of $G$-atoms of any atomic semi-orthogonal decomposition of $\Db(X)$.
For any two $G$-birational surfaces $X$ and $Y$, the difference
$\mathrm{At}(Y) - \mathrm{At}(X)$ consists of non-twisted permutation type  $G$-atoms, which we can identify with the corresponding 
finite transitive $G$-sets. We define 
\[
a(X, Y) := \mathrm{At}(Y) - \mathrm{At}(X) \in \mathrm{Burn}(G).
\]
Here $\mathrm{Burn}(G)$ is the Burnside group of $G$, which  is the free abelian group on 
isomorphism classes of
finite transitive $G$-sets.

On the other hand, for every $G$-birational map $\phi$ between $G$-surfaces $X \dto Y$
there is an invariant
\cite{LinShinderZimmermann, LinShinder, KreschTschinkel-invariant}
\[
c^G(\phi) = \ExDiv(\phi^{-1}) - \ExDiv(\phi) \in \mathrm{Burn}(G)
\]
where $\ExDiv(-)$ denotes the finite $G$-set of exceptional divisors of a given birational map.
This is the unique invariant which is additive, i.e. $c^G(\psi \circ \phi) = c^G(\psi) + c^G(\phi)$ and such that if $\phi\colon \tX \to X$ is a blow-up of a $G$-orbit $Z$, then
$
c^G(\phi) = -[Z].
$

For varieties of dimension $\ge 3$ this invariant, restricted to birational \textbf{automorphisms} of a variety, is in general non-zero \cite{LinShinder}.
On the other hand, for surfaces
we get the following, generalizing the same result in the arithmetic case proven in \cite{LinShinderZimmermann}.

\begin{theorem}\label{thm:motivic}
For all $G$-birational maps $\phi\colon X \dto Y$ we have $c^G(\phi) = a(X,Y)$. In particular, we have $c^G(\mathrm{Bir}^G(X)) = 0$ for any $G$-surface $X$.
\end{theorem}
\begin{proof}
Since both $c^G(\phi)$ and $a(X,Y)$ are additive for compositions, it suffices to check the claim for a blow-up $\phi\colon \wt{X} \to X$ of a $G$-orbit of points $Z \subset X$. In this case we have
\[
c^G(\phi) = -[Z] = a(\wt{X}, X)
\]
and the result follows for all birational maps by additivity.
The final claim follows since $a(X, X) = 0$.
\end{proof}

\subsection{Geometric case: $G$-linearisability and stable $G$-linearisability}
\label{section_53geomcase}
Assume $\kk$ is an algebraically closed field and $G$ is a finite group acting faithfully and $\kk$-linearly on a surface $X$ over $\kk$. Recall that a permutation type atom of $X$ is equivalent to $\Db(\kk^n{-}\mathrm{mod})$ with a $\kk$-linear $G$-action, such atoms are  classified by a finite transitive $G$-set and a cohomology class in $\rH^2(H,\kk^*)$ for the stabiliser $H\subset G$ of an element in the set, see Corollary~\ref{cor_Z}. In particular, non-twisted permutations type atoms are classified by finite transitive $G$-sets.

\begin{example}\label{ex:atoms-P2-linear}
Let $X = \P^2_{\kk}$ with a $\kk$-linear $G$-action given by an injective  homomorphism $G \to \PGL_3(\kk)$. 
Then 
the $G$-atomic decomposition of $X$ is
\[
\Db(X) = \langle \langle\cO(-2)\rangle, \langle\cO(-1)\rangle, \langle\cO\rangle \rangle.
\]
By Example \ref{example_cocycleforPn}
the twisting class $[\alpha] \in \rH^2(G, \kk^*)$ of the permutation atom $\langle \cO(-1) \rangle$ is given as  the restriction of the extension class in $\rH^2(\PGL_3(\kk), \kk^*)$ defined by the exact sequence
\[
1 \to \kk^* \to \GL_3(\kk) \to \PGL_3(\kk) \to 1.
\]
Thus $[\alpha]$ is trivial if and only if the action of $G$ on $\P^2$ factors through $G \to \GL_3(\kk)$.
In this case all $G$-atoms of $X$ are of non-twisted permutation type (again by Example~\ref{example_cocycleforPn}). 
\end{example}

We say that a $G$-surface $X$ is \emph{projectively $G$-linearisable} 
if 
$X$ is $G$-birational to $\P^2$ with some $G$-action.
Furthermore, $X$ is \emph{$G$-linearisable} if it is $G$-birational to $\P^2$ with an action  that factors through $\GL_3(\kk)$. 
Finally, $X$ is \emph{stably $G$-linearisable} if there is $m>0$ such that $X\times \P^{m}$ (with trivial $G$-action on $\P^m$) is $G$-birational to $\P^{m+2}$ with an action  that factors through $\GL_{m+3}(\kk)$. 
These are the same  conventions as in \cite{HT-torsors}, \cite{BvBT-cubics}, \cite{KT-toric}, \cite{Kunyavskii}, however in \cite{PSY} the term linearisable is used for what we call projectively linearisable.

\begin{proposition}\label{prop:Glin-trivial-atoms}
Assume that $X$ is a projectively $G$-linearisable surface. Then
\begin{enumerate}
    \item\label{it:Glin--perm} All atoms of $X$ are of  permutation type.
    \item\label{it:Glin--non-twisted} $X$ is $G$-linearisable if and only if the atoms of $X$ are non-twisted.
\end{enumerate}
\end{proposition}
\begin{proof}
By Lemma \ref{lem:blowup-perm} 
these properties are birational invariant.
Thus we can assume $X = \P^2$,
in which case \eqref{it:Glin--perm} and \eqref{it:Glin--non-twisted} are true by Example \ref{ex:atoms-P2-linear}.
\end{proof}

\begin{example}\label{ex:dP5-rigid}
    We describe an example of a $G$-minimal del Pezzo surface that has only non-twisted permutation atoms but is not projectively $G$-linearisable. 
    Let $X$ be a del Pezzo surface
    of degree $5$ over an algebraically closed field $\kk$ of characteristic zero. 
    We have $\Aut(X) = S_5$ and there are $5$ subgroups $G \subset S_5$, up to conjugation,  such that $\rho^G(X) = 1$. The latter condition is equivalent to $G$ acting transitively on the $5$ conic bundle structures on $X$.
    The possible subgroups are listed in \cite[Theorem 6.4]{DolgachevIskovskikh}: 
    \[
    S_5, \; A_5, \; F, \; D_5, \; C_5.
    \]
    Here $F$ is the order $20$ subgroup of $S_5$ generated by $(1\,2\,3\,4\,5)$ and $(2\,3\,5\,4)$. Let $G$ be one of these groups.
    
    By Proposition \ref{prop:bir-rich2} all atoms of $X$
    are non-twisted of permutation type.
    In particular, $X$ is projectively $G$-linearisable if and only if it is $G$-linearisable  by
    Proposition~\ref{prop:Glin-trivial-atoms}.
        
    $G$-linearisability of $X$ is well-understood, see \cite[\S4.2]{PSY} and the references therein.
    It turns out that~$X$ is not
    $G$-linearisable with respect to the first three groups $S_5$, $A_5$, $F$ and is 
    $G$-linearisable with respect to the last two groups $D_5$ and $C_5$.
\end{example}

\begin{conjecture}\label{conj:triv-stab-rat}
If a $G$-surface $X$ has only non-twisted permutation atoms, then $X$ is stably $G$-linearisable.
\end{conjecture}

Let us explain the evidence for this conjecture.
First note that by Proposition \ref{ref:prop-atoms-obstructions}, the standard obstructions to stable $G$-linearisability, namely the Bogomolov--Prokhorov obstruction and the Amitsur obstruction, both vanish.
By the classification of birationally rich surfaces (see Corollary \ref{cor:bir-rich-models}), the conjecture needs to be checked in the cases $\cD_9$, $\cD_8$, $\cD_6$, $\cD_5$, and $\cC_8$, and it indeed holds in many cases. 
In the~$\cD_9$ case, $X$ is linearisable by Example \ref{ex:atoms-P2-linear}.
See \cite[Proposition 12]{BvBT-cubics} for the $\cD_8$ case.
Some partial results are known for the $\cD_6$ case \cite{HT-torsors}, \cite{KT-toric}.
The result is true by
\cite[Proposition 4.7]{Prokhorov-fields} in the~$\cD_5$ case.

\subsection{Arithmetic case: rationality and birationality over non-closed fields}
\label{section_54arithmcase}
In this section, we consider the arithmetic case, that is
we work over a perfect field $\kk$ with the absolute Galois group  $G = \Gal(\ol{\kk}/\kk)$.

Atomic theory for surfaces over $\kk$ is constructed in~\S\ref{section_descent_for_AT} using descent, see Theorem~\ref{thm:atoms-surfaces2}. That is, we consider $X_{\bar\kk}$ with the natural $G$-action and use $G$-atomic theory on $X_{\bar\kk}$. The atoms of $X$ are then descended from $G$-atoms of 
$X_{\ol{\kk}}$ using Proposition~\ref{prop_SODdescent}. 
In particular, a permutation type $G$-atom on $X_{\ol{\kk}}$ generated by exceptional objects $E_1,\ldots,E_n$ descends to an atom of $X$ generated by a single weakly exceptional object $E$ such that the scalar extension of $E$ is $(E_1\oplus\cdots\oplus E_n)^{\oplus r}$ for some $r$. The atom is non-twisted if and only if $r=1$, see Proposition~\ref{prop_alphaD}.

\begin{definition}
We say that an atom $\bA$ of a surface $X$ over $\kk$ is \emph{small} if $\bA$ is generated by a single weakly exceptional object, and is \emph{big} otherwise.     
\end{definition}

Let $\bA$ be a small atom and  $E \in \bA$ be its weakly exceptional generator (which is unique up to shift). 
Then $D_\bA := \End(E)$ is a finite-dimensional division $\kk$-algebra, and we have an equivalence
\begin{equation}\label{eq:bA-DA}
\bA \cong \Db(\modd D_\bA).
\end{equation}
The centre  $L_\bA$ of $D_\bA$ is a finite field extension of~$\kk$.  We denote by $\ba_\bA=[D_\bA] \in \Br(L_\bA)$ the Brauer class of the central simple algebra $D_\bA$. In this section we will distinguish between a Brauer class $\ba\in\Br(L)$ and its second cohomology class  $[\alpha]\in \rH^2(\Gal(\bar L/L),\bar L^*)$. In particular, we will write the Brauer group additively, so that its trivial element is $0$.

By Proposition \ref{prop_descentforblocks}, small atoms of $X$ are exactly the ones descended from $G$-atoms of $X_{\ol{\kk}}$  of permutation type. By Proposition~\ref{prop_alphaD}, the twisting cohomology class $[\alpha]\in\rH^2(H,\bar\kk^*)$ of the 
corresponding $G$-atom of permutation type matches the Brauer class $\ba_\bA$.

Recall that the \emph{index} $\ind([A])$ of the class $[A] \in \Br(L)$ of a central simple $L$-algebra $A$ (and the index of $A$) is the degree
of the unique division algebra  which is Morita-equivalent to $A$. 

The equivalence class of a small atom
$\bA$ is determined by the isomorphism class of the division algebra~$D_{\bA}$ (using \eqref{eq:bA-DA}), hence by the pair $(L_{\bA}, \ba_\bA)$. We will often denote such an atom as $(L_{\bA}, \ba_\bA)$ or $\Db(L_{\bA}, \ba_\bA)$. Further, 
we associate two positive integers to $\bA$:
 $d_1=[L_\bA:\kk]$ and  $d_2=\deg_{L_\bA}D_{\bA}=\ind \ba_\bA$. We will say that such atom is of type $(d_1,d_2)$.

Surfaces with only small atoms have the following characterization. 

\begin{proposition}\label{prop:bir-rich-non-closed}
A surface $X$ is birationally rich if and only if all its atoms are small.    
\end{proposition}
\begin{proof}
This follows from Proposition \ref{prop:bir-rich} 
and Proposition \ref{prop_descentforblocks}.
\end{proof}

Recall that a surface $X$ over $\kk$ is \emph{toric}
if there is a torus $T$ defined over $\kk$ (that is, an algebraic group~$T$ over $\kk$ with $T_{\bar\kk}\cong (\mathbb{G}_m)^2$) acting on $X$ with a dense open orbit $U \subset X$, and the action on $U$ is free.
By \cite[Theorem 5.1]{Duncan-toric}
this is equivalent to an \emph{a priori} weaker statement that $X_{\ol{\kk}}$ is toric.

\begin{lemma}\label{lem:bir-rich-toric}     
A surface $X$ is birationally rich if and only if it  admits a toric model.    
If this is the case, we can choose a toric model of type $\cD_9$ (a Severi--Brauer surface), $\cD_8$, $\cD_6$, or $C \times C'$ where $C$ and $C'$ 
are conics.
\end{lemma}
\begin{proof}
Let $X$ be a minimal birationally rich surface.
By Corollary \ref{cor:bir-rich-models},
the $G$-surface $X_{\ol{\kk}}$
can be of type $\cD_9$, $\cD_8$, $\cD_6$, $\cD_5$, or $\cC_8$.
In the cases $\cD_9$, $\cD_8$, or $\cD_6$ the surface $X_{\ol{\kk}}$ is toric, hence $X$ is toric by the above remark.
In the $\cD_5$ case $X$ has a rational point \cite{SwinnertonDyer-dP5}, so $X$ is  rational  by \cite[Theorem 3.15]{Manin-rational}, hence $X$ has  a $\cD_9$ model. It remains to check that in the $\cC_8$ case
$X$ is birational to a product of two conics. 
Let $X \to C$ be a twisted form of the Hirzebruch surface $\bF_n$. 
If $n$ is odd, then $X$ is a rational surface \cite[Theorem 1.10(i)]{Manin-rational}.
If $n > 0$ is even, then $X$ can be nonrational in which case $X \cong \P_C(\cO \oplus \cO(n))$ where $C$ is a nonrational conic
\cite[Theorem 1.10(ii)]{Manin-rational}.
In this case, blowing up a  closed point of degree two
on the positive section of $X$, we obtain
a Sarkisov link of type II which replaces $n$ by $n-2$.
Eventually, we get a twisted form of the conic bundle $\bF_0 = \P^1 \times \P^1\to \P^1$  in which case $X \cong C \times C'$, see e.g. \cite[Lemma 7.3]{ShramovVologodsky}.
\end{proof}

\begin{remark}
\label{remark_index-pullback}
Small atoms are often related to  morphisms to  Severi--Brauer varieties which can be described in terms of invariant divisor classes \cite{Liedtke-SB}.
Namely, assume that a small atom $\bA$ of a surface $X$ over $\kk$ descends from a permutation-type atom of $X_{\bar \kk}$, which is generated by the $G$-orbit of a line bundle $\cO(D)$. Assume further that $\cO(D)$ is the pull-back of $\cO(k)$, where $k=\pm 1$, under a morphism $f\colon X_{\bar \kk}\xrightarrow{|kD|} \P^{n-1}_{\bar \kk}$. Then the index  of the corresponding Brauer class $\ba\in\Br(L_\bA)$ divides $n$. Indeed, let 
$H\subset G$ be the stabiliser of $\cO(D)$, so that $(\bar\kk)^H\cong L_\bA$. Then $f$ is $H$-equivariant for some Galois $H$-action on $\P^{n-1}_{\bar\kk}$, and $\ba$ matches the twisting class of the $H$-invariant object $\cO(k)$ by Lemma~\ref{lemma_alphaalpha}(5). By Example~\ref{example_SB}, the index of $\ba$ divides $n$.
\end{remark}

\begin{example}\label{ex:bir-rich-non-closed}
Let us describe the small atoms of the toric models from Lemma~\ref{lem:bir-rich-toric} explicitly. 
Here we will use the following notation: we write $(L_m, \ba_n)$ to represent the atom $\bA$, where the field $L_m=L_{\bA}$ has degree $m$ over $\kk$ and the index of the class $\ba_n=\ba_{\bA}$ divides~$n$. 
We refer to the formulas from Definition~\ref{def_standard} for the corresponding $G$-atoms of 
the surfaces $X_{\bar\kk}$. Note that the degree $[L_{\bA}:\kk]$ is the number of objects in the corresponding $G$-block,
see Proposition~\ref{prop_alphaD}.
\begin{enumerate}
    \item Let $X$ be a Severi--Brauer surface corresponding to a central simple algebra $A/\kk$ with the Brauer class $\ba_3$.
    Then the atoms of $X$ are  $(\kk, 0)$, $(\kk, \ba_3)$, $(\kk, 2\ba_3)$ (the index of $\ba_3$ divides $3$ by Example~\ref{example_SB}), see also Example~\ref{example_cocycleforPn}.
    \item Let $X$ be a minimal del Pezzo surface of degree $8$. Then the atoms of $X$ are $(\kk, 0)$, 
    $(L_2, \ba_2)$, $(\kk, \ba_4)$. Here the index of $\ba_2$ divides $2$ by Remark~\ref{remark_index-pullback} since the corresponding atom descends from the $G$-block $\langle\cO(-h_1), \cO(-h_2)\rangle$ and $|h_1|$ defines a map $X_{\bar\kk}\cong \P^1\times \P^1\to \P^1$. Likewise, the index of $\ba_4$ divides $4$ by Remark~\ref{remark_index-pullback} since this atom descends from the $G$-block $\langle\cO(-h_1-h_2)\rangle$ and $|h_1+h_2|$ defines a map $X_{\bar\kk}\cong \P^1\times \P^1\to \P^3$.     
    Note that $X$ is the Weil restriction:
    \begin{equation}
    \label{eq:dP8-RLK}
    X \cong \rR_{L_2/\kk} C
    \end{equation}
    where  $C$ is the conic corresponding to $\ba_2 \in \Br(L_2)$, see \cite[Proof of Proposition 5.3]{CTKM} or 
    \cite[Lemma 7.3]{ShramovVologodsky}.
    Therefore, if $\ba_2=0$ then $X \cong \rR_{L_2/\kk} \P^1_{L_2}\cong \P^1_{\kk}\times \P^1_{\kk}$ and $\ba_4=0$. 
    \item Let $X$ be a minimal del Pezzo surface of degree $6$. Then its atoms are
    $(\kk, 0)$,
    $(L_2, \ba_3)$, 
    $(L_3, \ba_2)$, which are descended from $\langle \cO\rangle$, $\langle \cO(-H_1),\cO(-H_2)\rangle$, $\langle \cO(-h_1),\cO(-h_2),\cO(-h_3)\rangle$.
    Hence $X_{L_2}$ admits a contraction to a Severi--Brauer surface as in Remark~\ref{remark_index-pullback}, and $\ba_3$ is the corresponding Brauer class. Similarly, $X_{L_3}$ admits a conic bundle structure over a conic, whose Brauer class is $\ba_2$. The two atoms
    $(L_2, \ba_3)$, 
    $(L_3, \ba_2)$
    are not independent, see Example \ref{ex:dP6-splittings}.
    \item Let $X = C \times C'$ where the Brauer classes of the conics are $\ba_2$ and $\ba_2'$. Then the atoms of $X$ are $(\kk, 0)$, 
    $(\kk, \ba_2)$,
    $(\kk, \ba_2')$,
    $(\kk, \ba_2+\ba_2')$.
\end{enumerate}
\end{example}

\begin{definition}
\label{def_trivial-atom}
A small atom $\bA$ of $X$ is \emph{trivial} if the corresponding  $G$-atom of $X_{\bar\kk}$ is non-twisted of permutation type. By Proposition~\ref{prop_alphaD}, an atom $\bA$ is trivial if and only if  
the division algebra $D_{\bA}$ is commutative: $D_\bA = L_\bA$, or equivalently if $[D_\bA] = 0 \in \Br(L_{\bA})$.   
\end{definition}

We have the following immediate corollary of Proposition~\ref{prop_essential}.

\begin{lemma}\label{lem:bir-non-closed-atoms}
If $X$ and $Y$ are birational, then they have the same atoms up to removing trivial atoms.
\end{lemma}

It turns out that, for birationally rich surfaces, the converse to Lemma \ref{lem:bir-non-closed-atoms} is also true, see Theorem~\ref{thm:almost-same-atoms-bir-rich} below, but is harder to prove.
Our next goal is to prove the following particular case (Theorem \ref{thm:rational-atoms}): surfaces with trivial atoms are rational.
We first need some preparation.

Recall that the \emph{index} $\ind(X)$ of a projective variety $X$ is defined as the greatest common divisor of degrees of closed points of $X$, or equivalently as the unique positive integer characterized by
\[
\Ima(\deg\colon \CH_0(X) \to \Z) = \ind(X) \cdot \Z.
\]
If $X$ has a rational point, then $\ind(X) = 1$, but the converse is   false in general.
We need the following 
well-known result, which can be checked case-by-case for the surfaces given in Lemma \ref{lem:bir-rich-toric}, see e.g. \cite{AuelBernardara}. We provide a short uniform proof based on the properties of tori.

\begin{lemma}\label{lem:bir-rich-rationality}
Let $X$ be a birationally rich surface.
The following conditions are equivalent:
\begin{enumerate}
    \item[(a)] $X$ is rational,
    \item[(b)] $X$ is stably rational,
    \item[(c)] $X(\kk) \ne \emptyset$,
    \item[(d)] $\mathrm{ind}(X) = 1$.
\end{enumerate}
\end{lemma}
\begin{proof}
It is clear that (a) $\implies$ (b) $\implies$ (c) $\implies$ (d).
Let us prove the converse implications.
We know that a birationally rich surface $X$ always
admits a toric model (Lemma \ref{lem:bir-rich-toric}). Assume that $\ind(X) = 1$. We will prove that the open dense orbit $U$, and hence $X$, has a rational point by the following standard argument.
Let $[U] \in \rH^1(\kk, T)$ be the class of torsor $U$ \cite[\S5.2]{Serre-GaloisCoh}.
Let $L/\kk$ be a finite extension 
such that $X(L) \ne \emptyset$
then it is known that
$U(L) \ne \emptyset$, i.e. $U_L$ is a trivial $T_L$-torsor 
\cite[Proposition 4]{VoskresenskiiKlyachko}.
Then
\[
[L:\kk]\cdot [U] = \mathrm{Nm}_{L/\kk} [U_L] = 0.
\]

Representing $\ind(X)$ as a finite
integral linear combination
of $[L_i:\kk]$ for finite field extensions $L_i/\kk$ (which are residue fields of certain closed points on $X$) we get that $\ind(X) \cdot[U] = 0$.
Since we assume that $\ind(X) = 1$, we obtain that $[U] = 0$ so that $U \cong T$ has a rational point.
This proves that (d) implies~(c).

To finish the proof, it suffices to check that (c) implies (a). 
First, note that if $X(\kk) \ne \emptyset$, then $X$ has an open dense subset isomorphic to a torus again by
\cite[Proposition 4]{VoskresenskiiKlyachko}. 
It remains to note that all $2$-dimensional tori are rational by a theorem of  
Voskresenskii
\cite{VoskresenskiiII}.
\end{proof}

Recall that $G = \Gal(\ol{\kk}/\kk)$.
The Amitsur group
of a smooth projective variety is defined (see~\cite{Liedtke-SB} and \cite[\S5]{CTKM}) as
$$\Am(X) := \Coker(\Pic(X) \to \Pic(X_{\ol{\kk}})^G) \subset \Br(\kk),$$
cf.~~\eqref{eq_Amitsur-geom} in the geometric case.
Let us 
 define the \emph{derived Amitsur group} as
\begin{equation}\label{eq:def-DAm}
\DAm(X) := \Coker\left(\rK_0(X) \to \rK_0(X_{\ol{\kk}})^G\right).
\end{equation}
Here $\rK_0(-)$ is the Grothendieck group. We note that 
$\DAm(X)$ is an invariant of the derived category $\Db(X)$. The cokernel in \eqref{eq:def-DAm}
has been considered in various contexts, in particular for Severi--Brauer varieties \cite{Karpenko-SB} and 
toric varieties \cite{MP-toric}.

\begin{proposition}\label{prop:DAm-sequence}
For a smooth projective geometrically rational surface $X$ we have a short exact sequence    
\[
0 \to \Z/\ind(X) \to \DAm(X) \to \Am(X) \to 0.
\]
\end{proposition}

\begin{proof}
We apply the Riemann--Roch theorem without denominators \cite[Example 15.1.5]{Fulton-intersection} for $X$ and~$X_{\ol{\kk}}$.
The result follows by applying the snake lemma to the following diagram
\[
\xymatrix{
0 \ar[r] & \Z \ar[r] & 
\rK_0(X_{\ol{\kk}})^G \ar[r]^{(\rank, c_1)\quad} & \Z \oplus \Pic(X_{\ol{\kk}})^G \ar[r] & 0 \\
0 \ar[r] & \CH_0(X) \ar[u]^{\deg} \ar[r] & \rK_0(X) \ar[r]^{(\rank, c_1)\quad} \ar[u] & \Z \oplus \Pic(X) \ar[r] \ar[u] & 0. \\
}
\]
Note that $\CH_0(X_{\bar\kk})\cong \Z$ since $X_{\bar\kk}$ is rational. 
Since the right vertical arrow is injective, the sequence of cokernels gives the desired exact sequence. \end{proof}

We  also define the derived Amitsur group of a central simple algebra by the same formula
\eqref{eq:def-DAm}. Similarly, we can define the derived Amitsur group of any triangulated subcategory in $\Db(X)$.

\begin{lemma}\label{lem:DAm-A}
If $L$ is a finite field extension of $\kk$ and $A$ is a central simple algebra over $L$, then $\DAm(\Db(A)) \cong \Z/n$ where $n$ is the index of $A$.
\end{lemma}
\begin{proof}
We can assume that $A$ is a division algebra, so that $\dim_L(A) = n^2$.
Let $d = [L:\kk]$ so that
\[
A \otimes_{\kk} \ol{\kk} \cong 
M_n(\ol{\kk})^{\times d}.
\]
We have the following diagram of infinite cyclic groups and injective maps between them:
\[
\xymatrix{
{\rK_0(M_n(\ol{\kk})^{\times d}})^G \ar[r]^{\quad\quad \dim_{\ol{\kk}}} & \Z \\ 
{\rK_0(A)} \ar[u]_{-\otimes \ol{\kk}} \ar[r]^{\quad\dim_{\kk}} & \Z. \ar[u]_= \\ 
}
\]
The image of the top horizontal arrow is $dn\Z$, and the image of the bottom horizontal arrow is $d n^2\Z$, thus the index of the image of the left vertical arrow is $n$.
\end{proof}

Recall that by Proposition \ref{prop:bir-rich-non-closed} birationally rich surfaces have only small atoms $\bA_1, \ldots, \bA_m$,
and that each small atom has an index, defined as the index of the corresponding division algebra.
In particular, an atom is trivial if and only if it has index one.

\begin{proposition}\label{prop:DAm-index-formula} 
For a birationally rich surface $X$ with atoms $\bA_1, \ldots \bA_m$, one has the formula
\begin{equation}\label{eq:index-formula}
\ind(X) \cdot |\Am(X)| = |\DAm(X)| = \prod_{i=1}^m \ind(\bA_i).
\end{equation}
\end{proposition}

\begin{proof}
The first equality follows from Proposition \ref{prop:DAm-sequence}.
The second equality follows from 
\[
\DAm(X)=\DAm(\Db(X)) \cong \bigoplus_{i=1}^m \DAm(\bA_i),
\]
equivalences \eqref{eq:bA-DA} and Lemma \ref{lem:DAm-A}.
\end{proof}

\begin{theorem}\label{thm:rational-atoms}
A surface $X$ is rational if and only if all its atoms are
trivial.   
\end{theorem}
\begin{proof}
If $X$ is rational, then its atoms are trivial by Lemma \ref{lem:bir-non-closed-atoms} because the atoms of $\P^2$ are trivial (see Example \ref{ex:bir-rich-non-closed}).
Conversely, assume that the atoms of $X$
are trivial.
By Proposition \ref{prop:bir-rich}, $X$ is birationally rich. 
By Proposition \ref{prop:DAm-index-formula},
we have $\ind(X) = 1$.
Finally, by Lemma \ref{lem:bir-rich-rationality} $X$ is rational.
\end{proof}

\begin{remark}
Interesting semi-orthogonal decompositions for higher-dimensional arithmetic toric varieties have been constructed in \cite{BDM} and \cite{BDLM-rationality}, and it is quite likely that those decompositions are atomic in an appropriate sense.
However, as shown in \cite{BDLM-rationality}, we should not expect that the 'if' implication of Theorem \ref{thm:rational-atoms} holds in higher dimension.
\end{remark}

\begin{corollary}\label{cor:finite-fields}
If $\kk$ is a finite field (or more generally, any field $\kk$ such that for all finite extensions $L/\kk$ we have $\Br(L) = 0$), then every birationally rich surface over $\kk$ is rational.
\end{corollary}
\begin{proof}
Since $X$ is birationally rich, each of its atom $\bA$ has the form $\Db(D_\bA)$, where $D_\bA$ is a finite-dimensional division $\kk$-algebra with centre $L_\bA$. By the assumption on $\kk$, we have that $D_\bA=L_\bA$ is commutative, i.e. the atom is trivial.
Now the result follows from Theorem \ref{thm:rational-atoms}.
\end{proof}

We say that an atom $\bA \subset \Db(X)$ is \emph{split} by an algebraic field extension $L/\kk$ if the atomic refinement of the subcategory $\bA_{L} \subset \Db(X_L)$ consists only of trivial atoms (see Definition~\ref{def_refinement} and Theorem~\ref{thm:atoms-surfaces2}).

\begin{corollary}\label{cor:splitting-atoms}
For any surface $X$ and an algebraic  field extension $L/\kk$, the base change $X_L$ is a rational surface if and only if $L$ splits all atoms of $X$.
\end{corollary}
\begin{proof}
This follows from Theorem \ref{thm:rational-atoms}.
\end{proof}

\begin{example}\label{ex:dP6-splittings}
Let $X$ be a minimal del Pezzo surface of degree $6$.
Then its atomic decomposition is
\[
\Db(X) = \langle \Db(L_2, \ba_3), \Db(L_3, \ba_2), \Db(\kk) \rangle
\]
for some field extensions $L_i/\kk$ of degree $i$ and a  Brauer classes $\ba_j\in\Br(L_i)$ of index dividing $j$ (see Example \ref{ex:bir-rich-non-closed}).
By the index formula \eqref{eq:index-formula}
we have 
\[
\ind(X) = \ind(\ba_2) \cdot \ind(\ba_3) \in \{1, 2, 3, 6\}.
\]

The atoms are not independent: we check below that $L_2$ splits the atom $(L_3,\ba_2)$, and $L_3$ splits the atom $(L_2,\ba_3)$. For the first, extend the base field to $L_2$.
Note that $\Db(L_2, \ba_3)_{L_2}\cong \Db(L_2, \ba_3)\times \Db(L_2, \ba_3')$ and must split into two atoms by Proposition~\ref{prop:bir-rich-non-closed}. 
Then the atomic decomposition  of $X_{L_2}$ is 
$$\Db(X_{L_2})=\langle\Db(L_2, \ba_3), \Db(L_2, \ba_3'), \Db(L_6, (\ba_2)_{L_6}), \Db(L_2)\rangle,$$ 
where $L_6=L_3\otimes_{\kk} L_2$ is a degree $6$ field extension of $\kk$ and $(\ba_2)_{L_6}$ is the image of $\ba_2$ under $\Br(L_3)\to \Br(L_6)$.
We claim that $(\ba_2)_{L_6}=0$. Otherwise, $X_{L_2}$ has an atom of type $(3,2)$. 
Let $X_{L_2}\to Z$ be a contraction to a minimal surface,
then $Z$ also has this atom. Hence $Z$ must be a minimal del Pezzo surface of degree $6$ (because other possible minimal surfaces do not have such atoms by Corollary~\ref{cor:bir-rich-models} and Example~\ref{ex:bir-rich-non-closed}), therefore $X_{L_2}=Z$. But $X_{L_2}$ has four atoms, a contradiction.
Note that by Example~\ref{ex:bir-rich-non-closed}(3) there is a contraction from $X_{L_2}$ to a Severi--Brauer surface, which  has $(L_2,\ba_3)$ and $(L_2,\ba'_3=-\ba_3)$ as atoms.

Similarly one can show that the atom $(L_2,\ba_3)$ is split by 
$L_3$, that is, $(\ba_3)_{L_6}=0$.
\end{example}

\begin{proposition}\label{prop:rich-irr}
\begin{enumerate}
\item\label{it:rich-irr--irr} If $X$ is a birationally rich surface which is not rational, then it is birational to 
a minimal surface
of \textbf{exactly} one of the following types: $\cD_9$, $\cD_8$, $\cD_6$, or $C \times C'$,
where $C$ and $C'$ are conics. Moreover, the type is determined by the non-trivial atoms of $X$.
\item\label{it:rich-irr--same} Let $X$ and $Y$ be two minimal surfaces of the same
type (rational or irrational) among $\cD_9$, $\cD_8$, $\cD_6$, or $C \times C'$.
Assume that $X$ and $Y$ have the same atoms.
Then in cases $\cD_8$ and $\cD_6$, $X$ and $Y$ are isomorphic.
In cases $\cD_9$ and $C \times C'$, $X$ and $Y$ are birational.
\end{enumerate}
\end{proposition}

\begin{proof}
\eqref{it:rich-irr--irr} By Lemma \ref{lem:bir-rich-toric}
$X$ has a model of type 
$\cD_9$, $\cD_8$, $\cD_6$, or $C \times C'$.
Atoms of surfaces of these types are given in Example \ref{ex:bir-rich-non-closed}.
Since $X$ is not rational, at least one of these atoms is nontrivial by Theorem \ref{thm:rational-atoms}. 
Therefore, the type of $X$ is uniquely determined by the presence of the following atoms, where we write $([L:\kk], \ind(\ba))$ for an atom $(L, \ba)$:
\[\begin{array}{|c|c|c|c|}
     \hline 
     \cD_9&\cD_8&\cD_6&C\times C'  \\\hline
     (1,3)&(2,2)&(3,2)\text{ or } (2, 3) &(1,2) \\
     \hline
\end{array}
\]
Note that
in the case $\cD_8$, the atom $(L_2, \ba_2)$ is not trivial as otherwise also $(\kk,\ba_4)$ is trivial by Example~\ref{ex:bir-rich-non-closed}(2).  

\eqref{it:rich-irr--same} We consider the cases separately:
\begin{itemize}
  \item $\cD_9$: $X$ is isomorphic to $\SB(A)$ or $\SB(A^\opp)$, where $A$ is the degree $3$ central simple algebra corresponding to $\ba \in \Br(\kk)$, and these two surfaces are birational by the standard Cremona transformation;
    \item $\cD_8$: $X$ and $Y$ are isomorphic by \eqref{eq:dP8-RLK}; 
    \item $\cD_6$:
 $X$ and $Y$ are isomorphic by a result of Blunk \cite{Blunk}, as explained in 
 \cite[Example 10.2]{Duncan-toric}, 
  see also 
  \cite[Theorem 4.2]{CTKM};   
    \item $C \times C'$: 
    let $C$ and $C'$ correspond to Brauer classes $\ba, \ba' \in \Br(\kk)$, then the atoms of $C \times C'$ are
    \[
    (\kk, 0), (\kk, \ba), 
    (\kk, \ba'), (\kk, \ba + \ba'). 
    \]
    If $\ba + \ba'$ is the class of some conic $C''$ then the surfaces $C \times C'$  $C \times C''$, and $C' \times C''$ are birational by
    \cite[Lemma 11]{Kollar-conics}, \cite[Theorem 1.5]{Trepalin}.
    If $\ba + \ba'$ is not the class of a conic then $C$, $C'$, and the product $C \times C'$ are uniquely defined  by the atoms.
 \end{itemize}
\end{proof}

\begin{theorem}\label{thm:almost-same-atoms-bir-rich}
Let $X$ and $Y$ be birationally rich surfaces. Then $X$ and $Y$ are birational if and only if they have the same atoms, up to removing trivial atoms.
\end{theorem}
\begin{proof}
Birational surfaces have the same atoms, up to removing trivial atoms, by Lemma \ref{lem:bir-non-closed-atoms}.

Conversely, let $X$ and $Y$ satisfy the condition on the atoms. We  assume that not all atoms in $X$ and~$Y$ are trivial, otherwise $X$ and $Y$ are both rational by Theorem \ref{thm:rational-atoms}.
Applying Proposition \ref{prop:rich-irr}(1), we can assume that $X$ and $Y$ are minimal surfaces of one of the types $\cD_9$, $\cD_8$, $\cD_6$, $C \times C'$, and they are of the same type.

In cases $\cD_9$, $\cD_8$, and $C \times C'$, the trivial atoms are all equivalent to $\Db(\kk)$, see Example~\ref{ex:bir-rich-non-closed}. Thus if the nontrivial atoms of $X$ and $Y$ are the same, then in fact, $X$ and $Y$ have the same atoms and the result follows from Proposition \ref{prop:rich-irr}(2).

It remains to prove the result in the case $\cD_6$. Note that by the index formula~\eqref{eq:index-formula}, $\ind(X) = \ind(Y) \in \{1, 2, 3, 6\}$.
    In the case $\ind(X) = \ind(Y)=1$ both $X$ and $Y$ are rational by Lemma \ref{lem:bir-rich-rationality}. In the case $\ind(X) = \ind(Y)=6$ both $X$ and $Y$ have two non-trivial atoms which are the same, and the third atom is $(\kk,0)$, so Proposition \ref{prop:rich-irr}(2) applies. 
    In the remaining cases we construct a Sarkisov link of type II from $Y$ to a surface which has the same atoms as $X$.
    
    Let us assume that $\ind(X) = \ind(Y) = 2$. Then the atoms of $X$ are $(\kk,0)$, $(L_2,0)$, $(L_3, \ba)$
    and atoms of $Y$ are
    $(\kk,0)$, $(L_2',0)$, $(L_3, \ba)$.
    By Example \ref{ex:dP6-splittings}, $L_2$ splits the atom $(L_3, \ba)$, hence by Corollary \ref{cor:splitting-atoms}, $Y_{L_2}$ is rational.
    In particular, there exist closed points on $Y$ with residue field isomorphic to $L_2$. 
    We claim that we can find such a point $y \in Y$ which is in general position so that the blow-up $\Bl_y(Y)$ is again a del Pezzo surface.
    We check the general position condition from Remark \ref{rem:general-pos}.
    Since 
    $\kk$ is an infinite field by Corollary \ref{cor:finite-fields}, we can find such a point $y \in Y$ away from the union of all the  $(-1)$-curves (in fact points of degree $2$ are  not contained in the union of $(-1)$-curves, see \cite[Proposition 5.13(1)]{KY-dP6}). Furthermore, $y$ can not be contained in a $\ol{\kk}$-conic (i.e. a $0$-curve) because such conic would have to be unique. Indeed different conics are either disjoint or intersect in one point, hence this conic would be defined over $\kk$, which contradicts to the assumption that $Y$ has Picard rank one.
    We blow up $y \in Y$ to get a del Pezzo surface $Z$ of degree $4$ and Picard rank two which gives a $(6, 4, 6)$ Sarkisov link of type II     from $Y$ to another del Pezzo surface $X'$ of degree $6$ (see Remark \ref{rem:dPZ-link} and Proposition~\ref{prop:SarkisovLinks}). 
    Let us check that $X'$ and $X$ have the same atoms. Indeed, the blow-up adds an atom $(L_2,0)$ and the blow-down removes the atom $(L_2',0)$. By Proposition \ref{prop:rich-irr}(2), 
    $X$ and $X'$ are isomorphic.
    Therefore $X$ and $Y$ are birational.
    
    The case $\ind(X) = \ind(Y) = 3$ is similar. The atoms of $X$ are $(\kk,0)$, $(L_2,\ba)$, $(L_3, 0)$
    and atoms of $Y$ are
    $(\kk,0)$, $(L_2,\ba)$, $(L_3', 0)$.
    By Example \ref{ex:dP6-splittings}, $L_3$ splits the atom $(L_2, \ba)$, hence by Corollary \ref{cor:splitting-atoms}, $Y_{L_3}$ is rational.
    In particular, there exist closed points on $Y$ with residue field isomorphic to $L_3$. We can find such a closed point $y \in Y$ with residue field $L_3$ away from the union of all the  $(-1)$-curves. This point can not be contained in any conic or $1$-curve defined over $\ol{\kk}$, because this curve would be unique, hence defined over $\kk$ which contradicts to the assumption that $Y$ has Picard rank one. Let us assume that $L_3 \not\cong L_3'$ (otherwise $X$ and $Y$ are isomorphic by Proposition \ref{prop:rich-irr}(2)). We show that no two points of the Galois orbit $y_{\ol{\kk}}$ can be contained in a $\ol{\kk}$-conic which is the final generality  condition that we need to check in order to blow-up $y \in Y$ and construct  a Sarkisov link. If this condition is violated,  the three points of $y_{\ol{\kk}}$ are three vertices of a triangle formed by a Galois orbit of $\ol{\kk}$-conics. 
    Hence the Galois orbit of points is isomorphic to the Galois orbit of conics, and this means that $L_3 \cong L_3'$, because $L_3$ corresponds to the Galois set $y_{\ol{\kk}}$ and $L_3'$ corresponds to the Galois set of the three classes of conics $h_1$, $h_2$, $h_3$ on $Y_{\ol{\kk}}$ by construction of atoms for $Y$ (see Definition \ref{def_standard}). As we assume $L_3 \not\cong L_3'$, the point $y$ is in general position, and the corresponding $(6,3,6)$ Sarkisov link of type II transforms $Y$ to a surface isomorphic to~$X$.
\end{proof}

Without any assumptions on $X$ and $Y$ the  result of Theorem \ref{thm:almost-same-atoms-bir-rich} cannot hold because even in the case $\kk = \bar\kk$
there are non-birational
K-nef surfaces, for example K3 or abelian surfaces, which are derived equivalent. 

\begin{conjecture}
The statement of Theorem \ref{thm:almost-same-atoms-bir-rich} holds for all geometrically rational surfaces.
\end{conjecture}

The cases to be checked are $G$-minimal
del Pezzo surfaces $\cD_d$ and $G$-Mori conic bundles $\cC_d$,
with $d \le 4$.
When both $X$ and $Y$
are minimal del Pezzo surfaces of degree $4$
with the same non-trivial atoms, they are
isomorphic by Theorem \ref{thm:dP4-bir}.

\begin{conjecture}\label{conj:stab-rat}
Let $X$ be a smooth projective surface such that 
$X \times \P^1$ is rational. Then 
$X$ is rational.
\end{conjecture}

The motivation for this conjecture is as follows.
First of all, over an algebraically closed field the conjecture holds by classification of surfaces, hence we can assume that $X$ is geometrically rational.
By Lemma \ref{lem:bir-rich-rationality}, for birationally rich surfaces rationality is the same as stable rationality, so the conjecture holds.
Thus let us consider the case when $X$ is not birationally rich, in particular, not rational.
In this case, by Proposition \ref{prop:bir-rich-non-closed}, 
$X$ has a big atom $\bA$. It is easy to see from construction of the atomic theory that for non-birationally rich geometrically rational surfaces this atom is
not equivalent to an admissible subcategory in $\Db(C)$ for any curve $C$.
Now let us assume that there is an atomic theory in dimension three which extends our atomic theory in dimension two. 
Then existence of the trivial conic bundle $X \times \P^1 \to X$ 
implies that $X \times \P^1$ also has an atom equivalent to $\bA$. On the other hand, this atom cannot appear by smooth blow-ups relating $\P^3$ to $X$ because $\bA$ is not equivalent to an admissible subcategory in a variety of dimension $\le 1$.

\begin{example}
Let $\mathrm{char}(\kk) \ne 2$.
 Consider a smooth affine surface $U\subset \AA^3_{\kk}$ given by
 \[
 x^2 - dy^2 = P(t)
 \]
 where $d \in \kk^* \setminus (\kk^*)^2$ and $P \in \kk[t]$ is an irreducible separable cubic polynomial with Galois group $S_3$ and discriminant $d$.
 It is a particular type of generalized Ch\^atelet surfaces \cite[IV]{CSD-Chatelet}.
 The morphism  $U\to\AA^1=\Spec \kk[t]$ admits a relatively minimal smooth conic bundle compactification $X \to \P^1$ with four singular fibres \cite[Proof of Th\'eor\`eme 2]{BCSD}.
 Thus we have $K_X^2 = 4$, so
  $X$ is not birationally rich, in particular $X$ is irrational.
  
 However, $X$ is stably rational. 
 Namely by
 \cite[Th\'eor\`eme 1]{BCSD} $X \times \P^3$ is rational,
 and it was later shown that
 $X \times \P^2$ is also rational \cite[Corollary 12, Proposition 13]{SB-Chatelet}.
 It is currently unknown if $X \times \P^1$ is rational; 
 Conjecture \ref{conj:stab-rat} predicts that this is not the case.
 \end{example}

%############## APPENDIX #################################################

\appendix

\section{Proof of Proposition~\ref{prop_sarkisov}}
\label{ap_proof_prop}

In this section, we prove Proposition~\ref{prop_sarkisov}, which states:

\begin{proposition}
Let $\chi\colon X_1\dashrightarrow X_2$ be a $G$-Sarkisov link 
 \[\begin{tikzcd}        &Z\ar[dl,swap,"\sigma_1"]\ar[dr,"\sigma_2"]&\\        X_1\ar[d,swap,"\pi_1"]\ar[dashed,rr,"\chi"]&& X_2\ar[d,"\pi_2"]\\ 
        B_1\ar[dr,"\beta_1",swap]&& B_2\ar[dl,"\beta_2"]\\
        &B.&
    \end{tikzcd}\] 
    Then the following SOD-s of $\Db(Z)$ are mutation-equivalent:
    \begin{equation*}
        \langle \cK(\sigma_1), \sigma_1^*\cBstd(X_1/B_1)\rangle \cong
        \langle \cK(\sigma_2), \sigma_2^*\cBstd(X_2/B_2)\rangle.
    \end{equation*}
\end{proposition}

Recall that
  $\cK(\sigma_i) = \ker \sigma_{i,*}$ (as in Definition \ref{def_AgivenBandf}). See Definition \ref{def_standard} for $\cBstd(X_i/B_i)$.

\begin{proof}
We use notation from Section~\ref{section_divisorsonsurfaces}. 
In particular, we denote by $E$, $E_i$ exceptional divisors, by $h$, $h_i$ the $0$-classes and by $H$, $H_i$ the $1$-classes.
We denote the mutation of the first component to the end by $-K$ and the mutation of the last component to the beginning by $K$. We denote by \eqref{lemma_OE} the mutation described in Lemma~\ref{lemma_OE}.
We go through each case of Sarkisov link from Proposition~\ref{prop:SarkisovLinks}.

\textbf{Links of type I}:
$$\xymatrix{ X\ar[d] & Y\ar[l]\ar[d]  \\
\Spec \kk  & \P^1, \ar[l]}$$
where $X$ is a $G$-del Pezzo surface of rank one and $Y/\P^1$ is a $G$-Mori conic bundle.

$\mathbf{(9,8)}$. We have $X\cong \P^2, Y\cong \bF_1$. Recall that $H=E+h$ and $-K=2H+h$ in $\Pic Y$. Two SOD-s on~$Y$ coming from $X$ and from $\P^1$ are 
\begin{equation}
\label{eq_98a}
\langle \cO_E(-1),\cO(-2H),\cO(-H),\cO\rangle,
\end{equation}
\begin{equation}
\label{eq_98b}
\langle \cO(-h-E),\cO(-E),\cO(-h),\cO\rangle
\end{equation}
Applying mutations we get
\begin{multline*}
\text{\eqref{eq_98a}}\overset{\eqref{lemma_OE}}\leadsto
    \langle \langle\cO(-2H)\rangle,\langle\cO(-H)\rangle,\langle\cO(-E)\rangle,\langle\cO\rangle\rangle\overset{-K}\leadsto
    \langle\langle\cO(-H)\rangle,\langle\cO(-E)\rangle,\langle\cO\rangle,\langle\cO(-2H-K)\rangle\rangle=\\
    =
    \langle\langle\cO(-h-E)\rangle,\langle\cO(-E)\rangle,\langle\cO\rangle,\langle\cO(h)\rangle\rangle\overset{L_4}\leadsto
    \text{\eqref{eq_98b}}
\end{multline*}

$\mathbf{(9,5)}$. One has $X\cong \P^2, Y$ is the blow-up of $\P^2$ at an orbit of order $4$. There are $4$ contractions $X\to \P^1$ given by the divisors $h_i=H-E_i$ ($i=1,\ldots,4$) which are not $G$-invariant, and the $G$-invariant contraction $Y\to \P^1$ is  given by the divisor $h_5:=-K-H$ (cf. Lemma~\ref{lemma_dualclasses}).
Two SOD-s on $Y$ coming from $X$ and from $\P^1$ are 
\begin{equation*}
\left\langle 
\left\langle \begin{smallmatrix}
\cO_{E_1}(-1)\\ \ldots \\ \cO_{E_4}(-1)
\end{smallmatrix} \right\rangle,\langle\cO(-2H)\rangle,\langle\cO(-H)\rangle,\langle\cO\rangle\right\rangle,
\end{equation*}
\begin{equation*}
\left\langle
 \langle\cE\rangle,
\left\langle \begin{smallmatrix}
\cO(-h_1)\\ \ldots \\ \cO(-h_4)
\end{smallmatrix} \right\rangle,
\langle\cO(-h_5)\rangle,
\langle\cO\rangle
\right\rangle,
\end{equation*}
by Lemma~\ref{lem:3blockdecomp}, they are  mutation-equivalent. 

$\mathbf{(8,6)}$. One has  $X\cong \bF_0, Y$ is the blow-up of an orbit of order $2$ on $\bF_0$. Let $h_1,h_2$ be the fibres of two fibrations on $\bF_0$, then the $G$-invariant contraction $Y\to \P^1$ is  given by the divisor  $h_3:=-K-h_1-h_2$ (as $h_1+h_2$ is a $2$-class, cf. Lemma~\ref{lemma_dualclasses}).
Two SOD-s on $Y$ coming from $X$ and from $\P^1$ are 
\begin{equation*}
\left\langle 
\left\langle \begin{smallmatrix}
\cO_{E_1}(-1))\\ \cO_{E_2}(-1)
\end{smallmatrix} \right\rangle,
\langle \cO(-h_1-h_2)\rangle,
\left\langle \begin{smallmatrix}
\cO(-h_1)\\ \cO(-h_2)
\end{smallmatrix} \right\rangle,
\langle\cO\rangle\right\rangle,
\end{equation*}
\begin{equation*}
\left\langle 
\left\langle \begin{smallmatrix}
\cO(-H_1)\\ \cO(-H_2)
\end{smallmatrix} \right\rangle,
\left\langle \begin{smallmatrix}
\cO(-h_1)\\ \cO(-h_2)
\end{smallmatrix} \right\rangle,
\langle \cO(-h_3)\rangle,
\langle\cO\rangle
\right\rangle,
\end{equation*}
and they are mutation-equivalent by Lemma~\ref{lem:StandardDecompbirrich}.

$\mathbf{(4,3)}$. Here $Y\to X$ is the blow-up of a point on a del Pezzo surface of degree $4$ with the exceptional divisor $E$ and the fibration $f\colon Y\to \P^1$ is given by divisor $h$ such that $E+h= -K_Y$ (cf. Lemma~\ref{lemma_dualclasses}).
Two SOD-s on $Y$ coming from $X$ and from $\P^1$ are 
\begin{equation}
\label{eq_43_1}
\langle\langle\cO_E(-1)\rangle,\bA, \langle\cO\rangle\rangle 
\end{equation}
and
\begin{equation}
\label{eq_43_2}
\langle\ker f_*,\langle\cO(-h)\rangle, \langle\cO\rangle\rangle. 
\end{equation}
We apply mutations:
$$
\text{\eqref{eq_43_1}}\overset{\eqref{lemma_OE}}{\leadsto}
\langle\bA,\langle\cO(-E)\rangle, \langle\cO\rangle\rangle \overset{R_1}\leadsto
\langle\langle\cO(-E)\rangle, \bA',\langle\cO\rangle\rangle \overset{-K}\leadsto
\langle \bA',\langle\cO\rangle, \langle\cO(h)\rangle\rangle \overset{L_3}\leadsto
\text{\eqref{eq_43_2}}
$$

\textbf{Links of type II over a point:} 
$$
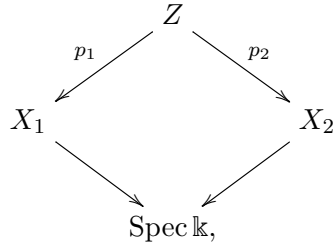
\xymatrix{  & Z\ar[ld]_{p_1}\ar[rd]^{p_2} &  \\
X_1\ar[rd] &&  X_2 \ar[ld]\\
& \Spec \kk, &}$$
where $X_1,X_2$ are $G$-minimal del Pezzo surfaces and $Z$ is their blow-up at some $G$-orbits.

First we consider \textbf{asymmetric} links.

$\mathbf{(9,7,8)}$. Here $X_1=\P^2$, $X_2=\bF_0$ and $Z$ is the blow-up of $X_1$ at an orbit of order $2$ and of $X_2$ at a fixed point.
The SOD-s that we compare are 
\begin{equation}
\label{eq_978a}
\left\langle 
\left\langle \begin{smallmatrix}
\cO_{E_1}(-1)\\ \cO_{E_2}(-1)
\end{smallmatrix} \right\rangle,
\langle\cO(-2H)\rangle,
\langle\cO(-H)\rangle,
\langle\cO\rangle
\right\rangle,
\end{equation}
\begin{equation}
\label{eq_978b}
\left\langle 
\left\langle \cO_{E'}(-1)
\right\rangle,
\langle\cO(-h'_1-h'_2)\rangle,
\left\langle \begin{smallmatrix}
\cO(-h'_1)\\ \cO(-h'_2)
\end{smallmatrix} \right\rangle,
\langle\cO\rangle
\right\rangle.
\end{equation}
We apply mutations using Lemma \ref{lemma_3mutations}:
\begin{multline*}
\text{\eqref{eq_978a}}\overset{L_2}\leadsto
\left\langle 
\langle\cO(-2H+E_1+E_2)\rangle,
\left\langle \begin{smallmatrix}
\cO_{E_1}(-1)\\ \cO_{E_2}(-1)
\end{smallmatrix} \right\rangle,
\langle\cO(-H)\rangle,
\langle\cO\rangle
\right\rangle\overset{R_2}\leadsto\\
\leadsto
\left\langle 
\langle\cO(-2H+E_1+E_2)\rangle,
\langle\cO(-H)\rangle,
\left\langle \begin{smallmatrix}
\cO(-H+E_1)\\ \cO(-H+E_2)
\end{smallmatrix} \right\rangle,
\langle\cO\rangle
\right\rangle=\\
=
\left\langle 
\langle\cO(-h'_1-h'_2)\rangle,
\langle\cO(-H)\rangle,
\left\langle \begin{smallmatrix}
\cO(-h'_1)\\ \cO(-h'_2)
\end{smallmatrix} \right\rangle,
\langle\cO\rangle
\right\rangle\overset{L_2}\leadsto
\text{\eqref{eq_978b}}
\end{multline*}

$\mathbf{(9,4,5)}$. Here $X_1=\P^2$, $X_2$ is a del Pezzo surface of degree $5$  and $Z$ is the blow-up of $X_1$ at an orbit of order $5$ and of $X_2$ at a fixed point. Denote $E=\sum_{i=1}^5E_i$, then one has $E'=2H-E$ and $h'_i=2H-E+E_i$ for $i=1,\ldots,5$.
The SOD-s that we compare are 
\begin{equation}
\label{eq_945a}
\left\langle 
\left\langle \begin{smallmatrix}
\cO_{E_1}(-1)\\ \ldots \\ \cO_{E_5}(-1)
\end{smallmatrix} \right\rangle,
\langle\cO(-2H)\rangle,
\langle\cO(-H)\rangle,
\langle\cO\rangle
\right\rangle,
\end{equation}
\begin{equation}
\label{eq_945b}
\left\langle 
\left\langle \cO_{E'}(-1)
\right\rangle,
\langle\cE\rangle,
\left\langle \begin{smallmatrix}
\cO(-h'_1)\\ \ldots \\ \cO(-h'_5)
\end{smallmatrix} \right\rangle,
\langle\cO\rangle
\right\rangle
\end{equation}

We apply mutations:
\begin{multline*}
\text{\eqref{eq_945a}}\overset{L_2}\leadsto
\left\langle 
\langle\cO(-2H{+}E)\rangle,
\left\langle \begin{smallmatrix}
\cO_{E_1}(-1)\\ \ldots \\ \cO_{E_5}(-1)
\end{smallmatrix} \right\rangle,
\langle\cO(-H)\rangle,
\langle\cO\rangle
\right\rangle\overset{L_2}\leadsto
\left\langle 
\left\langle \begin{smallmatrix}
\cO(-2H+E-E_1)\\ \ldots \\ \cO(-2H+E-E_5)
\end{smallmatrix} \right\rangle,
\langle\cO(-2H{+}E)\rangle,
\langle\cO(-H)\rangle,
\langle\cO\rangle
\right\rangle=\\
=
\left\langle 
\left\langle \begin{smallmatrix}
\cO(-h'_1)\\ \ldots \\ \cO(-h'_5)
\end{smallmatrix} \right\rangle,
\langle\cO(-E')\rangle,
\langle\cO(-H)\rangle,
\langle\cO\rangle
\right\rangle\overset{L_2L_3}\leadsto
\left\langle 
\langle\cE'\rangle,
\left\langle \begin{smallmatrix}
\cO(-h'_1)\\ \ldots \\ \cO(-h'_5)
\end{smallmatrix} \right\rangle,
\langle\cO(-E')\rangle,
\langle\cO\rangle
\right\rangle\overset{\eqref{lemma_OE}}\leadsto
\text{\eqref{eq_945b}}
\end{multline*}

$\mathbf{(8,5,6)}$. Here $X_1=\bF_0$, $X_2$ is a del Pezzo surface of degree $6$  and $Z$ is the blow-up of $X_1$ at an orbit of order $3$ and of $X_2$ at a fixed point. Denote $E=\sum_{i=1}^5E_i$. We have $E'=-(h_1+h_2)-K$, $H'_i=-h_i-K$, and $h'_i=-(h_1+h_2-E_i)-K$ (cf. Lemma~\ref{lemma_dualclasses}).

The SOD-s that we compare are 
\begin{equation*}
\left\langle 
\left\langle \begin{smallmatrix}
\cO_{E_1}(-1)\\ \cO_{E_2}(-1)\\  \cO_{E_3}(-1)
\end{smallmatrix} \right\rangle,
\langle\cO(-h_1-h_2)\rangle,
\left\langle \begin{smallmatrix}
\cO(-h_1)\\ \cO(-h_2)
\end{smallmatrix} \right\rangle,
\langle\cO\rangle
\right\rangle,
\end{equation*}
\begin{equation*}
\left\langle 
\left\langle \cO_{E'}(-1)
\right\rangle,
\left\langle \begin{smallmatrix}
\cO(-H'_1)\\  \cO(-H'_2)
\end{smallmatrix} \right\rangle,
\left\langle \begin{smallmatrix}
\cO(-h'_1)\\ \cO(-h'_2) \\ \cO(-h'_3)
\end{smallmatrix} \right\rangle,
\langle\cO\rangle
\right\rangle,
\end{equation*}
which are exactly \eqref{eq:5-6--1} and~\eqref{eq:5-8--1}.
It is shown in the proof of Lemma~\ref{lem:StandardDecompbirrich} that they are mutation-equivalent.

$\mathbf{(8,3,5)}$. Here $X_1=\bF_0$, $X_2$ is a del Pezzo surface of degree $5$  and $Z$ is the blow-up of $X_1$ at an orbit of order $5$ and of $X_2$ at an orbit of order $2$. We have $E_i'=-h_i-K$ for $i=1,2$ and $h'_i=-E_i-K$ for $i=1,\ldots,5$ (cf. Lemma~\ref{lemma_dualclasses}).
The SOD-s that we compare are 
\begin{equation}
\label{eq_835a}
\left\langle 
\left\langle \begin{smallmatrix}
\cO_{E_1}(-1)\\ \ldots \\ \cO_{E_5}(-1)
\end{smallmatrix} \right\rangle,
\langle\cO(-h_1-h_2)\rangle,
\left\langle \begin{smallmatrix}
\cO(-h_1)\\ \cO(-h_2)
\end{smallmatrix} \right\rangle,
\langle\cO\rangle
\right\rangle,
\end{equation}
\begin{equation}
\label{eq_835b}
\left\langle 
\left\langle \begin{smallmatrix}
\cO_{E_1'}(-1)\\ \cO_{E_2'}(-1)
\end{smallmatrix}
\right\rangle,
\langle\cE\rangle,
\left\langle \begin{smallmatrix}
\cO(-h'_1)\\ \ldots \\ \cO(-h'_5)
\end{smallmatrix} \right\rangle,
\langle\cO\rangle
\right\rangle.
\end{equation}

We apply mutations:
\begin{multline*}
\text{\eqref{eq_835a}}\overset{\eqref{lemma_OE}}\leadsto
\left\langle\langle\cO(-h_1-h_2)\rangle,
\left\langle \begin{smallmatrix}
\cO(-h_1)\\ \cO(-h_2)
\end{smallmatrix} \right\rangle,
\left\langle \begin{smallmatrix}
\cO(-E_1)\\ \ldots \\ \cO(-E_5)
\end{smallmatrix} \right\rangle,
\langle\cO\rangle
\right\rangle\overset{L_4(-K)}\leadsto
\left\langle
\left\langle \begin{smallmatrix}
\cO(-h_1)\\ \cO(-h_2)
\end{smallmatrix} \right\rangle,
\left\langle \begin{smallmatrix}
\cO(-E_1)\\ \ldots \\ \cO(-E_5)
\end{smallmatrix} \right\rangle,
\langle\cL\rangle,
\langle\cO\rangle
\right\rangle=\\
=
\left\langle
\left\langle \begin{smallmatrix}
\cO(E_1'+K)\\ \cO(E_2'+K)
\end{smallmatrix} \right\rangle,
\left\langle \begin{smallmatrix}
\cO(h_1+K)\\ \ldots \\ \cO(h_5+K)
\end{smallmatrix} \right\rangle,
\langle\cL\rangle,
\langle\cO\rangle
\right\rangle
\overset{L_4(-K)}\leadsto
\left\langle
\left\langle \begin{smallmatrix}
\cO(h_1+K)\\ \ldots \\ \cO(h_5+K)
\end{smallmatrix} \right\rangle,
\langle\cL\rangle,
\left\langle \begin{smallmatrix}
\cO_{E'_1}(-1)\\ \cO_{E'_2}(-1)
\end{smallmatrix} \right\rangle,
\langle\cO\rangle
\right\rangle
\overset{L_4(-K)}\leadsto\\
\left\langle
\langle\cL\rangle,
\left\langle \begin{smallmatrix}
\cO_{E'_1}(-1)\\ \cO_{E'_2}(-1)
\end{smallmatrix} \right\rangle,
\left\langle \begin{smallmatrix}
\cO(-h'_1)\\ \ldots \\ \cO(-h'_5)
\end{smallmatrix} \right\rangle,
\langle\cO\rangle
\right\rangle\overset{R_1}\leadsto
\text{\eqref{eq_835b}}\end{multline*}

\medskip
Now we treat \textbf{symmetric} links. 

$\mathbf{(d,1,d)}$ and $\mathbf{(d,2,d)}$.
We consider all cases where $\deg Z=1$ or $2$ uniformly. In these cases, the blow-ups $p_1$ and $p_2$ differ by the Bertini, respectively Geiser involution of $Z$, which we denote by $\sigma$. In particular, $X_1\cong X_2=:X$ as $G$-surfaces. The two SOD-s that we need to compare are 
\begin{gather}
\label{eq_BG1}
\Db(Z)=\langle \ker p_{1*}, p_1^*\bB_1,\ldots,p_1^*\bB_{m}\rangle,\\
\label{eq_BG2}
\Db(Z)=\langle \ker p_{2*}, p_2^*\bB_1,\ldots,p_2^*\bB_{m}\rangle,
\end{gather}
where $\Db(X)=\langle\bB_1,\ldots,\bB_m\rangle$ is the standard SOD (and $m=3$ or $2$ depending on whether $X$ is birationally rich or not). Recall that $\bB_m=\langle\cO_X\rangle$, so that $p_1^*\bB_m=p_2^*\bB_m=\langle\cO_Z\rangle$. 
Denote $\bA:=\cO_Z^{\perp}$, then $\bA=\langle\bA_1,\ldots,\bA_m\rangle$, where $\bA_1=\ker p_{1*}$, and $\bA_i=p_1^*\bB_{i-1}$ for $i=2,\ldots,m$. Let $S_{\bA}$ be the Serre functor of $\bA$.
Recall that $L_i$ denotes the mutation of the $i$-th component of an SOD to the left.  
Iterating~\eqref{eq_SerreMutat1} one gets
\begin{align*}
    (L_2L_3\ldots L_{m})(\langle \bA_1,\ldots\bA_m, \langle\cO_Z\rangle\rangle)&=\langle S_{\bA}(\bA_{m}),\bA_1, \ldots,\bA_{m-1}, \langle\cO_Z\rangle\rangle,\\
    (L_2L_3\ldots L_{m})^2(\langle \bA_1,\ldots\bA_m, \langle\cO_Z\rangle\rangle)&=\langle S_{\bA}(\bA_{m-1}),S_{\bA}(\bA_{m}),\bA_1, \ldots,\bA_{m-2}, \langle\cO_Z\rangle\rangle,\\
    &\ldots
    \\
    (L_2L_3\ldots L_{m})^m(\langle \bA_1,\ldots\bA_m, \langle\cO_Z\rangle\rangle)&=\langle S_{\bA}(\bA_{1}),\ldots,S_{\bA}(\bA_{m}), \langle\cO_Z\rangle\rangle,\\
    (L_2L_3\ldots L_{m})^{km}(\langle \bA_1,\ldots\bA_m, \langle\cO_Z\rangle\rangle)&=\langle S_{\bA}^k(\bA_{1}),\ldots,S_{\bA}^k(\bA_{m}), \langle\cO_Z\rangle\rangle.
\end{align*}
We use Lemma~\ref{lemma_fCY}, saying that $S_{\bA}^3=\sigma^*[5]$ for $\deg Z=1$ and  
$S_{\bA}^2=\sigma^*[3]$ for $\deg Z=2$. 
Therefore, $(L_2L_3\ldots L_{m})^{3m}$ for $\deg Z=1$ and $(L_2L_3\ldots L_{m})^{2m}$ for $\deg Z=2$ takes~\eqref{eq_BG1} to 
$$\langle \sigma^*\ker p_{1*}, \sigma^*p_1^*\bB_1,\ldots,\sigma^*p_1^*\bB_{m-1}, \langle\cO_Z\rangle\rangle,$$
which is~\eqref{eq_BG2} since $p_2 = p_1 \sigma$.

$\mathbf{(9,6,9)}$. Here $Z$ is the blow-up of $X_1=\P^2$ and of $X_2=\P^2$ at orbits of order $3$. We note that by Lemma~\ref{lem:3blockdecomp} the SOD \eqref{eq_vectorbundleE}
\begin{equation}
\label{eq_969a}
\left\langle 
\left\langle \begin{smallmatrix}
\cO_{E_1}(-1)\\ \cO_{E_2}(-1) \\ \cO_{E_3}(-1)
\end{smallmatrix} \right\rangle,
\langle\cO(-2H)\rangle,\langle\cO(-H)\rangle,\langle\cO\rangle\right\rangle
\end{equation}
pulled back from $X_1$ is mutation-equivalent to
\begin{equation}
\label{eq_969b}
\left\langle 
\langle \cO(-H_1)\rangle, \langle\cO(-H_2)\rangle,
\left\langle \begin{smallmatrix}
\cO(-h_1)\\ \cO(-h_2) \\ \cO(-h_3)
\end{smallmatrix} \right\rangle,
\langle\cO\rangle
\right\rangle,
\end{equation}
which is the standard $3$-block SOD on a del Pezzo surface of degree~$6$ 
where the first block is split into two.
Then the result  follows from symmetry after swapping the orthogonal line bundles $\cO(-H_1)$ and $\cO(-H_2)$.

$\mathbf{(9,3,9)}$. Here $X_1,X_2=\P^2$ and $Z$ is the blow-up of $X_1,X_2$ at orbits of order $6$.
Let $H,H'\in\Pic Z$ be the pull-backs of lines under $p_1,p_2$.
As usual, we write $E_i, E'_i$ for the exceptional divisors. 
Denote $E=\sum_{i=1}^6E_i$. One has $H'=-2K-H$ (cf. Lemma~\ref{lemma_dualclasses}) and $E'_i=2H-E+E_i$. The SOD-s that we compare are 
\begin{gather}
\label{eq_939a}
\left\langle 
\left\langle \begin{smallmatrix}
\cO_{E_1}(-1)\\ \ldots \\ \cO_{E_6}(-1)
\end{smallmatrix} \right\rangle,
\langle\cO(-2H)\rangle,
\langle\cO(-H)\rangle,
\langle\cO\rangle
\right\rangle,\\
\label{eq_939b}
\left\langle 
\left\langle \begin{smallmatrix}
\cO_{E'_1}(-1)\\ \ldots \\ \cO_{E'_6}(-1)
\end{smallmatrix} \right\rangle,\langle\cO(-2H')\rangle,\langle\cO(-H')\rangle,\langle\cO\rangle\right\rangle.
\end{gather}

Applying mutations we get
\begin{multline*}
\text{\eqref{eq_939a}}\overset{L_2}\leadsto
\left\langle 
\langle\cO(-2H+E)\rangle,
\left\langle \begin{smallmatrix}
\cO_{E_1}(-1)\\ \ldots  \\ \cO_{E_6}(-1)
\end{smallmatrix} \right\rangle,
\langle\cO(-H)\rangle, \langle\cO\rangle\right\rangle\overset{L_2}\leadsto
\left\langle 
\left\langle \begin{smallmatrix}
\cO(-2L+E-E_1)\\ \ldots \\ \cO(-2L+E-E_6)
\end{smallmatrix} \right\rangle,
\langle\cO(-2H+E)\rangle,
\langle\cO(-H)\rangle, 
\langle\cO\rangle
\right\rangle=\\
=
\left\langle 
\left\langle \begin{smallmatrix}
\cO(-E'_1)\\ \ldots \\ \cO(-E'_6)
\end{smallmatrix} \right\rangle,
\langle\cO(H+K)\rangle,
\langle\cO(-H)\rangle,
\langle\cO\rangle
\right\rangle\overset{L_2L_3}\leadsto
\left\langle 
\langle\cE\rangle,
\left\langle \begin{smallmatrix}
\cO(-E'_1)\\ \ldots \\ \cO(-E'_6)
\end{smallmatrix} \right\rangle,
\langle\cO(H+K)\rangle, 
\langle\cO\rangle
\right\rangle\overset{R_3}=\\
=
\left\langle 
\langle\cE\rangle,
\left\langle \begin{smallmatrix}
\cO(-E'_1)\\ \ldots \\ \cO(-E'_6)
\end{smallmatrix} \right\rangle,
\langle\cO\rangle,
\langle\cO(H+K)\rangle
\right\rangle\overset{K}\leadsto
\left\langle 
\langle\cO(H+2K)\rangle,
\langle\cE\rangle,
\left\langle \begin{smallmatrix}
\cO(-E'_1)\\ \ldots \\ \cO(-E'_6)
\end{smallmatrix} \right\rangle,
\langle\cO\rangle
\right\rangle=\\
=
\left\langle 
\langle\cO(-H')\rangle,
\langle\cE\rangle,
\left\langle \begin{smallmatrix}
\cO(-E'_1)\\ \ldots \\ \cO(-E'_6)
\end{smallmatrix} \right\rangle,
\langle\cO\rangle
\right\rangle\overset{\eqref{lemma_OE}}\leadsto
\left\langle 
\left\langle \begin{smallmatrix}
\cO_{E'_1}(-1)\\ \ldots \\ \cO{E'_6}(-1)
\end{smallmatrix} \right\rangle,\langle\cO(-H')\rangle,
\langle\cE\rangle,
\langle\cO\rangle
\right\rangle\overset{L_3}\leadsto
\text{\eqref{eq_939b}}.
\end{multline*}
Here  we used that $-(H+K)$ is a $(-2)$-class by Lemma~\ref{lemma_dualclasses} and hence the
sheaf $\cO(H+K)$ is orthogonal to $\cO$ by Lemma~\ref{lemma_classes}(1).

$\mathbf{(8,4,8)}$. Here $X_1,X_2=\bF_0$ and $Z$ is the blow-up of $X_1,X_2$ at orbits of order $4$.
Let $h_1,h_2\in\Pic Z$ (resp. $h'_1,h'_2\in\Pic Z$) be the pull-backs of fibres under $p_1$ (resp. $p_2$).
Denote $E=\sum_{i=1}^4E_i$. One has $h'_i=-K-h_i$ (cf. Lemma~\ref{lemma_dualclasses}) and $E'_i=h_1+h_2-E+E_i$. The SOD-s that we compare are 
\begin{gather}
\label{eq_848a}
\left\langle 
\left\langle \begin{smallmatrix}
\cO_{E_1}(-1)\\ \ldots \\ \cO_{E_4}(-1)
\end{smallmatrix} \right\rangle,
\langle\cO(-h_1-h_2)\rangle,
\left\langle \begin{smallmatrix}
\cO(-h_1)\\ \cO(-h_2)
\end{smallmatrix} \right\rangle,
\langle\cO\rangle
\right\rangle,\\
\label{eq_848b}
\left\langle 
\left\langle \begin{smallmatrix}
\cO_{E'_1}(-1)\\ \ldots \\ \cO_{E'_4}(-1)
\end{smallmatrix} \right\rangle,
\langle\cO(-h'_1-h'_2)\rangle,
\left\langle \begin{smallmatrix}
\cO(-h'_1)\\ \cO(-h'_2)
\end{smallmatrix} \right\rangle,
\langle\cO\rangle
\right\rangle.
\end{gather}

Applying mutations analogous to the case $(9,3,9)$ we get
\begin{multline*}
\text{\eqref{eq_848a}}\overset{L_2}\leadsto
\left\langle 
\langle\cO(-h_1-h_2+E)\rangle,
\left\langle \begin{smallmatrix}
\cO_{E_1}(-1)\\ \ldots  \\ \cO_{E_4}(-1)
\end{smallmatrix} \right\rangle,
\left\langle \begin{smallmatrix}
\cO(-h_1)\\ \cO(-h_2)
\end{smallmatrix} \right\rangle,
\langle\cO\rangle
\right\rangle\overset{L_2}\leadsto\\
\leadsto
\left\langle 
\left\langle \begin{smallmatrix}
\cO(-h_1-h_2+E-E_1)\\ \ldots  \\ \cO(-h_1-h_2+E-E_4)
\end{smallmatrix} \right\rangle,
\langle\cO(-h_1-h_2+E)\rangle,
\left\langle \begin{smallmatrix}
\cO(-h_1)\\ \cO(-h_2)
\end{smallmatrix} \right\rangle,
\langle\cO\rangle
\right\rangle=\\
=
\left\langle 
\left\langle \begin{smallmatrix}
\cO(-E'_1)\\ \ldots  \\ \cO(-E'_4)
\end{smallmatrix} \right\rangle,
\langle\cO(h_1+h_2+K)\rangle,
\left\langle \begin{smallmatrix}
\cO(-h_1)\\ \cO(-h_2)
\end{smallmatrix} \right\rangle,
\langle\cO\rangle
\right\rangle\overset{K R_3}\leadsto\\
\leadsto
\left\langle 
\left\langle \begin{smallmatrix}
\cO(h_1+K)\\ \cO(h_2+K)
\end{smallmatrix} \right\rangle,
\left\langle \begin{smallmatrix}
\cO(-E'_1)\\ \ldots  \\ \cO(-E'_4)
\end{smallmatrix} \right\rangle,
\langle\cO(h_1+h_2+K)\rangle,
\langle\cO\rangle
\right\rangle\overset{R_3}\leadsto\\
\leadsto
\left\langle 
\left\langle \begin{smallmatrix}
\cO(-h'_1)\\ \cO(-h'_2)
\end{smallmatrix} \right\rangle,
\left\langle \begin{smallmatrix}
\cO(-E'_1)\\ \ldots  \\ \cO(-E'_4)
\end{smallmatrix} \right\rangle,
\langle\cO\rangle,
\langle\cO(h_1+h_2+K)\rangle
\right\rangle\overset{K}\leadsto\\
\leadsto
\left\langle 
\langle\cO(h_1+h_2+2K)\rangle,
\left\langle \begin{smallmatrix}
\cO(-h'_1)\\ \cO(-h'_2)
\end{smallmatrix} \right\rangle,
\left\langle \begin{smallmatrix}
\cO(-E'_1)\\ \ldots  \\ \cO(-E'_4)
\end{smallmatrix} \right\rangle,
\langle\cO\rangle
\right\rangle
\overset{\eqref{lemma_OE}}\leadsto\\
\leadsto
\left\langle 
\left\langle \begin{smallmatrix}
\cO_{E'_1}(-1)\\ \ldots  \\ \cO{E'_4}(-1)
\end{smallmatrix} \right\rangle,
\langle\cO(-h'_1-h'_2)\rangle,
\left\langle \begin{smallmatrix}
\cO(-h'_1)\\ \cO(-h'_2)
\end{smallmatrix} \right\rangle,
\langle\cO\rangle
\right\rangle
=
\text{\eqref{eq_848b}}.
\end{multline*}
Here we used that $-(h_1+h_2+K)$ is a $(-2)$-class and the sheaf $\cO(h_1+h_2+K)$ is orthogonal to $\cO$ by Lemmas~\ref{lemma_dualclasses} and \ref{lemma_classes}(1).

$\mathbf{(6,4,6)}$. Here $X_1,X_2$ are del Pezzo surfaces of degree $6$ and $Z$ is the blow-up of $X_1,X_2$ at orbits of order $2$.
One has $h'_i=-K-h_i, H'_i=-K-E_i, E'_i=-K-H_i$ in $\Pic Z$ (cf. Lemma~\ref{lemma_dualclasses}). The SOD-s that we compare are 
\begin{gather}
\label{eq_646a}
\left\langle 
\left\langle \begin{smallmatrix}
\cO_{E_1}(-1)\\  \cO_{E_2}(-1)
\end{smallmatrix} \right\rangle,
\left\langle \begin{smallmatrix}
\cO(-H_1)\\  \cO(-H_2)
\end{smallmatrix} \right\rangle,
\left\langle \begin{smallmatrix}
\cO(-h_1)\\ \cO(-h_2) \\ \cO(-h_3)
\end{smallmatrix} \right\rangle,
\langle\cO\rangle
\right\rangle,\\
\label{eq_646b}
\left\langle 
\left\langle \begin{smallmatrix}
\cO_{E'_1}(-1)\\  \cO_{E'_2}(-1)
\end{smallmatrix} \right\rangle,
\left\langle \begin{smallmatrix}
\cO(-H'_1)\\  \cO(-H'_2)
\end{smallmatrix} \right\rangle,
\left\langle \begin{smallmatrix}
\cO(-h'_1)\\ \cO(-h'_2) \\ \cO(-h'_3)
\end{smallmatrix} \right\rangle,
\langle\cO\rangle
\right\rangle.
\end{gather}

Applying mutations we get
\begin{multline*}
\text{\eqref{eq_646a}}\overset{L_4(-K)}\leadsto
\left\langle 
\left\langle \begin{smallmatrix}
\cO(-H_1)\\  \cO(-H_2)
\end{smallmatrix} \right\rangle
\left\langle \begin{smallmatrix}
\cO(-h_1)\\ \cO(-h_2) \\ \cO(-h_3)
\end{smallmatrix} \right\rangle,
\left\langle \begin{smallmatrix}
\cO(-E_1)\\  \cO(-E_2)
\end{smallmatrix} \right\rangle,
\langle\cO\rangle
\right\rangle=
\left\langle 
\left\langle \begin{smallmatrix}
\cO(E'_1+K)\\  \cO(E'_2+K)
\end{smallmatrix} \right\rangle
\left\langle \begin{smallmatrix}
\cO(-h_1)\\ \cO(-h_2) \\ \cO(-h_3)
\end{smallmatrix} \right\rangle,
\left\langle \begin{smallmatrix}
\cO(-E_1)\\  \cO(-E_2)
\end{smallmatrix} \right\rangle,
\langle\cO\rangle
\right\rangle\overset{L_4(-K)}\leadsto\\
\leadsto
\left\langle 
\left\langle \begin{smallmatrix}
\cO(-h_1)\\ \cO(-h_2) \\ \cO(-h_3)
\end{smallmatrix} \right\rangle,
\left\langle \begin{smallmatrix}
\cO(-E_1)\\  \cO(-E_2)
\end{smallmatrix} \right\rangle,
\left\langle \begin{smallmatrix}
\cO_{E'_1}(-1)\\  \cO_{E'_2}(-1)
\end{smallmatrix} \right\rangle,
\langle\cO\rangle
\right\rangle=
\left\langle 
\left\langle \begin{smallmatrix}
\cO(h'_1+K)\\ \cO(h'_2+K) \\ \cO(h'_3+K)
\end{smallmatrix} \right\rangle,
\left\langle \begin{smallmatrix}
\cO(-E_1)\\  \cO(-E_2)
\end{smallmatrix} \right\rangle,
\left\langle \begin{smallmatrix}
\cO_{E'_1}(-1)\\  \cO_{E'_2}(-1)
\end{smallmatrix} \right\rangle,
\langle\cO\rangle
\right\rangle\overset{L_4(-K)}\leadsto\\
\leadsto
\left\langle 
\left\langle \begin{smallmatrix}
\cO(-E_1)\\  \cO(-E_2)
\end{smallmatrix} \right\rangle,
\left\langle \begin{smallmatrix}
\cO_{E'_1}(-1)\\  \cO_{E'_2}(-1)
\end{smallmatrix} \right\rangle,
\left\langle \begin{smallmatrix}
\cO(-h'_1)\\ \cO(-h'_2) \\ \cO(-h'_3)
\end{smallmatrix} \right\rangle,
\langle\cO\rangle
\right\rangle\overset{R_1}\leadsto
\text{\eqref{eq_646b}} 
\end{multline*}

$\mathbf{(6,3,6)}$. Here $X_1,X_2$ are del Pezzo surfaces of degree $6$ and $Z$ is the blow-up of $X_1,X_2$ at orbits of order $3$.
One has $h'_i=-K-E_i, E'_i=-K-h_i, H'_i=-2K-H_i$ in $\Pic Z$ by Lemma~\ref{lemma_dualclasses}. The SOD-s that we compare are 
\begin{gather}
\label{eq_636a}
\left\langle 
\left\langle \begin{smallmatrix}
\cO_{E_1}(-1)\\  \cO_{E_2}(-1) \\  \cO_{E_3}(-1)
\end{smallmatrix} \right\rangle,
\left\langle \begin{smallmatrix}
\cO(-H_1)\\  \cO(-H_2)
\end{smallmatrix} \right\rangle,
\left\langle \begin{smallmatrix}
\cO(-h_1)\\ \cO(-h_2) \\ \cO(-h_3)
\end{smallmatrix} \right\rangle,
\langle\cO\rangle
\right\rangle,\\
\label{eq_636b}
\left\langle 
\left\langle \begin{smallmatrix}
\cO_{E'_1}(-1)\\  \cO_{E'_2}(-1) \\  \cO_{E'_3}(-1)
\end{smallmatrix} \right\rangle,
\left\langle \begin{smallmatrix}
\cO(-H'_1)\\  \cO(-H'_2)
\end{smallmatrix} \right\rangle,
\left\langle \begin{smallmatrix}
\cO(-h'_1)\\ \cO(-h'_2) \\ \cO(-h'_3)
\end{smallmatrix} \right\rangle,
\langle\cO\rangle
\right\rangle.
\end{gather}

Applying mutations we get
\begin{multline*}
\text{\eqref{eq_636a}}\overset{L_4(-K)}\leadsto
\left\langle 
\left\langle \begin{smallmatrix}
\cO(-H_1)\\  \cO(-H_2)
\end{smallmatrix} \right\rangle
\left\langle \begin{smallmatrix}
\cO(-h_1)\\ \cO(-h_2) \\ \cO(-h_3)
\end{smallmatrix} \right\rangle,
\left\langle \begin{smallmatrix}
\cO(-E_1)\\  \cO(-E_2) \\ \cO(-E_3)
\end{smallmatrix} \right\rangle,
\langle\cO\rangle
\right\rangle\overset{(-K)}\leadsto\\
\leadsto
\left\langle 
\left\langle \begin{smallmatrix}
\cO(-h_1)\\ \cO(-h_2) \\ \cO(-h_3)
\end{smallmatrix} \right\rangle,
\left\langle \begin{smallmatrix}
\cO(-E_1)\\  \cO(-E_2) \\  \cO(-E_3)
\end{smallmatrix} \right\rangle,
\langle\cO\rangle,
\left\langle \begin{smallmatrix}
\cO(-H_1-K)\\  \cO(-H_2-K)
\end{smallmatrix} \right\rangle
\right\rangle\overset{L_4}\leadsto
\left\langle 
\left\langle \begin{smallmatrix}
\cO(-h_1)\\ \cO(-h_2) \\ \cO(-h_3)
\end{smallmatrix} \right\rangle,
\left\langle \begin{smallmatrix}
\cO(-E_1)\\  \cO(-E_2) \\  \cO(-E_3)
\end{smallmatrix} \right\rangle,
\left\langle \begin{smallmatrix}
\cO(-H_1-K)\\  \cO(-H_2-K)
\end{smallmatrix} \right\rangle,
\langle\cO\rangle
\right\rangle=\\
=\left\langle 
\left\langle \begin{smallmatrix}
\cO(E'_1+K)\\ \cO(E'_2+K) \\ \cO(E'_3+K)
\end{smallmatrix} \right\rangle,
\left\langle \begin{smallmatrix}
\cO(h'_1+K)\\  \cO(h'_2+K) \\  \cO(h'_3+K)
\end{smallmatrix} \right\rangle,
\left\langle \begin{smallmatrix}
\cO(-H_1-K)\\  \cO(-H_2-K)
\end{smallmatrix} \right\rangle,
\langle\cO\rangle
\right\rangle\overset{L_4(-K)}\leadsto\\
\leadsto
\left\langle 
\left\langle \begin{smallmatrix}
\cO(h'_1+K)\\  \cO(h'_2+K) \\  \cO(h'_3+K)
\end{smallmatrix} \right\rangle,
\left\langle \begin{smallmatrix}
\cO(-H_1-K)\\  \cO(-H_2-K)
\end{smallmatrix} \right\rangle,
\left\langle \begin{smallmatrix}
\cO_{E'_1}(-1)\\ \cO_{E'_2}(-1) \\ \cO_{E'_3}(-1)
\end{smallmatrix} \right\rangle,
\langle\cO\rangle
\right\rangle\overset{L_4(-K)}\leadsto\\
\leadsto\left\langle 
\left\langle \begin{smallmatrix}
\cO(-H_1-K)\\  \cO(-H_2-K)
\end{smallmatrix} \right\rangle,
\left\langle \begin{smallmatrix}
\cO_{E'_1}(-1)\\ \cO_{E'_2}(-1) \\ \cO_{E'_3}(-1)
\end{smallmatrix} \right\rangle,
\left\langle \begin{smallmatrix}
\cO(-h'_1)\\  \cO(-h'_2) \\  \cO(-h'_3)
\end{smallmatrix} \right\rangle,
\langle\cO\rangle
\right\rangle\overset{R_1}\leadsto
\text{\eqref{eq_636b}}.
\end{multline*}
As before, we use that $-(H_i+K)$ are $(-2)$-classes and the line bundles $\cO,\cO(-H_1-K), \cO(-H_2-K)$ are pairwise orthogonal by Lemmas~\ref{lemma_dualclasses} and \ref{lemma_classes}(1).

\textbf{Links of type II over a curve: }
$$
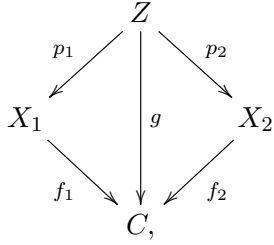
\xymatrix{  & Z\ar[ld]_{p_1}\ar[rd]^{p_2}\ar[dd]^g &  \\
X_1\ar[rd]_{f_1} &&  X_2 \ar[ld]^{f_2}\\
& C, &}
$$
where $X_1,X_2$ are $G$-Mori conic bundles over $C$ and $Z$ is their blow-up at $G$-orbits.

The fibration $X_1\to C$  
may have singular fibres, and $p_1\colon Z\to X_1$ is the blow-up of some orbit $\{x_1,\ldots,x_n\}$ consisting of points in distinct non-singular fibres. Let $E_1,\ldots,E_n\subset Z$ be the exceptional divisors of $p_1$, and let $E'_1,\ldots,E'_n$ be the strict transforms of the corresponding fibres of $f_1$.  Then $p_2\colon Z\to X_2$ is the blow-down of $E'_1,\ldots,E'_n$.

We first show that the two SOD-s

\begin{gather}
\label{eq_EECC1}
\Db(Z)=\left\langle \left\langle\begin{smallmatrix}\cO_{E_1}(-1)\\ \ldots \\\cO_{E_n}(-1)\end{smallmatrix}\right\rangle, p_1^*(\ker f_{1*}),p_1^*f_1^*\Db(C)\right\rangle,\\
\label{eq_EECC2}
\Db(Z)=\left\langle \left\langle\begin{smallmatrix}\cO_{E'_1}(-1)\\ \ldots \\\cO_{E'_n}(-1)\end{smallmatrix}\right\rangle, p_2^*(\ker f_{2*}),p_2^*f_2^*\Db(C)\right\rangle
\end{gather}
are mutation-equivalent.
For this, note first that $p_1^*f_1^*\Db(C)=p_2^*f_2^*\Db(C)=g^*\Db(C)$ since $g=f_1p_1=f_2p_2$.
We will use that the left mutation satisfies
$$L_{g^*\Db(C)}(\cO_{E_i})\cong \cO_{E'_i}(-1)[1],$$
which is true since there is an exact sequence 
$$0\to \cO_{E'_i}(-1)\to \cO_{E_i\cup E'_i}\to \cO_{E_i}\to 0$$
and $\cO_{E_i\cup E'_i}\cong g^*(\cO_{pt})\in g^*\Db(C)$ holds (recall that $g^*$ denotes the derived pull-back). 
Now, the following mutations yield the claim: 
\begin{equation}
\label{eq_wisemove}
\text{\eqref{eq_EECC1}}\overset{-K}\leadsto 
\left\langle p_1^*(\ker f_{1*}),g^*\Db(C), \left\langle\begin{smallmatrix}\cO_{E_1}\\ \ldots \\\cO_{E_n}\end{smallmatrix}\right\rangle\right\rangle \overset{L_3}\leadsto 
\left\langle p_1^*(\ker f_{1*}), \left\langle\begin{smallmatrix}\cO_{E'_1}(-1)\\ \ldots \\\cO_{E'_n}(-1)\end{smallmatrix}\right\rangle, g^*\Db(C)\right\rangle \overset{R_1}\leadsto \text{\eqref{eq_EECC2}}
\end{equation}
Therefore, in the case when $X$ is \textbf{not $G$-birationally rich}, the statement follows directly  because the two SOD-s \eqref{eq_EECC1} and \eqref{eq_EECC2} are exactly the atomic SOD-s 
$\langle \cK(p_i), p_i^*\cBstd(X_i/C)\rangle$ for $i=1,2$.

Assume now that $X$ is \textbf{$G$-birationally rich}. In this case the components $p_i^*(\ker f_{i*})$ of $\Db(Z)$ ($i=1,2$) split into two atoms (see Definition~\ref{def_standard}) and we should take extra care of them. There are three cases to consider: degree $5,6,8$.

\textbf{Degree of $X_1$ and $X_2$ is $5$.}
The pull-back of the standard atomic SOD-s for $X_1/\P^1$, $X_2/\P^1$ to $Z$ are 
\begin{gather}
\label{eq_EHHHH1}
\Db(Z) = \left\langle 
\left\langle \begin{smallmatrix}
\cO_{E_1}(-1)\\ \ldots \\ \cO_{E_n}(-1)
\end{smallmatrix} \right\rangle,
\langle \cE \rangle,
\left\langle \begin{smallmatrix}
\cO(-h_1)\\ \cO(-h_2)\\ \cO(-h_3)\\ \cO(-h_4)
\end{smallmatrix} \right\rangle,
\langle \cO(-h_5) \rangle,
\langle\cO\rangle
\right\rangle,\\
\label{eq_EHHHH2}
\Db(Z) = \left\langle 
\left\langle \begin{smallmatrix}
\cO_{E'_1}(-1)\\ \ldots \\ \cO_{E'_n}(-1)
\end{smallmatrix} \right\rangle,
\left\langle \cE' \right\rangle,
\left\langle \begin{smallmatrix}
\cO(-h'_1)\\ \cO(-h'_2)\\ \cO(-h'_3)\\ \cO(-h'_4)
\end{smallmatrix} \right\rangle,
\langle \cO(-h'_5) \rangle,
\langle\cO\rangle
\right\rangle,
\end{gather}
where fibration $X_1\to \P^1$, $X_2\to \P^1$ is given by the divisors $h_5$, $h'_5$, which coincide when we pull them back to $Z$, and the two last components generate $f^*\Db(\P^1)$. The first two mutations from~\eqref{eq_wisemove}
transform~\eqref{eq_EHHHH1} into 
\begin{equation}
\label{eq_EHHHH3}
\left\langle 
\langle \cE \rangle,
\left\langle \begin{smallmatrix}
\cO(-h_1)\\ \cO(-h_2)\\ \cO(-h_3)\\ \cO(-h_4)
\end{smallmatrix} \right\rangle,
\left\langle \begin{smallmatrix}
\cO_{E'_1}(-1)\\ \ldots \\ \cO_{E'_n}(-1)
\end{smallmatrix} \right\rangle,
\langle \cO(-h_5) \rangle,
\langle\cO\rangle
\right\rangle.
\end{equation}
Denote $E':=\sum_{j=1}^nE'_j$, then the right mutation $R_1R_2$ takes~\eqref{eq_EHHHH3} to 
\begin{equation}
\label{eq_EHHHH4}
\left\langle 
\left\langle \begin{smallmatrix}
\cO_{E'_1}(-1)\\ \ldots \\ \cO_{E'_n}(-1)
\end{smallmatrix} \right\rangle,
\langle \cE^+ \rangle,
\left\langle \begin{smallmatrix}
\cO(-h_1-E')\\ \cO(-h_2-E')\\ \cO(-h_3-E')\\ \cO(-h_4-E')
\end{smallmatrix} \right\rangle,
\langle \cO(-h_5) \rangle,
\langle\cO\rangle
\right\rangle,
\end{equation}
where $\cE^+$ is a vector bundle of rank two. Now we are left to check that some mutations of the pair
\begin{equation}
\label{eq_EHHHH5}
\langle \cE^+ \rangle, \left\langle \begin{smallmatrix}
\cO(-h_1-E')\\ \cO(-h_2-E')\\ \cO(-h_3-E')\\ \cO(-h_4-E')
\end{smallmatrix} \right\rangle
\end{equation}
is
\begin{equation}
\label{eq_EHHHH6}    
\langle \cE' \rangle,
\left\langle \begin{smallmatrix}
\cO(-h'_1)\\ \cO(-h'_2)\\ \cO(-h'_3)\\ \cO(-h'_4)
\end{smallmatrix} \right\rangle.
\end{equation}
It is easy to compute that~\eqref{eq_EHHHH5} and~\eqref{eq_EHHHH6} are strong exceptional collections, that is, the Ext-groups between objects in the collections are concentrated in degree zero.
Furthermore their endomorphism algebras are isomorphic to the path algebra of the  affine quiver $\tD_4$  of type $D_4$ (see Figure~\ref{fig:quivers}). It follows using \cite[\S5, \S6]{Bondal-algebras} that the  category $\cT$ generated by~\eqref{eq_EHHHH5} (or~\eqref{eq_EHHHH6})  is equivalent to $\Db(\modd \kk \tD_4 )$. Let $\phi$ (resp. $\phi'$) be the equivalence 
$\Db(\modd \kk \tD_4 )\to \cT$ sending indecomposable projective modules to objects of~\eqref{eq_EHHHH5} (resp. of~\eqref{eq_EHHHH6}). Then $(\phi')^{-1}\phi$ is an autoequivalence of $\Db(\modd \kk\tD_4)$. They are described in~\cite{MiyachiYekutieli}, in particular, up to some iteration of the Serre functor and a shift,    $(\phi')^{-1}\phi$ sends indecomposable projective modules to indecomposable projective modules. It follows that some iteration $S^N_{\cT}$, $N\in\Z$, of the Serre functor on $\cT$ sends~\eqref{eq_EHHHH5} to~\eqref{eq_EHHHH6}: 
$$S_{\cT}^N\left(\left\langle \begin{smallmatrix}
\cO(-h_1-E')\\ \cO(-h_2-E')\\ \cO(-h_3-E')\\ \cO(-h_4-E')
\end{smallmatrix} \right\rangle\right)= \left\langle \begin{smallmatrix}
\cO(-h'_1)\\ \cO(-h'_2)\\ \cO(-h'_3)\\ \cO(-h'_4)
\end{smallmatrix} \right\rangle.$$ 
It follows that~\eqref{eq_EHHHH5} mutates to~\eqref{eq_EHHHH6} as needed by 
$L_2^{2N}$, see~\eqref{eq_SerreMutat1}.

\textbf{Degree of $X_1$ and $X_2$ is $6$.}
The pull-back of the standard atomic SOD-s for $X_1/\P^1$, $X_2/\P^1$ to $Z$ are 
\begin{gather}
\label{eq_hhHHH1}
\Db(Z) = \left\langle 
\left\langle \begin{smallmatrix}
\cO_{E_1}(-1)\\ \ldots \\ \cO_{E_n}(-1)
\end{smallmatrix} \right\rangle,
\left\langle \begin{smallmatrix}
\cO(-H_1)\\ \cO(-H_2)
\end{smallmatrix} \right\rangle,
\left\langle \begin{smallmatrix}
\cO(-h_1)\\ \cO(-h_2)
\end{smallmatrix} \right\rangle,
\langle \cO(-h_3) \rangle,
\langle\cO\rangle
\right\rangle,\\
\label{eq_hhHHH2}
\Db(Z) = \left\langle 
\left\langle \begin{smallmatrix}
\cO_{E'_1}(-1)\\ \ldots \\ \cO_{E'_n}(-1)
\end{smallmatrix} \right\rangle,
\left\langle \begin{smallmatrix}
\cO(-H'_1)\\ \cO(-H'_2)
\end{smallmatrix} \right\rangle,
\left\langle \begin{smallmatrix}
\cO(-h'_1)\\ \cO(-h'_2)
\end{smallmatrix} \right\rangle,
\langle \cO(-h'_3) \rangle,
\langle\cO\rangle
\right\rangle,
\end{gather}
where fibrations $X_1\to \P^1$, $X_2\to \P^1$ are given by the divisors $h_3$, $h'_3$, which coincide when we lift them  to $Z$, and two last components generate $f^*\Db(\P^1)$. First two mutations from~\eqref{eq_wisemove}
transform~\eqref{eq_hhHHH1} into 
\begin{equation}
\label{eq_hhHHH3}
\left\langle 
\left\langle \begin{smallmatrix}
\cO(-H_1)\\ \cO(-H_2)
\end{smallmatrix} \right\rangle,
\left\langle \begin{smallmatrix}
\cO(-h_1)\\ \cO(-h_2)
\end{smallmatrix} \right\rangle,
\left\langle \begin{smallmatrix}
\cO_{E'_1}(-1)\\ \ldots \\ \cO_{E'_n}(-1)
\end{smallmatrix} \right\rangle,
\langle \cO(-h_3) \rangle,
\langle\cO\rangle
\right\rangle.
\end{equation}
Denote $E':=\sum_{j=1}^nE'_j$, then the right mutation $R_1R_2$ takes~\eqref{eq_hhHHH3}  to 
\begin{equation*}
\left\langle 
\left\langle \begin{smallmatrix}
\cO_{E'_1}(-1)\\ \ldots \\ \cO_{E'_n}(-1)
\end{smallmatrix} \right\rangle,
\left\langle \begin{smallmatrix}
\cO(-H_1-E')\\ \cO(-H_2-E')
\end{smallmatrix} \right\rangle,
\left\langle \begin{smallmatrix}
\cO(-h_1-E')\\ \cO(-h_2-E')
\end{smallmatrix} \right\rangle,
\langle \cO(-h_3) \rangle,
\langle\cO\rangle
\right\rangle.
\end{equation*}
We claim that the SOD 
\begin{equation*}
\left\langle\left\langle \begin{smallmatrix}
\cO(-H_1-E')\\ \cO(-H_2-E')
\end{smallmatrix} \right\rangle,
\left\langle \begin{smallmatrix}
\cO(-h_1-E')\\ \cO(-h_2-E')
\end{smallmatrix} \right\rangle
\right\rangle
\end{equation*}
is mutation-equivalent to the SOD
\begin{equation*}
\left\langle
\left\langle \begin{smallmatrix}
\cO(-H'_1)\\ \cO(-H'_2)
\end{smallmatrix} \right\rangle,
\left\langle \begin{smallmatrix}
\cO(-h'_1)\\ \cO(-h'_2)
\end{smallmatrix} \right\rangle
\right\rangle
\end{equation*}
which is pulled back from $X_2$ and conclude the proof.
This can be shown as in the degree $5$ case replacing the quiver $\tD_4$ by the quiver $\tA_3$ with alternating orientation (see Figure~\ref{fig:quivers}).

\tikzcdset{scale cd/.style={every label/.append style={scale=#1},
    cells={nodes={scale=#1}}},
    }
\begin{figure}
    \centering
    \begin{minipage}{0.3\textwidth}
         \begin{tikzcd}[scale cd= 0.7, sep=small]
        & \cO(-h_1') & \\
        \cO(-h_2')& \cE'\ar[u]\ar[d]\ar[r]\ar[l] & \cO(-h_3')\\
        &\cO(-h_4')&
    \end{tikzcd}    
    \end{minipage}
    \begin{minipage}{0.3\textwidth}
        \begin{tikzcd}[scale cd= 0.7, sep=small]
    \cO(-H_1')\ar[d]\ar[r]&\cO(-h_1')\\
    \cO(-h_2')&\cO(-H_2')\ar[l]\ar[u]\\
    \end{tikzcd} 
    \end{minipage}  
    \caption{Affine quivers $\tilde D_4$ and $\tilde A_3$}
    \label{fig:quivers}
\end{figure}

\textbf{Degree of $X_1$ and $X_2$ is $8$.}
The pull-backs of the standard atomic SOD-s from $X_1,X_2$ to $Z$ are (where $X_1,X_2$ are Hirzebruch surfaces)
\begin{gather}
\label{eq_EECC3}
\Db(Z)=\left\langle \left\langle \begin{smallmatrix}
\cO_{E_1}(-1)\\ \cO_{E_n}(-1)
\end{smallmatrix} \right\rangle, \langle\cO(-s_1-h)\rangle,\langle\cO(-s_1)\rangle,\langle\cO(-h)\rangle,\langle\cO\rangle\right\rangle,\\
\label{eq_EECC4}
\Db(Z)=\left\langle \left\langle \begin{smallmatrix}
\cO_{E'_1}(-1)\\ \cO_{E'_n}(-1)
\end{smallmatrix} \right\rangle, \langle\cO(-s_2-h)\rangle,\langle\cO(-s_2)\rangle,\langle\cO(-h)\rangle,\langle\cO\rangle\right\rangle,
\end{gather}
where $s_1,s_2$ are pull-backs of sections of $\P^1$-bundles $X_1,X_2\to \P^1$ with $s_i^2\le 0$.
Mutations~\eqref{eq_wisemove} transform~\eqref{eq_EECC3} into 
\begin{equation}
\label{eq_EECC5} 
\left\langle \left\langle \begin{smallmatrix}
\cO_{E'_1}(-1)\\ \cO_{E'_n}(-1)
\end{smallmatrix} \right\rangle, \langle\cL_1\rangle,\langle\cL_2\rangle,\langle\cO(-h)\rangle,\langle\cO\rangle\right\rangle,
\end{equation}
where $\cL_1,\cL_2$ are some line bundles. Note that the category 
$$\langle\langle\cL_1\rangle,\langle\cL_2\rangle\rangle=\langle\langle\cO(-s_2-h)\rangle,\langle\cO(-s_2)\rangle\rangle$$
is equivalent to $\Db(\P^1)$, so any two SOD-s of this category are mutations of each other, and hence~\eqref{eq_EECC4} is a mutation of~\eqref{eq_EECC5}.

\textbf{Links of type III.} These  are inverse to links of type I, nothing to check anymore.

\textbf{Links of type IV.} We need to check that if $X$ has two Mori fibre space  structures over curves $C_1,C_2$ then the two standard SOD-s  $\cBstd(X/C_1)$ and $\cBstd(X/C_2)$ are mutation-equivalent. 
By Proposition~\ref{prop:SarkisovLinks}, we have $\deg X\in\{1,2,4,8\}$ and consider these cases.
    
    \textbf{Degree of $X$ is $8$}. Here $ X\cong \P^1\times \P^1$, with the two fibrations denoted by $h_1,h_2$. Then by the definition, the standard SOD-s are  
    $$\langle \langle\cO(-h_1-h_2)\rangle,\langle\cO(-h_2)\rangle,\langle\cO(-h_1)\rangle,\langle\cO\rangle\rangle$$
    and
    $$\langle \langle\cO(-h_1-h_2)\rangle,\langle\cO(-h_1)\rangle,\langle\cO(-h_2)\rangle,\langle\cO\rangle\rangle,$$
    which differ by the orthogonal mutation of the two middle terms.

    \textbf{Degree of $X$ is $4$}. The two $G$-equivariant fibrations $ X\to \P^1$ are given by divisor classes $h_1,h_2$ such that $h_1+h_2=-K_{X}$ (cf. Lemma~\ref{lemma_dualclasses}), and the two standard SOD-s for $\Db(X)$ are
    $$\langle \bA_1, \cO(-h_1),\cO\rangle\quad\text{and}\quad \langle \bA_2, \cO(-h_2),\cO\rangle.$$
    Applying mutations to the first one we get the second one:
    \begin{multline*}
    \langle \bA_1, \cO(-h_1),\cO\rangle  \overset{R_1}\leadsto\langle \cO(-h_1),\bA'_1, \cO\rangle  \overset{K}\leadsto \langle \bA'_1, \cO,\cO(-h_1-K)\rangle=
    \langle \bA'_1, \cO,\cO(h_2)\rangle\overset{L_3}\leadsto\\
    \leadsto\langle \bA'_1, \cO(-h_2),\cO\rangle=\langle \bA_2, \cO(-h_2),\cO\rangle.
    \end{multline*}

    \textbf{Degree of $X$ is $1$ or $2$}. We treat the cases when $\deg X=1,2$ uniformly and similarly to the cases of links of type II where the roof is a del Pezzo surface of degrees $1,2$. Let $\sigma$ be the Bertini, respectively Geiser involution on $X$. Then 
    (see \cite[Th. 2.6]{Isk96}) 
    the two $G$-equivariant fibrations $X\to \P^1$ are given by divisor classes $h_1,h_2$ such that $\sigma(h_1)=h_2$, and the two standard SOD-s for $\Db(X)$ are
    $$\langle \bA_1, \cO(-h_1),\cO\rangle\quad\text{and}\quad \langle \bA_2, \cO(-h_2),\cO\rangle.$$
    We use Lemma~\ref{lemma_fCY} as before. Assume $\deg X=2$, then the lemma says that $S_{\cO^{\perp}}^2\cong \sigma^*[3]$, where $S_{\cO^\perp}$ is the Serre functor on $\cO^\perp\subset \Db(X)$.
    Mutations $(L_2)^{4}$ take $\langle \bA_1, \cO(-h_1),\cO\rangle$
    to 
    $$\langle S_{\cO^\perp}^2\bA_1, S_{\cO^\perp}^2\cO(-h_1), \langle\cO\rangle\rangle=
    \langle \sigma^*\bA_1, \sigma^*\cO(-h_1), \langle\cO\rangle\rangle=
    \langle \bA_2, \cO(-h_2),\cO\rangle,$$
    see~\eqref{eq_SerreMutat1}.
    For $\deg X=1$ the argument is similar, it suffices to apply $(L_2)^6$.
    
\medskip
Proposition~\ref{prop_sarkisov} is proven.
\end{proof}

\bibliography{atoms}
\bibliographystyle{alpha}

\end{document}